%% file: newlinesApril12.tex
\documentclass{article}

\usepackage{color}
\usepackage{fullpage, amssymb, theorem, epsfig}
\usepackage[all,2cell,dvips]{xy} \UseAllTwocells \SilentMatrices

\makeatletter
\def\eqalign#1{\null\,\vcenter{\openup\jot\m@th
  \ialign{\strut\hfil$\displaystyle{##}$&$\displaystyle{{}##}$\hfil
      \crcr#1\crcr}}\,}
\makeatother

\newcommand\eindproof{\unskip\nobreak\hfill\hbox{\quad $\square$}\par \medskip}

\newtheorem{theorem}{Theorem}[subsection]
\newtheorem{proposition}[theorem]{Proposition}
\newtheorem{convention}{Convention}
\newtheorem{definition}[theorem]{Definition}

\newtheorem{lemma}[theorem]{Lemma}

\newtheorem{corollary}[theorem]{Corollary}
\newenvironment{proof}{\noindent {\bf Proof.}}{\eindproof}

{\theorembodyfont{\rmfamily}
\newtheorem{example}[theorem]{Example}
\newtheorem{remark}[theorem]{Remark}}

\newcounter{nootje}
\renewcommand\check[1]{\relax}

\renewcommand\hat{\widehat}
\newcommand\isom{\cong}
\newcommand\Sel{\mathop{{\rm Sel}}\nolimits}
\newcommand\car{\mathop{{\rm char}}\nolimits}

\newcommand\Aut{\mathop{{\rm Aut}}\nolimits}
\newcommand\Sym{\mathop{{\rm Sym}}\nolimits}
\newcommand\Hom{\mathop{{\rm Hom}}\nolimits}
\newcommand\NS{\mathop{{\rm NS}}\nolimits}
\newcommand\Gal{\mathop{{\rm Gal}}\nolimits}
\newcommand\GL{\mathop{{\rm GL}}\nolimits}
\newcommand\PGL{\mathop{{\rm PGL}}\nolimits}
\newcommand\SL{\mathop{{\rm SL}}\nolimits}
\newcommand\PSL{\mathop{{\rm PSL}}\nolimits}
\newcommand\Pic{\mathop{{\rm Pic}}\nolimits}
\newcommand\Spec{\mathop{{\rm Spec}}\nolimits}
\newcommand\disc{\mathop{{\rm disc}}\nolimits}
\newcommand\rk{\mathop{{\rm rk}}\nolimits}

\renewcommand\ell{\mathop{{\rm ell}}\nolimits}
\newcommand\SQi{\Aut \Lambda}

\renewcommand\P{\mathbb{P}}
\newcommand\Z{\mathbb{Z}}
\newcommand\Q{\mathbb{Q}}

\renewcommand\overline{\bar}
\renewcommand\O{\mathcal{O}}

\newcommand\kbar{\overline{k}}
\newcommand\Abar{\overline{A}}
\newcommand\Vbar{\overline{V}}
\newcommand\ra{\rightarrow}

\newcommand\semi{\rtimes}

\newcommand\cV{\mathcal{V}}
\newcommand\cS{\mathcal{S}}
\newcommand\cN{{\Psi}}
\newcommand\cL{\mathcal{L}}
\newcommand\cP{\mathcal{P}}

\newcommand\gonquad{exhibit}          
\newcommand\dgonquad{gallery}   
\newcommand\gonquads{exhibits}
\newcommand\dgonquads{galleries}

\newcommand\tmu{\tilde{\mu}}
\newcommand\tT{\tilde{T}}
\newcommand\sqdelta{\xi}
\newcommand\sqd{\epsilon}
\newcommand\Avar{z}
\newcommand\muvar{\zeta}
\newcommand\tmuvar{\tilde{\zeta}}

\DeclareFontEncoding{OT2}{}{} 
  \newcommand{\textcyr}[1]{%
    {\fontencoding{OT2}\fontfamily{wncyr}\fontseries{m}\fontshape{n}%
     \selectfont #1}}
\newcommand{\Sha}{{\mbox{\textcyr{Sh}}}}

\begin{document}

\begin{center}
{\bf Nontrivial elements of Sha explained through K3 surfaces}\\
Adam Logan and Ronald van Luijk \\

\vspace{1cm}

\begin{minipage}{4in}
\small In this paper we present a new method to show that a
principal homogeneous space of the Jacobian of a curve of genus two
is nontrivial. The idea is to exhibit a Brauer-Manin obstruction
to the existence of rational points on a quotient of this
principal homogeneous space. In an explicit example we apply the
method to show that a specific curve has infinitely many quadratic
twists whose Jacobians have nontrivial Tate-Shafarevich group.
\end{minipage}
\end{center}

\vspace{1cm}


\smallskip

\section{Introduction}

By Faltings' Theorem every curve of genus $2$ defined over a
number field $k$ has only finitely many rational points. Several
methods have been developed to find all rational points of a given
curve $C$, such as the method of Chabauty-Coleman, the
Mordell-Weil sieve, and combinations of these with covering
techniques. All these methods require that we know the finitely
generated abelian group $J(k)$ of rational points on the Jacobian
$J$ of $C$, at least up to a subgroup of finite index. The torsion
subgroup of $J(k)$ is generally easy to find and the problem is
therefore to find the rank $r$ of $J(k)$. The rank can be read off
from the size of the group $J(k)/2J(k)$ once the torsion subgroup
is known. This group fits into an exact sequence $$ 0 \ra
J(k)/2J(k) \ra \Sel^{(2)}(J/k) \ra \Sha(J/k)[2] \ra 0, $$ where
$\Sel^{(2)}(J/k)$ is the $2$-Selmer group of $J/k$ and
$\Sha(J/k)[2]$ is the $2$-torsion subgroup of the Tate-Shafarevich
group $\Sha(J/k)$ of $J/k$. The $2$-Selmer group is computable
(see \cite{S}). It is, however, not even known whether the
Tate-Shafarevich group is always finite. Many papers have been
devoted to exhibiting nontrivial elements of $\Sha(J/k)$. In this
paper we will follow a new method, suggested by Michael Stoll,
which leads to the following result, our main theorem.

\begin{theorem}\label{main}
Let $S$ be the union of $\{5\}$
with the set of primes that split completely in the field of definition
of the lines of $V$, which is
$$\Q\left(\sqrt {-1}, \sqrt 2, \sqrt 5, \sqrt {-3(1 + \sqrt 2)}, \sqrt {6(1 + \sqrt 5)}\right).$$
Then for all $n$ which are products of elements of $S$, the $2$-part of the
Tate-Shafarevich
group of the Jacobian of the curve defined by
$$y^2 = -6n(x^2+1)(x^2-2x-1)(x^2+x-1)$$
is nontrivial.
\end{theorem}

Our method uses the fact that every element of $\Sel^{(2)}(J/k)$
can be represented by an everywhere locally solvable $2$-covering
of $J$. A $2$-covering of $J$ is a surface $X$ together with a
morphism $\pi \colon X \ra J$, defined over $k$, such that over the
algebraic closure $\kbar$ there exists an isomorphism $X_{\kbar} \isom
J_{\kbar}$ making the following diagram commutative.
$$
\xymatrix{
X_{\kbar} \ar[r]^{\isom}\ar[dr]_{\pi} & J_{\kbar} \ar[d]^{[2]} \\
&J_{\kbar} \\
}
$$
The image of $X$ in $\Sha(J/k)$ is trivial if and only if $X$ has a
rational point. Unfortunately, the easiest way to describe $X$ in
general is as the intersection of
$72$ quadrics in $\P^{15}$ (see \cite{CF}, section 2.3
for the statement). The isomorphism
in the diagram above is determined up to translation of $J$ by a
$2$-torsion point. Since multiplication by $-1$ commutes with these
translations, it induces a unique involution $\iota$ on $X$. Our
strategy to prove that $X$ has no rational points is to show there is a
Brauer-Manin obstruction to the existence of rational points on
$X/\iota$, or rather on a minimal nonsingular model $V$ of this
quotient variety. Note that $V$ is a twist of the Kummer surface
associated to $J$.

The goal of this paper is twofold.  In addition to proving the main theorem,
we will analyze the geometry of $V$. By \cite{CF}, chapter 16, the
surface $V$ can be embedded in $\P^5$ as the complete intersection
of three quadrics. In this reference this is only done when $V$ is a
trivial twist, but we will see that it holds for any twist.
In this embedding, $V$ contains $32$ lines that
generically generate the N\'eron-Severi group of $V$. We will also
investigate the intersection pairing among these lines and exhibit
$15$ pairs of elliptic fibrations, each associated to one of the
nontrivial $2$-torsion points of $J$.  In section \ref{geometry} we
will find the fields of
definition of these lines and elliptic fibrations
together with explicit equations for them.  Then in section
\ref{sectionsha}, we will use the information we have acquired in section
\ref{geometry} to exhibit an explicit example for which we are able to
show a Brauer-Manin obstruction.
This will be the most important part
of the proof of the main theorem.

We thank Michael Stoll for suggesting this project to us, Nils Bruin
and Victor Flynn for helpful suggestions and explanations, William
Stein for letting us use his computers, CRM in Montreal and MSRI
for their hospitality and financial support, and the {\sc
magma} group for developing their software. Also, the first author
thanks the Nuffield Foundation for funding his research with a
grant in their Awards to Newly Appointed Lecturers program, which
supported him at CRM in the fall of 2005, and the University of
Waterloo.  The second author also
thanks PIMS, Simon Fraser University, the University of British Columbia,
and Universidad de los Andes.

\section{The geometry of the surface}\label{geometry}

In this section we will investigate the geometry of the K3 surfaces
that arise as the quotients of the principal homogeneous spaces under
the Jacobian as described in the introduction. These K3 surfaces are
twists of the Kummer surface associated to the Jacobian.
In \cite{CF}, sect.\ 16.2, it is remarked that the Kummer surface
itself, i.e., the trivial twist, contains $32$ lines. We will give a
direct proof that all twists contain $32$ lines. We will analyze the
Galois action on the set of these lines. We will also describe certain
elliptic fibrations, coming in pairs associated to pairs of roots of
$f$. In the next section these will be used to find Brauer-Manin
obstructions to the existence of rational points on some of these
surfaces. In this section we will rarely use the fact that these K3
surfaces are twists of the Kummer surface. Given that our goal is
to actually implement an algorithm, we will keep everything very
explicit, including our proofs. We will, however, refer to \cite{CF}
at times in order not to lose the context our work should be seen in.

\subsection{The surface}

%
%

Let $k$ be a field and $W$ a vector space over $k$ of dimension $r\geq 1$.
We let $\P(W)$ denote the projective space $(W - \{0\})/k^*$ associated
to $W$. The homogeneous coordinate ring of $\P(W)$ is the symmetric
algebra $S(\hat{W}) = \bigoplus_{n \geq 0} S^n(\hat{W})$, where
$\hat{W}=\Hom_k(W,k)$ is the dual of $W$. Let $(x_1, \ldots, x_r)$ be
a basis of $\hat{W}$. This basis yields an isomorphism
$\P(W) \ra \P_k^{r-1}$ that sends the element $w\in W$ to
$[x_1(w):\ldots : x_r(w)]$. Thus the $x_i$ determine a coordinate
system on $\P(W)$. The symmetric algebra $S(\hat{W})$ is isomorphic to the
polynomial ring $k[x_1,\ldots, x_r]$.

Let $f \in k[X]$ be a separable polynomial of degree $6$, and set
$A_f=k[X]/f$. Consider $\delta \in A_f^*$ and set
$$
\cV_{f,\delta}=
\left\{q \in A_f \,\,:\,\, \exists c_0,c_1,c_2
\mbox{ such that } \delta q^2 = c_2X^2+c_1X+c_0 \right\}.
$$
Let $V_{f,\delta}$ denote the subset of $\P(A_f)$ corresponding to
$\cV_{f,\delta}$. For any $c \in k^*$ we obviously have $A_{cf} =
A_f$,  $\cV_{cf,\delta} = \cV_{f,\delta}$, and $V_{cf,\delta} = V_{f,\delta}$.
We will often leave any subscript out of the notation that is clear
from the context. Let $(a_0,\ldots,a_5)$
be the canonical basis of $\hat{A}$ associated to the basis $(1,X,\ldots, X^5)$
of $A$, so that any $q \in A$ can be written as $q = \sum_{i=0}^5
a_i(q) X^i$. As above the $a_i$ determine a coordinate system on $\P(A)$.
Writing out $\delta q^2$, we see that there are quadratic forms
$C_0,\ldots, C_5$ in the homogeneous coordinate ring
$S(\hat{A})$ of $\P(A)$, depending on $f$ and $\delta$, such that
$a_i(\delta q^2) = C_i(q)$ for any $q \in A$.
We have $q \in \cV$ if and only if we have $C_i(q) = 0$ for
$3 \leq i \leq 5$. This implies that $V$ is an algebraic set in
$\P(A)$, defined
over $k$ by the three quadrics $C_3,C_4$, and $C_5$.
We will express the $C_i$ in a new coordinate
system, inspired by \cite{CF}, Chapter 16.

For any field extension $k'$ of $k$ we write $A_{k'}=A \otimes_k
k'$, viewed as a vector space over $k'$, so that we have
$\P(A_{k'}) \isom \P(A)_{k'}$. We write $\bar A$ and $\bar V$ for
$A_{\kbar}$ and $V_{\kbar}$ respectively, where $\kbar$ is a fixed
algebraic closure $k$. Let $\Omega$ denote the set of roots of $f$
in $\kbar$. Then $l=k(\Omega)$ is the splitting field of $f$. For
$\omega \in \Omega$ we let $\varphi_\omega$ denote the $l$-algebra
homomorphism $A_{l}\ra l$ given by $X \mapsto \omega$. The
$\varphi_\omega$ form a basis of $\hat{A}_{l}$ and therefore
induce a coordinate system on $\P(A_{l})$.

\begin{remark}\label{legendre}
Let $(P_\omega)_\omega$ be the canonical basis of $A_l$ associated to
the basis $(\varphi_\omega)_\omega$ of $\hat{A}_l$.
Then for each $q \in A_l$ we have
$$
C_i(q) = a_i(\delta q^2)
  = a_i \left(\sum_\omega \varphi_\omega(\delta q^2) P_\omega\right)
  = \sum_\omega a_i(P_\omega) \varphi_\omega(\delta)
     \varphi_\omega (q)^2,
$$
which implies $C_i = \sum_\omega a_i(P_\omega) \delta_\omega
\varphi_\omega^2$, with $\delta_\omega  = \varphi_\omega(\delta)$.
Note that we have $\varphi_\omega = \sum_{i=0}^5 \omega^i a_i$, so we
can also write the $C_i$ in terms of the coordinates $a_i$.
We can make the constants $a_i(P_\omega)$ explicit by setting
$P'_\omega = \prod_{\theta \in \Omega\setminus \{\omega\}}(X-\theta)$.
Then $P_\omega$ equals the Legendre polynomial $P'_\omega(\omega)^{-1}
P'_\omega$.
\end{remark}

For all $\omega \in \Omega$ we set $\lambda_\omega =
\varphi_\omega(P'_\omega) = P'_\omega(\omega)$ with $P'$ as in Remark
\ref{legendre}. For $j=0,1,2$, set
$$
Q_j = \sum_\omega \omega^j\lambda_\omega^{-1} \delta_\omega \varphi_\omega^2.
$$

\begin{convention}
\bf From now on we will assume that the characteristic of $k$ is
different from $2$.
\end{convention}

\begin{proposition}\label{threequad}
The algebraic set $V_{f,\delta}$ is a smooth, geometrically integral
complete intersection of the three quadrics $Q_0$,
$Q_1$, and $Q_2$. It is a K3 surface of degree $8$.
\end{proposition}
\begin{proof}
Suppose $f=\sum_{i=0}^6 f_iX^i$. The set $V$ is defined by the
quadrics $C_3,C_4,C_5$, so it is also defined by
$$
Q_0' = C_5, \qquad Q_1' = C_4-f_5f_6^{-1}C_5, \qquad \mbox{and} \qquad
Q_2' = C_3 - f_5f_6^{-1}C_4 +(f_5^2f_6^{-2}-f_4f_6^{-1})C_5.
$$
From the equations $-f_5f_6^{-1} = \sum_\omega \omega$ and $f_4f_6^{-1} =
\sum_{\psi \neq \omega} \psi\omega$ we find $Q_i' = Q_i$ for $i=0,1,2$.
One checks that the
quadrics define a smooth complete intersection. Every smooth complete
intersection of three quadrics in $\P^5$ is a K3 surface of degree $8$.
\end{proof}

\begin{remark}\label{Vistwist}
The statement that $V$ is a K3 surface also
follows from the fact that $V$ is the twist by $\delta$ of the
desingularized Kummer surface associated to the Jacobian of the curve
given by $y^2 = f$; see \cite{CF}, Chapter 16.
\end{remark}

\begin{corollary}\label{VisK3}
The N\'eron-Severi group $\NS(V)$ of $V$ is free, finitely generated,
isomorphic
to $\Pic V$, and it has a lattice structure induced by the
intersection pairing.
\end{corollary}
\begin{proof}
There are injections
$\Pic V \hookrightarrow \Pic \bar V$ and
$\Pic^0 V \hookrightarrow \Pic^0 \bar V$.
As $V$ is a complete intersection by Proposition \ref{threequad},
we find from \cite{sga7ii}, Thm. 1.8, that $\Pic^0 \bar V=0$,
so $\NS(V) = \Pic V$. The N\'eron-Severi group of any projective
variety is finitely generated, see \cite{hag}, exc. V.1.7.
Also by \cite{sga7ii}, Thm. 1.8, we find that $\Pic \bar V$ is
torsion-free, so it is free. In general the intersection pairing
induces a lattice structure on the N\'eron-Severi group modulo
torsion (see \cite{hagtwo}), which in this case is isomorphic to
$\Pic V$.

This theorem also follows from the fact that $V$ is a K3 surface,
as shown for characteristic $0$ in
Prop.\ VIII.3.2 and on page 120 of \cite{BPV}, and for positive
characteristic in Theorem 5 of \cite{BM}.
\end{proof}

%

\begin{remark}\label{coneovercone}
Consider the net of quadrics $pQ_0+qQ_1+rQ_2$ vanishing on $V$.
The curve $C$ in $\P^2(p,q,r)$ corresponding to singular quadrics is given
by the equation $\det(pM_0+qM_1+rM_2)=0$ of degree $6$, where $M_i$ is
the symmetric matrix corresponding to the quadratic form $Q_i$. For
any $\omega \in \Omega$ the quadric hypersurface corresponding to any
point on the line $p+\omega q +\omega^2 r =0$ is singular at the
point in $\P(A_l)$ given by $\varphi_\theta=0$ for all $\theta\neq
\omega$. This
implies that $C$ consists of $6$ lines. The $15$ intersection points
are parametrized by pairs $(\omega,\psi) \in \Omega^2$ with $\omega
\neq \psi$. The corresponding quadrics are given by $Q_{\omega \psi} =
\omega \psi Q_0-(\omega+\psi)Q_1 +Q_2$. The hypersurface given by
$Q_{\omega\psi}$ is singular at every point on the line
$m_{\omega\psi}$ given by
$\varphi_\theta =0$ for $\theta \neq \omega,\psi$. This hypersurface
is a cone over a cone over a quadric $D_{\omega\psi}$ in the $\P^3$
obtained by projecting $\P(A_l)$ away from $m_{\omega\psi}$, and
therefore contains two families of
linear three-spaces. Each family cuts out a family of curves on $V$,
given by the two quadrics $Q_0,Q_1$ in these three-spaces. This yields
two elliptic
fibrations of $V$, both defined over a quadratic extension of
$k(\omega\psi, \omega+\psi)$. We will see later which extension this
is. Note that the projection from $m_{\omega\psi}$ induces a
$4$--to--$1$ map from $V$ to $D_{\omega\psi}$. The elliptic fibrations
factor through this map. Since $D_{\omega\psi}$
satisfies the Hasse principle this map may be used to obtain
information about the arithmetic of $V$.
\end{remark}

Let $k'$ be any field extension of $k$.
For every $\Avar \in A_{k'}^*$, let $[\Avar]$ denote the automorphism of
$\P(A_{k'})$ induced by multiplication by $\Avar$.
Note that $[\Avar]$
maps $V_\delta$ isomorphically to $V_{\delta \Avar^{-2}}$, so if $\delta$ is
a square in $A_{k'}^*$, then $V_{\delta}$ is isomorphic to $V_1$ over $k'$.
For any commutative
ring $R$ let $\mu(R)$ denote the kernel of the endomorphism $x \mapsto
x^2$ of $R^*$.
The scheme $\Spec A[t]/(t^2-1)$ represents the functor from the category
of $A$-algebras to the category of groups that sends $R$ to $\mu(R)$ in the
sense that the elements of $\mu(R)$ are parametrized by the maps from
$\Spec R$ to $\Spec A[t]/(t^2-1)$ that respect the map to $\Spec A$.
Such a map corresponds to the image of $t$ under the associated
homomorphism $A[t]/(t^2-1) \rightarrow R$.
Let $\mu_A$ be the Weil restriction of this scheme associated to the
extension $A/k$. Then $\mu_A$ is a $k$-scheme representing the functor
that sends a field extension $l$ of $k$ to $\mu(A_l)$. Let $\tmu$ be the quotient
of $\mu_A$ by the automorphism that is induced by $t \mapsto -t$ on
$\Spec A[t]/(t^2-1)$. Then for all field extensions $l$ of $k$ we have
$\tmu(l) = (\mu(A_{\overline{l}})/\langle-1\rangle)^{G_l}$, where $G_l$ is the absolute Galois group of $l$.

\begin{lemma}
The homomorphism $A_{k'}^* \ra \Aut_{k'} \P(A_{k'})$ that sends $\Avar$
to $[\Avar]$ has kernel $k'^*$. It induces
an injective homomorphism $\tmu({k'}) \ra \Aut_{k'} V_{k'}$.
\end{lemma}
\begin{proof}
Note that for $z \in \mu(A_{k})'$ the automorphisms $[z]$ and
$[-z]$ are equal, so the homomorphism $\tmu({k'}) \ra \Aut_{\kbar}
V_{\kbar}$ is well defined and has image in $\Aut_{k'} V_{k'}$. We
may therefore assume that $k'$ is algebraically closed, so that
$\tmu(k') = \mu(A_{k'})/\langle -1\rangle$. Let $\rho$ denote
the homomorphism
$\Avar \mapsto [\Avar]$ in question. If $\rho(\Avar)$ is the
identity, then we have $\Avar\cdot 1 = 1$ in
$(A_{k'}-\{0\})/k'^*$, which implies $\Avar \in k'^*$. Set $H_V =
\{ \tau \in \Aut_{k'} \P(A_{k'}) \,\, : \,\, \tau(V) = V\}$. Since
$[\Avar]$ maps $V_{\delta}$ to $V_{\delta \Avar^{-2}}$, the
restriction $\rho_\mu$ of $\rho$ to $\mu(A_{k'})$ factors through
$H_V$. Because $V$ is not contained in a hyperplane, the map
$H^0(\P^5,\O_{\P^5}(1)) \ra H^0(V,\O_V(1))$ is injective. As every
element in $\Aut_{k'} \P^5_{k'}$ is determined by its action on
$H^0(\P^5,\O_{\P^5}(1))$, this implies that the restriction map
$r\colon \,H_V \ra \Aut_{k'} V_{k'}$ is injective. Hence, the
composition $r\circ{\rho_\mu}\colon \, \mu(A_{k'}) \ra \Aut_{k'}
V_{k'}$ has kernel $\ker \rho_\mu = \mu(A_{k'}) \cap k'^* = \{\pm
1\}$ and therefore induces the injective homomorphism already
mentioned.
\end{proof}

\noindent For any $\muvar \in \mu(A_{k'})$ we write $\tmuvar$ for the
image of $\muvar$ in $\tmu({k'})$, and $[\muvar]$ or $[\tmuvar]$ for
the induced action by multiplication on $\P(A_{k'})$ and $V_{k'}$.
Let $T$ be the Weil restriction of the scheme $\Spec A[t]/(t^2-\delta)$
associated to the extension $A/k$ and let $\tT$ be the $k$-scheme that
is the quotient of $T$ by the automorphism induced by $t \mapsto -t$
on $\Spec A[t]/(t^2-\delta)$. As for $\mu_A$ and $\tmu$ above, we
can make the identifications
$$
T(l) = \{ \sqdelta \in A_{l} \,\, : \,\, \sqdelta^2 = \delta\},
\qquad \mbox{and} \qquad
\tT(l) = \left(\{ \sqdelta \in A_{\overline{l}} \,\, : \,\,
   \sqdelta^2 = \delta\}/\langle [-1] \rangle \right)^{G_l}.
$$
Clearly $T$ is a $k$-torsor under $\mu_A$, with the transitive free action
of $\mu_A(k')=\mu(A_{k'})$ on $T(k')$ given by multiplication in $A_{k'}$.
Similarly, $\tT$ is a $k$-torsor under $\tmu$.

For any $\sqdelta\in T(k')$ the $2$-dimensional subspace
$$
\cL_{\sqdelta} = \{\sqdelta^{-1} (sX+t) \,\, : \,\, s,t \in k'\}
$$
of $A_{k'}$ corresponds to a line in  $\P(A_{k'})$,
defined over $k'$, which is contained in $V_{k'}$ and which we will
denote by $L_\sqdelta$.
Since $L_{\sqdelta} = L_{-\sqdelta}$, this implies that to each
$\tilde{\sqdelta} \in \tT(k')$ we can associate a unique line
$L_{\tilde{\sqdelta}}$, namely $L_{\tilde{\sqdelta}}=L_{\sqdelta}$,
where $\sqdelta \in T(\overline{k}')$ is a lift of $\tilde{\sqdelta}$.
Let $\Lambda(k')$ denote the set
of all lines $L_{\tilde{\sqdelta}}$ corresponding to some
$\tilde{\sqdelta} \in \tT(k')$.
Note that for any $\Avar \in \mu(A_{\overline{k'}})$ the automorphism $[\Avar]$ maps $L_{\sqdelta}$ to $L_{\sqdelta \Avar^{-1}}$.
This induces an action of $\tmu({k'})$ on $\Lambda(k')$.

\begin{lemma}\label{transfree}
The action of $\tmu({k'})$ on $\Lambda(k')$ is transitive and free.
\end{lemma}
\begin{proof}
Transitivity follows from the fact that the action of $\mu(A_{\overline{k'}})$ on $T(\overline{k'})$ is transitive
and the map $\tT(k')\ra \Lambda(k')$ sending $\tilde{\sqdelta}$ to $L_{\tilde{\sqdelta}}$ is surjective and respects the action of
$\mu(A_{k'})$. To show that the action is free, we may assume that
$k'$ is algebraically closed.
Suppose that for $\tmuvar \in \tmu({k'})$ we have
$\tmuvar L_{\sqdelta} = L_{\sqdelta}$ and let $\muvar \in \mu(A_{k'})$ be
a lift of $\tmuvar$. Then multiplication by $\muvar$ sends the
subspace $\cL_{\sqdelta}$ to itself, so it also sends the subspace
$W=\{sX+t \,\, :\,\, s,t \in k'\}$ to itself. In particular this
implies that
there are $s,t\in k'$ such that $\muvar\cdot 1  = sX+t$. From the fact
that $\muvar \cdot X
\in W$ we then find $s=0$, so $\muvar=t$ is in $\mu(A_{k'}) \cap k' =
\{\pm 1\}$, which means $\tmuvar=1$. We conclude that the action of
$\tmu({k'})$ on $\Lambda(k')$ is free.
\end{proof}


\begin{lemma}\label{paramlines}
The map $T(\overline{k'}) \ra \Lambda(\overline{k'})$ that sends
$\sqdelta$ to $L_\sqdelta$
induces a bijection $\tT(k') \ra \Lambda(k')$ that respects the action of
$\tmu({k'})$ and $G(\kbar/k)$.
\end{lemma}
\begin{proof}
It is obvious that we obtain a surjective map $\tT(k') \ra \Lambda(k')$ that
respects the action of $\tmu({k'})$ and $G(\kbar/k)$.
Injectivity then follows from the fact that $\tmu({k'})$ acts
transitively and freely on both $\tT(k')$ and $\Lambda(k')$, as we saw in
Lemma
\ref{transfree}.
\end{proof}

\begin{remark}\label{torsorslines}
Lemmas \ref{transfree} and \ref{paramlines} combined say that
$\tT(\kbar)$ and $\Lambda(\kbar)$ are isomorphic over $k$
as $k$-torsors under $\tmu({\kbar})$.
\end{remark}

\begin{convention}
\bf For the rest of this section we will suppose that $l$ is contained
in $k'$.
\end{convention}

By the Chinese Remainder Theorem, the map
$$
\varphi_{k'}\colon \, A_{k'} \ra \bigoplus_{\omega\in \Omega} k'
$$
induced by the $\varphi_\omega$ is an
isomorphism, defined over $l$. Note that the induced Galois action on
$\bigoplus_\omega k'$ is given by acting on both the indices and
the coefficients in $k'$. In other words, for $\sigma \in G(\kbar/k)$
we have
$$
{}^\sigma\!\big((c_\omega)_{\omega\in\Omega}\big) =
\left({}^\sigma\!c_{{}^{\sigma^{-1}}\!\omega}\right)_{\omega\in\Omega}.
$$
It follows that $\mu(A_{k'})$ and $\tmu({k'})$ are isomorphic to
$\bigoplus_{\omega} \{\pm 1\}$ and $(\bigoplus_{\omega} \{\pm
1\})/\{\pm 1\}$ respectively.

\begin{lemma}\label{zott}
Either the set $\Lambda(k')$ is empty, or it contains exactly $32$ lines.
\end{lemma}
\begin{proof}
Since $\tmu({k'})$ has exactly $32$ elements, this follows from
Lemma \ref{transfree}.
\end{proof}

For $I \subset \Omega$ let $\muvar_I$ denote the unique
element in $\mu(A_{k'})$ with $\varphi_\omega(\muvar_I) = -1$ for
$\omega \in I$ and $\varphi_\omega(\muvar_I) = 1$ for $\omega \not \in I$.
Note that we have $\tmuvar_I =
\tmuvar_{\Omega\setminus I}$. We will also denote this element by
$\tmuvar_\pi$, where $\pi$ is the partition $\{I,\Omega\setminus I\}$
of $\Omega$. The map from the set $\cP(\Omega)$ of all subsets of $\Omega$ to
$\mu(A_{k'})$ that sends $I$ to $\muvar_I$ is a bijection.
It induces a bijection from the set $\Pi$ of partitions of
$\Omega$ into two sets to $\tmu({k'})$, sending $\pi$ to
$\tmuvar_\pi$. The inverse $\pi\colon \,\tmu({k'})\ra \Pi$ of the
latter bijection is given by
$$
\pi (\tmuvar) = \left\{\{\omega \,\,:\,\, \varphi_\omega(\muvar) =
1\}, \{\omega \,\,:\,\, \varphi_\omega(\muvar) = -1\}\right\},
$$
where $\muvar \in \mu(A_{k'})$ is a lift of $\tmuvar$.
For any $\pi =\{I,J\}\in \Pi$ the automorphism
$[\tmuvar_\pi] = [\muvar_I]$ is defined
over the fixed field of the group $\{g \in G(\kbar/k)\,\,:\,\,
{}^g\pi = \pi\}$. Note that $[\muvar_I]$ acts on $\P(A_l)$ by sending
the coordinate $\varphi_\omega$ to $\pm \varphi_\omega$, where the
sign is negative if and only if we have $\omega \in I$.

For any $\omega \in \Omega$ and $\sqdelta \in T(k')$ we let
$P_{\sqdelta,\omega}$
denote the point on the line $L_{\sqdelta}$ corresponding to the set
$$
    \{\sqdelta^{-1} s(X-\omega) \,\, : \,\, s \in k'\}\subset A_{k'}.
$$
For any $\Avar \in \mu(A_{k'})$ the map $[\Avar]$ sends
$P_{\sqdelta,\omega}$ to $P_{\sqdelta \Avar^{-1},\omega}$.
The notation distinguishes the points
$P_{\sqdelta,\omega}$, indexed by two subscripts,
from the polynomials $P_\omega$ from Remark \ref{legendre}, which are
indexed by only one.

\begin{proposition}\label{intersectionpoints}
For all $\omega \in \Omega$ and all $\sqdelta \in T(k')$ we have
$\muvar_\omega P_{\sqdelta,\omega} = P_{\sqdelta,\omega}$.
Two lines $L,L'\in \Lambda(k')$ intersect if and only if there exists an
$\omega \in \Omega$ such that $\muvar_\omega L = L'$, in which
case the intersection point is $P_{\sqdelta,\omega} =
P_{\sqdelta',\omega}$, where $\sqdelta,\sqdelta'\in T(k')$ are such
that $L = L_{\sqdelta}$ and $L' = L_{\sqdelta'}$.
\end{proposition}
\begin{proof}
One easily checks $\varphi((\muvar_\omega-1)(X-\omega)) =0$, so we
have $X-\omega = \muvar_\omega(X-\omega)$ for all $\omega \in
\Omega$. This implies that for all $\sqdelta \in T(k')$ we have
$\muvar_\omega P_{\sqdelta,\omega} =P_{\sqdelta\muvar_\omega^{-1},\omega}
 =P_{\sqdelta,\omega}$,
which proves the first statement. Let $\sqdelta,\sqdelta'\in T(k')$ be such
that $L = L_{\sqdelta}$ and $L' = L_{\sqdelta'}$.
Suppose there is an $\omega \in \Omega$ such that $\muvar_\omega L = L'$.
Then $L$ and $L'$ both go through the point
$P_{\sqdelta,\omega}=\muvar_\omega P_{\sqdelta,\omega}$, so they intersect.

Conversely, suppose $L$ and $L'$ intersect. Then the subspaces
$\cL_\sqdelta$ and $\cL_{\sqdelta'}$ have a nonzero
intersection, so we can choose $s,t,s',t' \in k'$ such that
$\sqdelta^{-1}(sX+t) = \sqdelta'^{-1}(s'X+t') \neq 0$. Applying
$\varphi_\omega$ we find $\varphi_\omega(\muvar)(s\omega+t)=s'\omega+t'$
for all $\omega \in \Omega$, with
$\muvar = \sqdelta^{-1}\sqdelta' \in \mu(A_{k'})$.
After replacing $\sqdelta'$ by $-\sqdelta'$ if necessary,
we may assume that there are at least three $\omega\in \Omega$ with
$\varphi_\omega(\muvar) =1$. Then the equation $sx+t = s'x+t'$ has at
least three solutions in $x$, which implies $s'=s$ and $t'=t$.
From $L \neq L'$ we deduce $\muvar\neq 1$,
so there is an $\omega$ with $\varphi_\omega(\muvar) \neq 1$. The
equation $\varphi_\omega(\muvar)(s\omega+t)=s'\omega+t'=s\omega +t$ then
yields $s\omega+t =0$. Since the equation $sx+t=0$ has at most one
solution in $x$, this shows $\muvar =\muvar_\omega$ and $\muvar_\omega
L = L'$.
\end{proof}

\begin{corollary}\label{whichinter}
For any line $L\in \Lambda(k')$ and any elements $\tmuvar,\tmuvar'\in
\tmu({k'})$ the lines $\tmuvar L$ and $\tmuvar' L$ intersect if and only
if we have $\tmuvar \cdot \tmuvar' =
\tmuvar_\omega$ for some $\omega \in \Omega$.
\end{corollary}
\begin{proof}
By Lemma \ref{transfree} the element
$\tmuvar'' = \tmuvar \cdot \tmuvar' = \tmuvar^{-1} \cdot
\tmuvar'$ is the unique element in $\tmu({k'})$ for which we have
$\tmuvar'' \cdot \tmuvar L = \tmuvar' L$. By Proposition
\ref{intersectionpoints} the lines intersect if and only if we have
$\tmuvar'' = \tmuvar_\omega$ for some $\omega \in \Omega$.
\end{proof}

Put $\Abar = A_{\kbar}$, $\Vbar = V_{\kbar}$, and $\Lambda = \Lambda(\kbar)$.
For any $L,L' \in \Lambda$ we say that $L$ and $L'$ have the same or opposite
parity if for the unique $\tmuvar \in \tmu(\kbar)$ with $\tmuvar L =
L'$, the number of elements of the sets in the partition
$\pi(\tmuvar)$ is even or odd respectively. For any $\omega,\psi \in
\Omega$, let $\Phi_{\omega\psi}$ denote the subgroup of $\tmu(\kbar)$
generated by $\tmuvar_\omega$ and $\tmuvar_\psi$.

\begin{lemma}\label{fourgons}
Let $L,L' \in \Lambda$ be different lines of the same parity. Then $L$ and
$L'$ do not intersect and there are exactly two lines $M$ and $M'$ of
the opposite parity that intersect both $L$ and $L'$. There are
$\omega,\psi \in \Omega$ such that the set $\{L,L',M,M'\}$ is an orbit
of $\Lambda$ under the action of $\Phi_{\omega\psi}$.
\end{lemma}
\begin{proof}
Let $I \subset \Omega$ be such that $\muvar_I L  = L'$. Then $\# I$ is
even, so $L$ and $L'$ do not intersect by Corollary
\ref{whichinter}. After replacing $I$ by $\Omega\setminus I$
if necessary, there are $\omega,\psi \in \Omega$, such that $I =
\{\omega, \psi\}$. From Corollary \ref{whichinter} we deduce that the only
lines that intersect both $L$ and $L'$ are $M=\muvar_\omega L$ and
$M'=\muvar_\psi L$. Indeed the set $\{L,L',M,M'\}$ is an orbit
under $\Phi_{\omega\psi}$.
\end{proof}

\begin{remark}\label{nodestropes}
Remembering that $\Vbar$ is a twist of the Kummer surface associated
to the Jacobian $J$ of the curve given by $y^2 = f(x)$, we note that
the lines of one parity correspond to the $16$ blow-ups of the nodes
on the singular surface $J/\langle -1 \rangle$. The lines of the other
parity correspond to the tropes, see \cite{CF}, Sect. 3.7.
The intersection numbers among these lines are well known.
\end{remark}

\begin{definition}
A {\em $4$-gon} is a set $\{L,L',M,M'\}$ of four lines, such that
$L$ and $L'$ intersect both $M$ and $M'$.
\end{definition}

By Lemma \ref{fourgons} any two lines of the same parity determine a
unique $4$-gon. All $4$-gons arise in this way, because if the lines
$L$ and $L'$ both intersect a line $M$, then by Lemma \ref{fourgons}
both $L$ and $L'$ are of the opposite parity than $M$, so $L$ and $L'$
have the same parity.

\begin{lemma}\label{gonint}
Let $\omega,\psi\in \Omega$ and any $I,J \subset \Omega$ be such that
the lines $\muvar_I L$ and $\muvar_J L$ are not in the same orbit
under $\Phi_{\omega\psi}$. Then
the cardinalities of the sets $I \cap (\Omega\setminus \{\omega,\psi\})$ and
$J \cap (\Omega\setminus \{\omega,\psi\})$ have different parity
if and only if the line $\muvar_I L$ intersects
some line in the orbit under $\Phi_{\omega\psi}$ of the line $\muvar_J
L$ in which case it intersects exactly one line in this orbit.
\end{lemma}
\begin{proof}
Suppose that the cardinalities of the sets
$I \cap (\Omega\setminus \{\omega,\psi\})$ and
$J \cap (\Omega\setminus \{\omega,\psi\})$
have the same parity and that they are not equal.
Let $\pi=\{\pi_1,\pi_2\}\in \Pi$ be such that $\tmuvar_\pi=
\tmuvar_I\tmuvar_J$. Then $\muvar_I L$ and $\muvar_J L$ are in the
same orbit under $\Phi_{\omega\psi}$ if and only if we have $\pi_i
\subset \{\omega,\psi\}$ for $i=1$ or $i=2$. We conclude  $\pi_i
\not \subset \{\omega,\psi\}$ for $i=1,2$. Suppose that $I \cap
(\Omega\setminus \{\omega,\psi\})$ and $J \cap (\Omega\setminus
\{\omega,\psi\})$ have the same parity. Then $\pi_i \cap (\Omega
\setminus \{\omega,\psi\})\neq \emptyset$ has even parity for $i=1$
and $i=2$. It follows that for each $\tmuvar \in  \Phi_{\omega\psi}$ the
sets in the partition $\pi(\muvar\muvar_I\muvar_J) =
\pi(\muvar\muvar_\pi)$ contain exactly $2$ elements of
$\Omega\setminus \{\omega,\psi\}$, so $\tmuvar_I L$ does not intersect
any of the lines $\tmuvar\tmuvar_J L$ in the orbit of $\tmuvar_J L$ by
Corollary \ref{whichinter}.

Conversely, suppose that $I \cap (\Omega\setminus \{\omega,\psi\})$ and
$J \cap (\Omega\setminus \{\omega,\psi\})$ have different parity. Then
we have $\pi = \{K \cup \{\theta\}, \Omega \setminus (\{\theta\} \cup K
)\}$ for some $\theta \in \Omega \setminus \{\omega,\psi\}$ and
$K \subset \{\omega,\psi\}$. Then by Corollary \ref{whichinter} the
line $\muvar_I L$ intersects exactly one line in the orbit of
$\tmuvar_J L$, namely $\tmuvar_K \tmuvar_J L$.
\end{proof}

\begin{lemma}\label{gonquads}
Let $F_1$ be a $4$-gon. Then there are exactly $12$ lines in $\Lambda$ that do
not intersect any line in $F_1$.  The set of these $12$ lines can be
partitioned into three $4$-gons $F_2,F_3,F_4$ and no other subset of this
set is a $4$-gon.
The set of the remaining $16$ lines can be partitioned into four $4$-gons
$G_1,G_2,G_3,G_4$ in such a way that,
for every $i,j\in \{1,2,3,4\}$, each line
in $F_i$ intersects exactly one of the lines in $G_j$ and each line
in $G_i$ intersects exactly one of the lines in $F_j$.
For any different $i,j\in \{1,2,3,4\}$, no line in $F_i$ intersects a
line in $F_j$ and no line in $G_i$ intersects a line in $G_j$.
There are $\omega,\psi \in \Omega$ such that the $F_i$ and $G_i$ are
the orbits of $\Lambda$ under the action of $\Phi_{\omega\psi}$. If $L$
is a line in $F_1$, then for $i \in \{2,3,4\}$ there are
$\theta,\theta' \in \Omega \setminus \{\omega,\psi\}$ such that $F_i$
is the orbit of $L_{\theta\theta'}$, and for $j \in \{1,2,3,4\}$ there
is a $\theta \in \Omega \setminus \{\omega,\psi\}$ such that $G_j$
is the orbit of $L_{\theta}$.
\end{lemma}
\begin{proof}
Let $\omega,\psi \in \Omega$ be such that $F_1$ is an orbit under
$\Phi_{\omega\psi}$, and let $L$ denote some line in $F_1$.
For any $\theta, \theta' \in
\Omega\setminus\{\omega, \psi\}$ the lines $\muvar_{\theta\theta'} L$
and $\muvar_{\omega \theta\theta'} L$ do not intersect any line in $F_1$
by Lemma \ref{gonint}. This gives $12$ lines and one checks
that the only three $4$-gons contained in the set of these $12$ lines are
of the form
$$
\left\{\muvar_{\theta_1\theta_2} L,\muvar_{\theta_3\theta_4} L,
\muvar_{\omega\theta_1\theta_2} L,\muvar_{\omega\theta_3\theta_4} L\right\},
$$
for some permutation $(\theta_i)_i$ of the elements in
$\Omega\setminus\{\omega, \psi\}$. From the equalities
$\muvar_{\theta_3\theta_4} L = \muvar_{\omega\psi\theta_1\theta_2} L$
and $\muvar_{\omega\theta_3\theta_4} L = \muvar_{\psi\theta_1\theta_2}
L$ we deduce that these $4$-gons are also orbits under
$\Phi_{\omega\psi}$, each containing an element $L_{\theta\theta'}$ for
some $\theta, \theta' \in \Omega\setminus\{\omega, \psi\}$. The
remaining $16$ lines do intersect a line in $F_1$ by Lemma \ref{gonint}
and the only four $4$-gons
contained in the set of these $16$ lines are of the form
$$
\left\{
\muvar_\theta L, \muvar_{\omega\theta} L, \muvar_{\psi\theta} L,
\muvar_{\omega\psi\theta} L \right\},
$$
for some $\theta \in \Omega\setminus\{\omega, \psi\}$. Clearly these
$4$-gons are orbits under $\Phi_{\omega\psi}$ as well. The remaining
statements follow from Lemma \ref{gonint}.
\end{proof}

\begin{definition}
An {\em \gonquad} is a quadruple $\cS = \{S_1,S_2,S_3,S_4\}$ of
$4$-gons such that
the $S_i$ are pairwise disjoint and no line in $S_i$ intersects a line
in $S_j$ for $i \neq j$. A {\em \dgonquad} is an unordered pair
$\{\cS_1,\cS_2\}$
of \gonquads, such that $\bigcup_{S\in \cS_1}$ and $\bigcup_{S\in
  \cS_2}$ are disjoint. For any \dgonquad{} $\{\cS_1,\cS_2\}$ we say
that $\cS_2$ is the {\em complementary \gonquad} of $\cS_1$.
\end{definition}

\begin{remark}\label{paramgalleries}
Lemma \ref{gonquads} says that each $4$-gon is contained in a unique
\gonquad, which is contained in a unique \dgonquad. It also implies
that the set of \dgonquads{} is in bijection with the set of $15$ pairs of
different elements in $\Omega$. Figure \ref{figints}
displays a \dgonquad{} and the intersections among all the $32$ lines in
$\Lambda$. In Figure \ref{figints} we use the notation $I$ for $\muvar_I L$
for some fixed line $L$. The elements of $\Omega$ are denoted by
$\omega,\psi,1,2,3,4$. The two \gonquads{} are made up by the $4$-gons
on the bottom and the left of the figure respectively.
\end{remark}

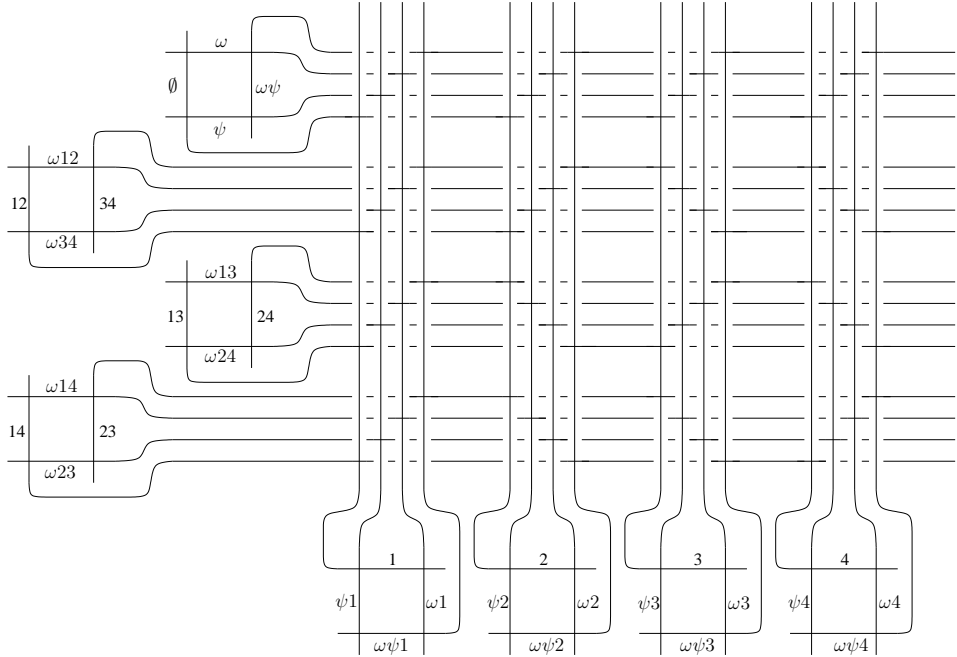
\begin{figure}[ht]
\begin{center}
\resizebox{5in}{!}{\input{allints.pstex_t}}
\end{center}
\caption{the intersections among the $32$ lines in $\Lambda$}
\label{figints}
\end{figure}


\begin{lemma}\label{adjunction}
For each smooth curve $C$ of genus $g$ on a K3 surface, we have
$C^2 = 2g-2$.
\end{lemma}
\begin{proof}
The adjunction formula gives $C\cdot (C+K)  = 2g-2$, where $K$ is the
canonical divisor of the surface. The lemma follows from the fact that
the canonical divisor of a K3 surface is trivial.
\end{proof}

\begin{proposition}\label{seventeen}
The elements of $\Lambda$ generate a sublattice of the N\'eron-Severi group
$\NS(\Vbar)$ of rank $17$ and discriminant $64$.
\end{proposition}
\begin{proof}
Lemma \ref{adjunction} implies that $L^2=-2$
for all $L\in \Lambda$. From Corollary \ref{whichinter} we can deduce all
other intersection numbers among elements of $\Lambda$. This gives a
$32\times 32$ Gram matrix of intersection numbers that has rank $17$.
The matrix also allows us to pick a basis of this sublattice.
The Gram matrix with respect to such a basis turns out to have
determinant $64$.
\end{proof}

\begin{proposition}\label{sixteenplusrkNS}
The rank of the N\'eron-Severi group $\NS(\Vbar)$ equals
$16 + \rk \NS(J)$, where $J$ is the Jacobian of
the curve given by $y^2 = f(x)$.
\end{proposition}
\begin{proof}
The surface $\Vbar$ is isomorphic to the desingularized Kummer
surface associated to $J$ by \cite{CF}, Chapter 16. The statement
therefore follows from \cite{shi}, Prop. 1.
\end{proof}

\begin{proposition}\label{fullNS}
Generically the lines in $\Lambda$ generate a lattice of
finite index in the N\'eron-Severi group $\NS(\Vbar)$.
\end{proposition}
\begin{proof}
Let $J$ be as in Proposition \ref{sixteenplusrkNS}. Generically
we have $\rk \NS(J)=1$, so $\rk \NS(\Vbar) = 17$. By Proposition
\ref{seventeen} the elements of $\Lambda$ generate a lattice of
rank $17$ as well, so this lattice has finite index in $\NS(\Vbar)$.
\end{proof}

In fact the finite index in Proposition \ref{fullNS} is equal
to $1$ as we will see in Proposition \ref{reallyfull}. The reason
for stating that result separately is that one can compute
the rank of the N\'eron-Severi group in explicit cases using
methods from \cite{heron} and \cite{picone}.

For any $4$-gon $S$ let $D_S$ denote the divisor that is the sum of
the lines in $S$.

\begin{lemma}\label{hyp}
Let $S$ and $S'$ be two $4$-gons in complementary \gonquads. Then the
image of $D_{S}+D_{S'}$ in $\Pic \Vbar$ is the class of hyperplane sections.
\end{lemma}
\begin{proof}
Let $L$ be a line in $S$ and let $\omega,\psi\in \Omega$ be such that
$S$ is the orbit of $L$ under $\Phi_{\omega\psi}$. By Lemma
\ref{gonquads}, the $4$-gon $S'$ is also an orbit under
$\Phi_{\omega\psi}$. It follows from Lemma \ref{gonint} that there is a
$\theta \in \Omega \setminus \{\omega,\psi\}$ such that $S'$ is the
orbit of $\muvar_\theta L$. We deduce
$$
D_S+D_{S'} = L+\muvar_{\omega} L +\muvar_{\psi} L+ \muvar_{\psi\omega}
L+\muvar_{\theta}L+\muvar_{\theta\omega} L +\muvar_{\theta\psi} L+
\muvar_{\theta\psi\omega} L.
$$
One checks that for each $L'\in \Lambda$ we have $(D_S+D_{S'}) \cdot L' = 1$.
For a hyperplane section $H$ we also have $H\cdot L'=1$ for all $L'
\in \Lambda$. Since the intersection pairing on $\Pic \Vbar$ is
nondegenerate and the lines generically generate a lattice of
finite index in $\Pic \Vbar$ by
Proposition \ref{fullNS}, we find that generically $D_S+D_{S'}$ is a
hyperplane section. By specializing the transcendentals,
this implies that $D_S+D_{S'}$ is always a hyperplane section.
\end{proof}

\begin{lemma}\label{fourgonslineq}
Let $S$ and $S'$ be two $4$-gons in the same \gonquad. Then $D_{S}$
and $D_{S'}$ are linearly equivalent.
\end{lemma}
\begin{proof}
Let $S''$ be any $4$-gon in the complementary \gonquad, and let $H$ denote a
hyperplane section. Then by Lemma \ref{hyp} both
$D_{S}$ and $D_{S'}$ are linearly equivalent with $H-D_{S''}$.
\end{proof}

\begin{proposition}\label{K3ell}
Let $X$ be a K3 surface over a field and $F$ a
reduced and connected curve on $X$ that
satisfies $F^2=0$.
Suppose further that the linear system $|F|$ does not have a base curve.
Then there is an elliptic fibration $X \ra
\P^1$ whose fibers are the elements of $|F|$.
Up to an automorphism of $\P^1$ this fibration is unique.
\end{proposition}
\begin{proof}
By the adjunction formula we have $F \cdot (F+K_X) = 2p_a-2$, where $p_a$
is the arithmetic genus of $F$, but $F^2 = 0$ and $K_X = 0$, so $p_a = 1$.
By the Riemann-Roch theorem for surfaces (\cite{hag}, thm.\ V.1.6),
we have $h^0(\O_X(-F))
- h^1(\O_X(-F)) + h^0(\O_X(K_X+F)) = 1/2(F \cdot (F-K_X))
+ 1 + p_a = 2$.  Here
$h^0(\O_X(F)) = 0$ because $F$ is a nonzero
effective divisor.  Let us prove that $h^1(\O_X(-F)) = 0$.
From the exact sequence $0 \ra \O_X(-F) \ra \O_X \ra \O_F \ra 0$ of sheaves
on $X$ we obtain the exact sequence of cohomology groups
$$H^0(X,\O_X) \ra H^0(X,\O_F) \ra H^1(X,\O_X(-F)) \ra H^1(X,\O_X).$$
Since $F$ is reduced and
connected, $H^0(X,\O_F)$ consists only of sections constant
on $F$, and so the map from $H^0(X,\O_X)$ is surjective.  On the other
hand, $H^1(X,\O_X) = 0$ because $X$ is a K3 surface.  It follows that
$H^1(X,\O_X(-F)) = 0$ as claimed, and therefore that $h^0(\O_X(-F)) = 2$.

The only maps whose fibers are elements of the linear system $F$ are those
associated to subseries of the complete linear system $|F|$.  Since
$\O_X(-F)$ has a $2$-dimensional space of sections, the only nonconstant
maps of this kind are those associated to the complete linear system.
This map is a fibration,
for by hypothesis there is no curve contained in all divisors
in the linear system $|F|$,
so $F^2 = 0$ implies that no two fibers
intersect.
\end{proof}

\begin{lemma}\label{zeroint}
For any $4$-gon $S$ we have $D_{S}^2=0$.
\end{lemma}
\begin{proof}
Write $D_S = D_1 + D_2 + D_3 + D_4$, where the $D_i$ are the geometric
irreducible components of $D_S$.
By Lemma \ref{adjunction} each $D_i$ has self-intersection $-2$.
Also, $D_i \cdot D_{i+1} = 1$ and $D_i \cdot D_{i+2} = 0$, with indices read
mod $4$.  The self-intersection of
$D_S$ is therefore $4 \cdot -2 + 4 \cdot 2 = 0$.
\end{proof}

\begin{lemma}\label{Vbarell}
Let $\cS$ be an \gonquad. Then there is an elliptic fibration of
$\Vbar$ for which the lines in each $4$-gon $S\in \cS$ are the
irreducible components of a fiber.
Up to an automorphism of $\P^1$ this fibration is unique.
\end{lemma}
\begin{proof}
By Proposition \ref{threequad} the surface $\Vbar$ is a K3
surface. By Lemma \ref{zeroint} we have $D_{S}^2=0$ for any $4$-gon
$S\in \cS$. By Lemma
\ref{fourgonslineq} the effective divisors $D_{S}$ with $S \in \cS$
are all contained in the same linear system. The lemma now follows
immediately from Proposition \ref{K3ell}.
\end{proof}

\begin{remark}\label{paramfibs}
Since the \gonquads{} come in pairs, so do the elliptic fibrations
mentioned in Lemma \ref{Vbarell}. By Remark \ref{paramgalleries}
these pairs of fibrations are parametrized by the $15$ pairs of
elements in $\Omega$.
\end{remark}

It is known that generically the lines associated to the nodes
and the tropes on the desingularized
Kummer surface generate the full N\'eron-Severi group (see Remark
\ref{nodestropes}). Together with Propositions
\ref{seventeen} and \ref{fullNS}, the following statement
is slightly stronger.

\begin{proposition}\label{reallyfull}
If $\rk \NS(\Vbar) = 17$, then the lines in $\Lambda$ generate
the full N\'eron-Severi group $\NS(\Vbar)$.
\end{proposition}
\begin{proof}
Let $L$ denote the sublattice of $\NS(\Vbar)$ generated by
the elements of $\Lambda$. By Proposition \ref{seventeen}, the
lattice $L$ has finite index in $\NS(\Vbar)$. Suppose this index
is not $1$. Then from the equality $\disc L = [\NS(\Vbar) : L]^2
\cdot \disc \NS(\Vbar)$, we find it is divisible by $2$, and
$\disc \NS(\Vbar)$ is a divisor of $64/2^2 = 16$.

Consider the transcendental lattice $T_{\Vbar}$
of $\Vbar$, which is the orthogonal complement of $\NS(\Vbar)$
in $H^2(\Vbar,\Z)$. As the lattice $H^2(\Vbar,\Z)$ is unimodular,
we have $|\disc T_{\Vbar}| = |\disc \NS(\Vbar)|$, so $\disc T_{\Vbar}$
is a divisor of $16$ as well. However, since $\Vbar$ is isomorphic
to the Kummer surface associated to the Jacobian $J$ of the curve
$y^2 = f(x)$ (see Remark \ref{Vistwist}), we find from
\cite{mor}, Prop. 4.3, that $T_{\Vbar}$ is isomorphic to $T_J(2)$,
the transcendental lattice of $J$, scaled by a factor of $2$.
Since $T_{\Vbar}$ has rank $22-17=5$, its discriminant is divisible
by $2^5 = 32$. From this contradiction we conclude that the index
does equal $1$.
\end{proof}

\subsection{Fields of definition}

Recall that $T(F) = \{ \sqdelta \in A_{F} \,\, : \,\, \sqdelta^2 = \delta\}$
for any field $F$ for which $\delta$ is contained in $A_F$
and that $l=k(\Omega)$ is the splitting field of $f$.
Also recall that if $\omega \in F$ is a root of $f$, then
$\varphi_\omega$ denotes the map $A_F \rightarrow F$ that sends
$g(X)$ to $g(\omega)$. These maps induce the isomorphism
$\varphi \colon A_l \rightarrow \bigoplus_{\omega \in \Omega} l$
given by $X \mapsto (\varphi_\omega(X))_\omega = (\omega)_\omega$.

\begin{lemma}\label{Ptoomega}
For any $\sqdelta \in T(\kbar)$ and $\omega,\psi \in \Omega$ we have
$\varphi_\psi(P_{\sqdelta,\omega})=0$ if and only if $\psi=\omega$.
\end{lemma}
\begin{proof}
This follows from the definition of $P_{\sqdelta,\omega}$ and the fact
that $\sqdelta\in \Abar$ is a unit.
\end{proof}

For any object $Y$ to which we can apply every Galois automorphism
in $G(\kbar/k)$ we will say that $Y$ is defined over the field
extension $k'\subset \kbar$ of $k$ if for all $\sigma \in
G(\kbar/k')$ we have ${}^\sigma Y = Y$. The smallest field over
which $Y$ is defined will be called the field of definition of $Y$ and
denoted by $k(Y)$. Note that every element of $\tT(k')$ is defined
over $k'$, even though it may be represented by an element in
$T(k')$ that is only defined over a quadratic extension of $k'$.
Note that if $Y=(y_1,\ldots, y_n)$ is a sequence,
then $k(Y)$ is the composition of all the $k(y_i)$. If
$Z=\{z_1,\ldots,z_n\}$ is a set, then $k(Z)$ may be strictly
smaller then the field of definition of the sequence $(z_1,\ldots,z_n)$,
a field that we will denote by $k([Z])$.

\begin{lemma}\label{lkLkLambda}
Inside $\kbar$ we have $l\cdot k(L) = k([\Lambda])$ for all $L \in \Lambda$.
\end{lemma}
\begin{proof}
Suppose $\sigma \in G(\kbar/k)$ acts trivially on $\Lambda$. Then for all
$\omega \in \Omega$ and $\sqdelta \in T(\kbar)$ the automorphism $\sigma$
fixes the intersection point $P_{\sqdelta,\omega}$ of $L_\sqdelta$ and
$\muvar_\omega L_\sqdelta$. The point $P_{\sqdelta,\omega}$ determines
$\omega$ uniquely by Lemma \ref{Ptoomega}. We conclude that
$\sigma$ fixes all $\omega\in \Omega$, so $\sigma$ fixes $l$ and $l$
is contained in $k([\Lambda])$. Clearly we also have $k(L) \subset
k([\Lambda])$, so we find $l \cdot k(L) \subset k([\Lambda])$. For every
other $L' \in \Lambda$ there is a $\muvar \in \mu(A_l)$ such that
$\muvar L = L'$. As the
automorphism $[\muvar]$ is defined over $l$, we find that $L'$ is
defined over $l\cdot k(L)$, so $k(L')\subset l\cdot k(L)$. This holds
for all $L'\in \Lambda$ so we find $k([\Lambda]) \subset l\cdot k(L)$ and thus
$k([\Lambda]) = l\cdot k(L)$.
\end{proof}

For every $\omega$ we fix a square root $\sqrt{\delta_\omega}$ of $\delta_\omega
= \varphi_\omega(\delta)$ in $\kbar$, yielding also a fixed square
root $\sqdelta_0 = \varphi^{-1}\big((\sqrt{\delta_\omega})_{\omega \in
  \Omega}\big)$ of $\delta$.
Note that with the Legendre polynomials $P_\omega$ of Remark
\ref{legendre} we can write $\sqdelta_0 = \sum_\omega
\sqrt{\delta_\omega} P_\omega$. We define the fields
$$
m' = l(\{\sqrt{\delta_\omega}\,\,:\,\, \omega\in \Omega\}), \qquad
\mbox{and} \qquad m =l(\{\sqrt{\delta_\omega}
\sqrt{\delta_\psi}\,\,:\,\, \omega,\psi\in \Omega\}).
$$
The square root $\sqdelta_0$ of $\delta$ trivializes the torsors
$T$ and $\tT$ under $\mu_A$ and $\tmu$ respectively, identifying
$\muvar \in \mu_A(\kbar) = \mu(\Abar)$ with $\muvar \sqdelta_0 \in
T(\kbar)$. By Remark \ref{torsorslines} the $k$-torsors
$\tT(\kbar)$ and $\Lambda(\kbar)$ under $\tmu({\kbar})$
are isomorphic over $k$ as well, identifying the class of
$\sqdelta$ in $\tT(\kbar)$ with the line $L_\sqdelta$.
Just after Lemma \ref{zott} we identified the subset $I \subset
\Omega$ with the element $\muvar_I \in \mu_A(\kbar)$. Similarly,
we write $\sqdelta_I = \muvar_I \sqdelta_0$. We also set $L_0 =
L_{\sqdelta_0}$ and write $L_I = \muvar_I L_0 = L_{\sqdelta_I}$.
Note that $L_I = L_{\Omega - I}$.

\begin{lemma}\label{whichI}
Fix $\sigma \in G(\kbar/k)$. Then ${}^\sigma\! L_0 = L_I$ if and
only if $I$ or $\Omega-I$ equals
$$
\Big\{{}^\sigma\!\omega\,\,:\,\, \omega \in\Omega\,\,,\,\, {}^\sigma\!
\sqrt{\delta_\omega}
=\sqrt{\delta_{{}^\sigma\!\omega}}\Big\}.
$$
\end{lemma}
\begin{proof}
This follows immediately from Lemma \ref{transfree} and the equation
$$
{}^\sigma\! \sqdelta_0 =
\varphi^{-1}\left(\Big({}^\sigma\!
\sqrt{\delta_{{}^{\sigma^{-1}}\!\omega}}\Big)_{\omega \in
  \Omega}\right) = \muvar_J \sqdelta_0,
$$
where $J$ is the set given in the lemma.
\end{proof}

\begin{lemma}\label{kLambdam}
We have $k([\Lambda]) = m$.
\end{lemma}
\begin{proof}
For every $\psi\in \Omega$ the element $\sqrt{\delta_\psi}\sqdelta_0 =
\sum_\omega\sqrt{\delta_\psi}\sqrt{\delta_\omega} P_\omega$
is defined over $m$, where $P_\omega$ is the Legendre polynomial
introduced in Remark \ref{legendre}.
The line $L_0$ corresponds to the
  subspace $\{(\sqrt{\delta_\psi}\sqdelta_0)^{-1}(sX+t)\,\,:\,\, s,t \in m\}$
  of $A_m$, so $L_0$ is defined over $m$ as well and we have
  $k(L_0) \subset m$. From Lemma \ref{lkLkLambda} we deduce $k([\Lambda])
  \subset  m$. For the converse,
  consider $\sigma \in G(\kbar/k([\Lambda]))$.
From Lemma \ref{whichI} and the equation $L_0 = {}^\sigma\! L_0$ we find
that either we have ${}^\sigma\!\sqrt{\delta_\omega} = \sqrt{\delta_\omega}$
for all $\omega$, or we have ${}^\sigma\!\sqrt{\delta_\omega} =
-\sqrt{\delta_\omega}$ for
all $\omega$. In both cases we find that
$\sigma$ acts trivially on $m$, so we also have $m \subset k([\Lambda])$.
\end{proof}

\begin{lemma}\label{gonquadfields}
Let $\cS = \{\cS_1,\cS_2\}$ be a \dgonquad. Then there are $\omega,\psi
\in \Omega$ such that we have
$$
k(\cS) = k(\omega+\psi,\omega\psi)\qquad \mbox{and} \qquad k(\cS_1)
= k(\cS_2) = k\Big(\omega+\psi,\omega\psi,\prod_{\theta \in
\Omega \setminus \{\omega,\psi\}} \sqrt{\delta_\theta}\Big).
$$
\end{lemma}
\begin{proof}
Let $\omega,\psi$ be such that the $4$-gons in the $\cS_i$ are orbits
under $\Phi_{\omega\psi}$.
Write $k' = k(\omega+\psi,\omega\psi)$ and suppose we have $\sigma \in
G(\kbar/k')$. Then $\sigma$ fixes the polynomial
$(x-\omega)(x-\psi)$, so it permutes $\omega$ and $\psi$. Therefore, $\sigma$
permutes the orbits under $\Phi_{\omega\psi}$, which are the $4$-gons in
$\cS_1 \cup \cS_2$ (compare Lemma \ref{gonquads}).
Since every $4$-gon is contained in a unique
\dgonquad, this implies that $\sigma$ fixes $\cS$, so we have $k(\cS)
\subset k'$. For the converse, suppose we have $\sigma \in
G(\kbar/k(\cS))$. Then $\sigma$ permutes the intersection points
among any two lines in the same $4$-gon in $\cS_1 \cup \cS_2$. These
intersection points are all of the form $P_{\sqdelta,\omega}$ or
$P_{\sqdelta, \psi}$ for some $\sqdelta \in T(\kbar)$. By Lemma
\ref{Ptoomega} this implies that $\sigma$ permutes $\omega$ and
$\psi$, so it acts trivially on $k'$ and we find $k' \subset k(\cS)$,
so $k' = k(\cS)$. For the second equality, set $N = \prod_{\theta \in
\Omega \setminus \{\omega,\psi\}} \sqrt{\delta_\theta}$ and consider
$\sigma \in G(\kbar/k')$. Then $\sigma$ fixes $\cS$, so it permutes
$\cS_1$ and $\cS_2$, and $\sigma$ sends $N$ to $\pm N$.
Let $n$ denote the number of $\theta \in \Omega
\setminus \{\omega,\psi\}$ for which we have ${}^\sigma\!
\sqrt{\delta_\theta} = -\sqrt{\delta_{{}^\sigma\! \theta}}$. Then
$\sigma$ fixes $N$ if and only if $n$ is even.
Let $I \subset \Omega$ be such that
${}^\sigma\! \sqdelta_0 = \muvar_I \sqdelta_0$. Then we have
$n = \# I \cap \Omega \setminus \{\omega,\psi\}$ and
${}^\sigma\! L_0 = L_I$. Suppose that $n=0$ or
$n=4$. Then $L_0$ and ${}^\sigma\! L_0=L_I$ are in the same orbit under
$\Phi_{\omega\psi}$, so in the same $4$-gon in $\cS_1 \cup \cS_2$. By
Lemma \ref{gonquads} each $4$-gon is contained in a unique \gonquad,
so $\sigma$ fixes $\cS_1$ or $\cS_2$, and thus both.
Now suppose $n\in\{1,2,3\}$. Then ${}^\sigma\! L_0$ is in a different
orbit under
$\Phi_{\omega\psi}$ than $L_0$. By Lemma \ref{gonint} the number $n$ is odd
if and only if the line $L_0$ intersects some line in the orbit of
${}^\sigma\! L_0$, which,  by Lemma \ref{gonquads},
happens if and only if $L_0$ and ${}^\sigma\! L_0$ are
in opposite \gonquads{}.
We conclude that for all $n$ the automorphism $\sigma$ fixes $\cS_1$
and $\cS_2$ if and only if $n$ is even, so if and only if $\sigma$
fixes $N$. This implies that $k(\cS_1) = k(\cS_2) =  k'(N)$.
\end{proof}

\begin{remark}\label{factorell}
The first equality of Lemma \ref{gonquadfields} is not surprising
as we already saw in Remark \ref{paramgalleries} that
\dgonquads{} are parametrized by pairs of elements in $\Omega$. Note
that $k' = k(\cS) = k(\omega+\psi,\omega\psi)$ is the smallest
field over which $f$ factors as $f = f_2f_4$, where $f_2$ has
degree $2$ and roots $\omega$ and $\psi$. It is the field of
definition of the $2$-torsion point $(\omega,0)+(\psi,0)-2\cdot
\infty$ on the Jacobian of the curve $y^2 = f(x)$. If we set $A_4
= k'[X]/f_4$ and we let $\delta'$ denote the image of $\delta$
under the natural map $A_{k'} \ra A_4$ then the element
$\prod_\theta \sqrt{\delta_\theta}$ in Lemma \ref{gonquadfields}
is a square root of the norm $N_{A_4/k'} \delta'$ of $\delta'$
from $A_4$ to $k'$. The two elliptic fibrations associated to
$\cS_1$ and $\cS_2$ in Lemma \ref{Vbarell} are also defined over
the field $k(\cS_1) = k(\cS_2) = k'(\sqrt{N_{A_4/k'} \delta'})$.
We will say that these are the elliptic fibrations associated to
the pair $(\omega,\psi)$, or to the factorization $f=f_2f_4$.
The $4$-gons in $\cS_1$ and $\cS_2$ that make up the fibers of these
fibrations are orbits of $\Lambda$ under the group $\Phi_{\omega\psi}$.
In section \ref{sectionellfibs} we will find explicit equations for
these fibrations.
\end{remark}

Let $\Lambda_1$ and $\Lambda_2$ be the two maximal subsets $S$ of
$\Lambda$ for which all lines in $S$ have the same parity.

\begin{lemma}
We have $k(\Lambda_1) = k(\Lambda_2) = k\big(\sqrt{N(\delta)}\big)$, where $N
= N_{A/k}$ is the norm from $A$ to $k$.
\end{lemma}
\begin{proof}
Take $\sigma \in G(\kbar/k)$. Let $I \subset \Omega$ be such that
${}^\sigma\! \sqdelta_0 = \sqdelta_I$, and set $n = \# I$.
The automorphism $\sigma$ permutes $\Lambda_1$ and
$\Lambda_2$, so it fixes both if and only if $L_0$ and ${}^\sigma\! L_0=L_I$
have the same parity, i.e., if and only if $n$ is even. Note that
$n$ also equals the number of $\omega \in \Omega$ with ${}^\sigma\!
\sqrt{\delta_\omega} = -\sqrt{\delta_{{}^\sigma\! \omega}}$, so $n$ is
even if and only if $\sigma$ fixes the element $\prod_\omega
\sqrt{\delta_\omega}=\sqrt{N(\delta)}$. We conclude
that $\sigma$ fixes $\Lambda_1$ and $\Lambda_2$ if and only if
$\sigma$ fixes $\sqrt{N(\delta)}$, which shows that $k(\Lambda_1) =
k(\Lambda_2) = k\big(\sqrt{N(\delta)}\big)$.
\end{proof}

Let $\Aut \Lambda$ denote the group of permutations of $\Lambda$ that
respect the intersection pairing. Let $\rho\colon G(\kbar/k) \ra \Aut
\Lambda$ be the corresponding Galois representation.

\begin{lemma}\label{kernelrho}
The kernel of the representation $\rho\colon G(\kbar/k) \ra \Aut
\Lambda$ is $G(\kbar/m)$.
\end{lemma}
\begin{proof}
This follows from Lemma \ref{kLambdam}.
\end{proof}

\begin{proposition}\label{exactgal}
All extensions among the fields $k \subset l \subset m \subset m'$ are
Galois and we have exact sequences
$$
\eqalign{
1 \ra \Gal(m'/m) \ra \Gal&(m'/l) \ra \Gal(m/l) \ra 1 \cr
1 \ra \Gal(m/l) \ra \Gal&(m/k) \ra \Gal(l/k) \ra 1 \cr
1 \ra \Gal(m'/m) \ra \Gal(m'/l)&\ra \Gal(m/k) \ra \Gal(l/k) \ra 1.\cr
}
$$
\end{proposition}
\begin{proof}
The extension $l/k$ is normal because $l$ is the splitting field of
$f$ over $k$, and separable because $f$ is.
Since $[m':l] = 64 = 2^6$, and we have assumed that $\car k \ne 2$,
the extension $m'/k$ and all subextensions are separable.
The group $\Gal(\kbar/m)$ is normal in $\Gal(\kbar/k)$ because it
is the kernel of $\rho$ by Lemma \ref{kernelrho}. This
means that $m/k$ is Galois, and therefore so is $m/l$.
Similarly, the group $\Gal(\kbar/m')$ is normal in $\Gal(\kbar/k)$
because it is the kernel of the representation $\Gal(\kbar/k) \ra
\Aut T(\kbar)$. This
implies that $m'/k$ is Galois, which also follows from the fact that
$m'$ is obtained from $l$ by adjoining a square root of an element in
$l$ as well as the square roots of all conjugates of that element
under $\Gal(l/k)$. Therefore, the extensions $m'/m$ and
$m'/l$ are Galois, too. The first two exact sequences are
the standard short exact sequences of Galois groups associated to
the double extensions $k \subset l \subset m$ and $l \subset m \subset
m'$. They can be combined to give the last sequence.
\end{proof}

\begin{example}\label{generic}
Let $F$ be any field and define the generic fields
$$
\eqalign{
m_g' &= F(\omega_1,\ldots,\omega_6,d_0,\ldots,d_5)[T_1, \ldots,
  T_6] / \left(T_j^2 - \sum_{i=0}^5 d_i \omega_j^i \,\, :
\,\, 1\leq j \leq 6\right),\cr
m_g &= F(\omega_1,\ldots,\omega_6, d_0,\ldots, d_5,
\{\sqd_i\sqd_j\,\,:\,\,1 \leq i,j\leq 6\}),\cr
l_g &= F(\omega_1,\ldots,\omega_6,d_0,\ldots, d_5),\cr
k_g&= F(s_1,\ldots,s_6,d_0,\ldots,d_5),\cr
}
$$
where $\omega_1,\ldots,\omega_6,d_1, \ldots, d_6$ are
independent transcendentals, $s_j$ denotes the elementary symmetric
polynomial of degree $j$ in the variables $\omega_1,\ldots, \omega_6$,
and $\sqd_j$ is the image of $T_j$ in $m_g'$. We have $k_g \subset l_g
\subset m_g \subset m_g'$. Set
$$
f=\prod_{j=1}^6 (X-\omega_j) =
X^6-s_1X^5+s_2X^4-s_3X^3+s_2X^4-s_5X^5+s_6 \in k_g[X],
$$
and define $A = k_g[X]/f$. By abuse of notation we will also write $X$
for the image of $X$ in $A$. Set $\delta = \sum_{i=0}^5 d_iX^i \in A$.
The evaluation maps $\varphi_j\colon \, A \ra l_g$ sending $X$ to
$\omega_j$ induce an isomorphism $\varphi \colon \, A_{l_g} \ra
\bigoplus_{j=1}^6 l_g$. We have $\sqd_j^2 = \delta_j$ with $\delta_j =
\varphi_j(\delta)$, so
the fields $l_g$, $m_g$, and $m_g'$ depend on $k_g$, $f$ and $\delta$ exactly
as the corresponding fields without the subscript $g$ for ``generic'' did
before, abbreviating $\omega_j$ to $j$ in any index.
We will give explicit equations for the intersection points
of the lines in $\Lambda_g$ in this generic situation.
As in Remark \ref{legendre}, let $P_j$ denote the Legendre polynomial
$P_j = \prod_{i \neq j} (X-\omega_i)/(\omega_j-\omega_i) \in A_{l_g}$ for
$1\leq j \leq 6$. Then $\varphi^{-1}$ sends
$(c_j)_{j=1}^6$ to $\sum_{j=1}^6 c_jP_j$. Set $\sqdelta_0 =
\varphi^{-1}\big((\sqd_j)_j\big) = \sum_{j=1}^6
\sqd_jP_j$. Then we have $\sqdelta_0^2 = \delta$. Let $b_{ji}\in
F(\omega_1,\ldots,\omega_6)$ be such that $P_j = \sum_{i=0}^5 a_{ji}
X^i$. In $A_{l_g}$ we have $XP_j=\omega_jP_j$, so we find that the
coordinates of the point
$P_{\varepsilon_0,\omega_r}$ in terms of the $a_i$
are given by the coefficients of
$$
\sqdelta_0^{-1}(X-\omega_r) = \sum_{j=1}^6
\sqd_j^{-1}(X-\omega_r)P_j = \sum_{j=1}^6
\sqd_j^{-1}(\omega_j-\omega_r)P_j =
\sum_{i=0}^5 \left(\sum_{j=1}^6\sqd_j^{-1}
(\omega_j-\omega_r)b_{ji}\right) X^i.
$$
Multiplying all the coefficients by one of the $\sqd_j^{-1}$ shows that
the point $P_{\sqdelta_0,\omega_r}$ is indeed defined over $m_g$. All other
intersection points are obtained by replacing some of the $\sqd_j$
by their negatives and $r$ by some $r' \in \{1,\ldots,6\}$. By
specialization, these formulas give explicit equations for the intersection
points of the lines in $\Lambda$ over any field. This also gives all the lines.
Note that the group $\Gal(m_g'/l_g)$ is isomorphic to
$\bigoplus_{j=1}^6 \Z/2\Z$, where the generator of the $j$-th
component sends $\sqd_j$ to $-\sqd_j$. The group
$\Gal(m_g'/m_g)\isom \Z/2\Z$ embeds diagonally into $\Gal(m_g'/l_g)$, sending
every $\sqd_j$ to $-\sqd_j$. Hence, the group $\Gamma=\Gal(m_g/l_g)$ is
isomorphic to $(\Z/2\Z)^6/(\Z/2\Z)$. The group $\Gal(l_g/k_g)$ is
isomorphic to $S_6$. There is a section $\iota'$ of the homomorphism
$\Gal(m_g'/k_g) \ra \Gal(l_g/k_g)$ that sends an element $\sigma \in
\Gal(l_g/k_g)$ to the unique lift that sends the set $\{\sqd_1,
\ldots, \sqd_6\}$ to itself. The composition $\iota$ of $\iota'$ and the
restriction map $\Gal(m_g'/k_g)\ra \Gal(m_g/k_g)$ is a section of the homomorphism
$\Gal(m_g/k_g) \ra \Gal(l_g/k_g)$. Through $\iota$ the group
$\Gal(l_g/k_g)\isom S_6$ acts on $\Gamma$ by conjugation. This action
is induced by permutation of the
components of $\bigoplus_{j=1}^6 \Z/2\Z = \Gal(m_g'/l_g)$ in the obvious
way. Since the middle sequence of
Proposition \ref{exactgal} splits in this generic case,
we find that $\Gal(m_g/k_g)$ is isomorphic to the semidirect product
$\Gamma \semi S_6$, which has $32\cdot 6! = 23040$ elements.
\end{example}

\begin{proposition}\label{rhosurj}
Generically, the representation $\rho\colon\, \Gal(\kbar/k) \ra \SQi$ is
surjective.
\end{proposition}
\begin{proof}
It suffices to show that $\rho$ is
surjective in the case of the generic situation of
Example \ref{generic}, so suppose we are in that case.
We will use the same notation as in Example \ref{generic}, including
abbreviating $\omega_j$ to $j$ in indices of lines and points. Take any
$\tau\in \SQi_g$ and consider the line $L_0 = L_{\sqdelta_0}$. Since $\Gamma =
\Gal(m_g/l_g)$ acts transitively on $\Lambda_g$, there is a $\sigma_1 \in
\Gamma$ with $\rho(\sigma_1)(L_0) = \tau(L_0)$. Then $\tau' =
\rho(\sigma_1)^{-1}\tau$ fixes $L_0$, so it permutes the six lines $L_j$
that intersect $L_0$. The corresponding six intersection points are
$P_{\sqdelta_0,\omega_j}$ for $1\leq j \leq 6$, so $\tau'$ induces
a unique permutation of the $\omega_j$ by Lemma \ref{Ptoomega},
which corresponds to an element
$\psi \in \Gal(l_g/k_g)$. Set $\sigma_2 = \iota'(\psi)$. Then $\sigma_2$
sends the set $\{\sqd_1,\ldots,\sqd_6\}$ to itself, so it
fixes $\sqdelta_0 = \sum_{j=1}^6 \sqd_j P_j$ as both the
$\sqd_j$ and the $P_j$ are acted on according to their
indices. This implies that
$\rho(\sigma_2)$ fixes $L_0$, while it permutes the intersection points
$P_{\sqdelta_0,\omega_j}$ in the same way $\tau'$ does. Therefore
$\tau'' = \rho(\sigma_2)^{-1}\tau'$ fixes $L_0$ and the six lines
$L_j$. For $i\neq j$ the line $L_{ij}$ is
the unique line that intersects $L_{i}$ and $L_{j}$ that
is not equal to $L_0$. This implies that $\tau''$ also fixes
$L_{ij}$. Similarly, the line
$L_{ijr}$ is the unique line that intersects
both $L_{ij}$ and $L_{ir}$ that is not
equal to $L_{i}$. This implies that $\tau''$ also fixes
$L_{ijr}$. We conclude that $\tau''$ is the
identity, so $\tau = \rho(\sigma_1\sigma_2)$ and $\rho$ is surjective.
\end{proof}

By Lemma \ref{kernelrho} and Proposition \ref{rhosurj} the generic
representation $\rho_g \colon \Gamma \semi S_6 \isom \Gal(m_g/k_g) \ra
\Aut \Lambda_g \isom \Aut \Lambda$ is an isomorphism. We will denote
the composition $\rho_g \circ \iota \colon S_6 \ra \Aut \Lambda$
by $\iota$ as well.

It will be useful to have names for the elements of $\SQi$. For every
set $I \subset \Omega$, let $s_I\in \SQi$ denote the permutation
induced by $[\muvar_I]$. Note we have $s_I=s_{\Omega\setminus I}$
and $s_I$ and $s_J$ commute for every $I,J \subset \Omega$.
For any permutation $\sigma \in
S_6 = \Sym(\Omega)$, let $t_\sigma$ denote the permutation that sends
$L_I$ to $L_{{}^\sigma\! I}$. For $\sigma, \tau \in S_6$ we have
$t_\sigma \circ t_\tau = t_{\sigma\tau}$, and
\begin{equation}\label{action}
t_\sigma \circ s_I = s_{{}^\sigma\! I} \circ t_\sigma.
\end{equation}
Note that the action of $S_6$ on $\Lambda$ that we have defined
depends on the choice of
$L_0$, or the $\sqrt{\delta_\omega}$, just as the section $\iota$ in Example
\ref{generic} depends on the choice of the square roots $\sqd_j$ of
the $\delta_j$. We will state some of the following lemmas under an
extra condition on $L_0$, knowing that the general case may always be
obtained by changing some of the $\sqrt{\delta_\omega}$ to their
negatives and changing the $L_I$ and $t_\sigma$ accordingly.
By specialization of the generic $\omega_j \in m_g'$ of Example \ref{generic}
to the $\omega \in \Omega$, and the $\sqd_j$ to the corresponding
$\sqrt{\delta_\omega}$, we specialize $k_g,l_g,m_g,\Lambda_g$, and the
corresponding generic representation $\rho_g$
to $k,l,m,\Lambda$, and $\rho$ respectively. Let $r$ denote the associated
injective map from $\Gal(m/k)$ to $\Gal(m_g/k_g)$. Then we have the following
commutative diagram.
$$
\xymatrix{
\Gal(m_g/k_g) \ar[rr]^{\rho_g}_{\isom} & & \Aut \Lambda_g \\
&&\\
\Gal(m/k) \ar[uu]^r \ar[rr]^{\rho} & & \Aut \Lambda \ar@{-}[uu]_{\isom} \\
}
$$

\begin{lemma}\label{specialization}
Let $H$ be a subgroup of $\Aut \Lambda$ and let $H_g$ be the
corresponding subgroup of $\Aut \Lambda_g$. Then the fixed field of
$\rho^{-1} H$ is exactly the specialization of the fixed field of
$\rho_g^{-1} H_g$.
\end{lemma}
\begin{proof}
Set $H' = \rho^{-1} H$ and $H_g' = \rho_g^{-1} H_g$. By the
commutative diagram above we have $r^{-1}(H_g') = H'$.
For every specialization $k'$ of a
subextension $k_g'$ of $m_g$ over $k_g$, we have $\Gal(m/k') =
r^{-1}(\Gal(m_g/k_g'))$. 
In other words, the fixed field of $H' = r^{-1}(H_g')$ is
exactly the specialization of the fixed field of $H_g'$.
\end{proof}

For any $\omega, \psi \in \Omega$, let $\cN_{\omega\psi}$ denote the
group generated by $s_\omega$ and $s_\psi$. Then $\Phi_{\omega\psi}$
acts on $\Lambda$ through $\cN_{\omega\psi}$.
For any \gonquad{} $\cS$, let $G_\cS$ denote the maximal subgroup of $\Aut
\Lambda$ that fixes $\cS$ and let $G_{[\cS]}$ denote the maximal subgroup
that fixes all $4$-gons in $\cS$. We will use Lemma
\ref{specialization} to find generators of $k([\cS])$, the compositum of the
fields $k(S)$ for all $S$ in some \gonquad{} $\cS$. This field will be
used in Section \ref{sectionellfibs} to find explicit equations for
the elliptic fibration associated to $\cS$ in Lemma \ref{Vbarell}.

\begin{lemma}\label{GcS}
For any \gonquad{} $\cS$ the
group $G_\cS$ has order $768$ and the natural homomorphism from
$G_\cS$ to the group $\Sym(\cS)$ of permutations of the $4$-gons in
$\cS$ is surjective. Its kernel is $G_{[\cS]}$.
\end{lemma}
\begin{proof}
In the generic case of Example \ref{generic}, the field $k_g(\cS)$ has
degree $30$ by Lemma \ref{gonquadfields}. Therefore the group
$\Gal(m_g/k_g(\cS))$ has order $23040/30 =768$. By Lemma
\ref{kernelrho} and Proposition \ref{rhosurj}, the
representation $\rho_g \colon \, \Gal(m_g/k_g) \ra \Aut \Lambda_g$ is an
isomorphism, so we find that $G_\cS$ has order
$768$ as well. Clearly the kernel of the homomorphism $\chi \colon
G_\cS \ra \Sym(\cS)$ equals $G_{[\cS]}$.
Let $\omega$ and $\psi$ be such that the $4$-gons in
$\cS$ are orbits under $\Phi_{\omega\psi}$, and set $H = \{t_\sigma \,\,:
\,\, \sigma \in \Sym(\Omega \setminus \{\omega,\psi\})\}\subset \Aut \Lambda$.
Each $h \in H$
sends orbits under the group $\Phi_{\omega\psi}$ to orbits under the same
group, and as we have $h(L_0) = L_0$, the permutation
$h$ fixes at least one of these orbits, so it fixes
the two complementary \gonquads{} associated to the pair
$(\omega,\psi)$. We deduce $H \subset G_\cS$.
Without loss of generality we
will assume that $L_0$ is not contained in any of the $4$-gons of
$\cS$. Then by Lemma \ref{gonquads}, for each $4$-gon $S$ in $\cS$ there
is a $\theta \in \Omega \setminus \{\omega,\psi\}$ such that $S$ is
the orbit of $L_\theta$. It follows that each permutation of $\cS$ is
induced by a permutation of $\Omega \setminus \{\omega,\psi\}$, so the
restriction of $\chi$ to $H$ is surjective, which implies that $\chi$ is
surjective. 
\end{proof}


\begin{lemma}\label{KcS}
Let $\cS$ be an \gonquad, let $\omega,\psi \in \Omega$ be such that
the $4$-gons of $\cS$ are orbits under $\Phi_{\omega\psi}$, and
assume that $L_0$ is contained in one of the $4$-gons of $\cS$. Let
$\theta_1,\theta_2,\theta_3,\theta_4$ be the elements of
$\Omega\setminus \{\omega,\psi\}$.
Then $G_{[\cS]}$ is generated by $s_\omega$, $s_\psi$,
and $t_\sigma$ for $\sigma \in
\langle\,(\omega\,\psi),\, (\theta_1\, \theta_2)(\theta_3\, \theta_4), \,
(\theta_1\, \theta_3)(\theta_2\, \theta_4)\, \rangle$.
\end{lemma}
\begin{proof}
Set $B = \langle\, (\theta_1\, \theta_2)(\theta_3\, \theta_4),
(\theta_1\, \theta_3)(\theta_2\, \theta_4)\, \rangle$ and
let $H$ denote the subgroup of $\Aut \Lambda$ generated by $s_\omega$,
$s_\psi$, $t_{(\omega\,\psi)}$, and $t_\sigma$ for $\sigma \in B$.
Note that every $\sigma \in B$ fixes $\omega$ and $\psi$. By
(\ref{action}) this implies that for every $h \in H$ we have
$h\cN_{\omega\psi}h^{-1} =
\cN_{\omega\psi}$, so $h$ sends orbits under $\cN_{\omega\psi}$ to
orbits under $\cN_{\omega\psi}$, i.e., $h$ permutes the $4$-gons in
$\cS$ and its complementary \gonquad. The elements
$s_\omega$, $s_\psi$, and $t_{(\omega\,\psi)}$ fix each of these $4$-gons.
Let $S \in \cS$ be the $4$-gon containing $L_0$. We have $t_\sigma
(L_0) = L_0$ for all $\sigma \in B$, so $h$ sends $S$ to $S$ for all $h \in
H$. Let $S' \in \cS$ be a different $4$-gon. Then by Lemma \ref{gonquads}
there are $\theta,\theta' \in \Omega\setminus \{\omega,\psi\}$ such
that $s_{\theta\theta'}(S) = S'$. For each $\sigma \in B$ we have
$$
t_\sigma(s_{\theta\theta'}(L_0))) =
s_{{}^\sigma\!(\theta\theta')}(t_\sigma(L_0)) =
s_{{}^\sigma\!(\theta\theta')}(L_0).
$$
For all $\sigma \in B$ we have $s_{\theta\theta'}S =
s_{{}^\sigma\!(\theta\theta')}S$, so $t_\sigma$ also fixes $S'$. We
conclude $H \subset G_{[\cS]}$. By Lemma \ref{GcS} the order of $G_{[\cS]}$
equals $768/4!=32=\#H$, so we have $H=G_{[\cS]}$.
\end{proof}

We can now find explicit generators of the field $k([\cS])$ in the
generic case.

\begin{lemma}\label{gen}
Consider the generic case of Example {\rm \ref{generic}}.
Let $\{\cS,\cS'\}$ be a \dgonquad, such that the $4$-gons of $\cS$ are orbits
under $\Phi_{\omega_5\omega_6}$, and
assume that $L_0$ is contained in one of the $4$-gons of $\cS$.
Set $N=\sqd_1\sqd_2\sqd_3\sqd_4$, $\alpha_1 =
\omega_1\omega_4+\omega_2\omega_3$, $\alpha_2 =
\omega_1\omega_3+\omega_2\omega_4$, $\alpha_3 =
\omega_1\omega_2+\omega_3\omega_4$, $\beta_1
=\sqd_1\sqd_4+\sqd_2\sqd_3$, $\beta_2
=\sqd_1\sqd_3+\sqd_2\sqd_4$, and $\beta_3
=\sqd_1\sqd_2+\sqd_3\sqd_4$. Then
$k_g' = k_g(\omega_5+\omega_6,\omega_5\omega_6,\alpha_1,\alpha_2,\alpha_3)$
is the unique $S_3$-extension of
$k_g(\{\cS,\cS'\}) = k_g(\omega_5+\omega_6,\omega_5\omega_6)$ contained in the
$S_4$-extension $k_g(\{\cS,\cS'\})(\omega_1,\omega_2,\omega_3,\omega_4)$.
Set $n_g = k_g'(N)$. Then the field
$k_g([\cS])$ equals $n_g(\beta_1,\beta_2,\beta_3)$ and is an
$S_4$-extension of $k_g(\cS)=k_g(\{\cS,\cS'\})(N)$.
Its unique $S_3$-subextension is $n_g$,
and its unique quadratic subextensions of $n_g$ are generated by the $\beta_i$.
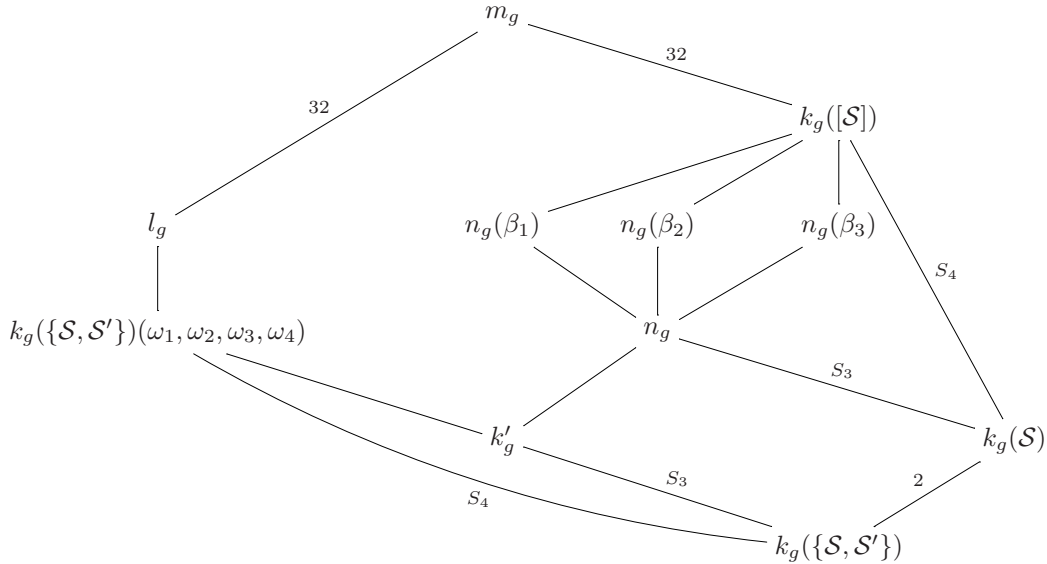
$$
\xymatrix{
&&m_g \ar@{-}[rrd]^{32}\ar@{-}[ddll]_{32}\\
&&&& k_g([\cS])\ar@{-}[rddd]^{S_4}\\
l_g \ar@{-}[d]&& n_g(\beta_1)\ar@{-}[urr]\ar@{-}[dr] & n_g(\beta_2)\ar@{-}[ur]
            \ar@{-}[d] &n_g(\beta_3) \ar@{-}[u]\ar@{-}[dl]\\
k_g(\{\cS,\cS'\})(\omega_1,\omega_2,\omega_3,\omega_4)\ar@{-}[rrd]
  \ar@/_6mm/@{-}[rrrrdd]_{S_4} & & & n_g \ar@{-}[ld]\ar@{-}[rrd]^{S_3}\\
& & k_g' \ar@{-}[rrd]^{S_3} & & & k_g(\cS)\ar@{-}[ld]_2 \\
&&&& k_g(\{\cS,\cS'\}) \\
}
$$
\end{lemma}
\begin{proof}
The first statement is elementary Galois theory. The field
$k_g([\cS])$ is the fixed field of the group
$\rho_g^{-1}(G_{[\cS]})$. By Lemma \ref{GcS} this field is an
$S_4$-extension of the fixed field $k_g(\cS)$ of $G_\cS$, which equals
$k_g(\omega_5+\omega_6,\omega_5\omega_6,N)=k_g(\{\cS,\cS'\})(N)$ by Lemma
\ref{gonquadfields}. Using Lemma \ref{KcS} one checks that the group
$\rho_g^{-1}(G_{[\cS]})$ acts trivially on
$n_g(\beta_1,\beta_2,\beta_3)$, so we conclude
$n_g(\beta_1,\beta_2,\beta_3)\subset k_g([\cS])$.
Since $k_g'$ and $k_g(\cS)$ intersect in $k_g(\{\cS,\cS'\})$, we
find that the compositum $n_g$ is an $S_3$-extension of $k_g(\cS)$,
and therefore the unique $S_3$-extension contained in $k_g([\cS])$.
By elementary Galois theory and group theory this implies that there
are three quadratic extensions of $n_g$ contained in $k_g([\cS])$.
Note that $\rho_g^{-1}(s_{\omega_1\omega_4})$ and
$\rho_g^{-1}(s_{\omega_1\omega_3})$ act trivially on $n_g(\beta_1)$ and
$n_g(\beta_2)$ respectively, but nontrivially on $\beta_2$ and
$\beta_1$ respectively. We conclude
that $\beta_1$ and $\beta_2$ generate two different quadratic
extensions of $n_g$. By symmetry, $\beta_3$ generates a third.
This implies $[n_g(\beta_1,\beta_2,\beta_3) : n_g]\geq
4 = \big[k_g([\cS]) : n_g\big]$, so we deduce that
$k_g([\cS]) = n_g(\beta_1,\beta_2,\beta_3)$.
\end{proof}

The generators of the field $k([\cS])$ in any other special
case follow immediately.

\begin{corollary}
Let the notation be as in Lemma \ref{KcS}. Set $N =
\sqrt{\delta_1\delta_2\delta_3\delta_4}$, $\alpha_1 =
\theta_1\theta_4+\theta_2\theta_3$, $\alpha_2 =
\theta_1\theta_3+\theta_2\theta_4$, $\alpha_3 =
\theta_1\theta_2+\theta_3\theta_4$, $\beta_1
=\sqrt{\delta_1\delta_4}+\sqrt{\delta_2\delta_3}$, $\beta_2
=\sqrt{\delta_1\delta_3}+\sqrt{\delta_2\delta_4}$, and $\beta_3
=\sqrt{\delta_1\delta_2}+\sqrt{\delta_3\delta_4}$. Then
$k_g([\cS])$ equals
$$
k(\omega+\psi,\omega\psi,N,\alpha_1,\alpha_2,\alpha_3,\beta_1,\beta_2,\beta_3).
$$
\end{corollary}
\begin{proof}
Since $k_g([\cS])$ is the fixed field of the group
$\rho^{-1}(G_{[\cS]})$, it follows from Lemma \ref{specialization} that
it suffices to do this in the generic case.  This is dealt with in
Lemma \ref{gen}.
\end{proof}

\subsection{The elliptic fibrations}\label{sectionellfibs}

Let $\cS=\{S_1,S_2,S_3,S_4\}$ and $\cS' =
\{S_5,S_6,S_7,S_8\}$ be complementary \gonquads.
By Lemma \ref{Vbarell} there is an elliptic
fibration $\Vbar \ra \P^1$ such that the $4$-gons in $\cS$ are some of
the fibers. For any $S \in \cS$ this fibration can be defined over the
field $k(S)$. It is possible, however, that none of the
fibers is defined over the field $k(\cS)$. This implies that the base
of the family of fibers is not isomorphic to $\P^1$ over
$k(\cS)$. As the base curve does become isomorphic to $\P^1$ over some
extension field, it is isomorphic to a conic. In this section we will
give explicit
equations, both for such a conic and for the fibration map in the generic
case of Example \ref{generic}. We will use the notation introduced in
that example. The equations for any special case follow by specialization.
Although the expressions involved become quite large, all computations 
in this section can still be checked by hand. We recommend, however, 
to check them with the {\sc magma} script provided \cite{elec}. 
We will first give the elliptic fibration over
over the field $k_g([\cS])$, over which the base curve can be taken to
be the projective line.

\subsubsection{A fibration over the projective line}

After renumbering the elements of $\Omega$, we
may assume that the $4$-gons of $\cS$ are orbits under
$\Phi_{\omega_5\omega_6}$. After applying an automorphism
that sends some of the $\epsilon_i$ to $-\epsilon_i$ (for notation,
see Example \ref{generic}), we may also assume that $L_0$ is contained in
one of the $4$-gons of $\cS$. We renumber $S_1,\ldots,S_4$ and
$S_5,\ldots,S_8$, so that
$S_1,\ldots, S_8$ contain the lines $L_{14}$, $L_{24}$, $L_{34}$,
$L_{0}$, $L_1$, $L_2$, $L_3$, and $L_4$ respectively. In particular
this means
$$
\begin{array}{ll}
S_1 = \{L_{14},L_{23},L_{145},L_{235}\}, &
S_5 = \{L_{1},L_{15},L_{16},L_{156}\}, \cr
S_2 = \{L_{13},L_{24},L_{135},L_{245}\}, &
S_6 = \{L_{2},L_{25},L_{26},L_{256}\}. \cr
\end{array}
$$

For notational convenience, we let $N$,
$\alpha_i$ and $\beta_i$ be as in Lemma \ref{gen}. We also set
$$
\begin{array}{lcll}
\gamma_1 = \alpha_3-\alpha_2, &\qquad& \Delta_4 = \prod_{1 \leq i < j
  \leq 4} (\omega_i - \omega_j), \cr
\gamma_2 = \alpha_1-\alpha_3, && \cr
\gamma_3 = \alpha_2-\alpha_1, && \kappa_j = \prod_{1\leq i \leq 4}^{i
  \neq j} (\omega_j -\omega_i),&\quad 1\leq j \leq 4,\cr
\eta = \sum_{i=1}^4 \sqd_i,&&  \cr
\multicolumn{4}{l}{
c_r = \mbox{elementary symmetric polynomial in the $\omega_j$ ($1\leq
  j \leq 4$) of degree $r$.}}
\end{array}
$$
Note that for $1\leq j \leq 4$ and
$J \subset \Omega$ we have $\varphi_j(\sqdelta_J) = \pm \sqd_j$,
where the sign is negative if and only if we have $j \in J$.
For the evaluation of various linear forms at the
intersection points of the lines in $\Lambda$, it will also be
convenient to notice that we have
\begin{equation}\label{usefulus}
\sum_{j=1}^4 \omega_j^r \kappa_j^{-1} =
\left\{\begin{array}{ll}
-c_4^{-1} & r =-1, \cr
0 &  r=0,1,2, \cr
1 & r=3. \cr
\end{array}
\right.
\end{equation}

It will turn out that the elliptic fibration associated to $\cS$
factors through the projection of $\P(\Abar)$ to $\P^3$ by the
coordinates $\varphi_i$ for $1\leq i \leq 4$, i.e., the projection
away from the line given by $\varphi_i=0$ for $1\leq i \leq 4$.
The image of $\Vbar$ under this projection is the nonsingular quadric
$D_{\omega_5\omega_6}$ of Remark \ref{coneovercone}. Note that $\Vbar$
is contained in the inverse image of $D_{\omega_5\omega_6}$ under the
indicated projection, which is the
cone over the cone over $D_{\omega_5\omega_6}$ in $\P(\Abar)$, given by
$Q=0$ with
\begin{equation}\label{B}
Q =\omega_5\omega_6Q_0-(\omega_5+\omega_6)Q_1+Q_2
   = \sum_{j=1}^4\kappa_j^{-1} \delta_j \varphi_j^2,
\end{equation}
as was pointed out in Remark \ref{coneovercone}. Consider the linear forms
$$
\begin{array}{lcl}
l_1 = \sum_{j=1}^4 \kappa_j^{-1}\sqd_j\varphi_j, &\quad& m_1=\sum_{j=1}^4
            \omega_j(2\omega_j-c_1)\kappa_j^{-1}\sqd_j\varphi_j,\cr
l_2=\sum_{j=1}^4\omega_j\kappa_j^{-1}\sqd_j\varphi_j,&\quad&m_2=\sum_{j=1}^4
            (2c_4\omega_j^{-1}-c_3)\kappa_j^{-1}\sqd_j\varphi_j.\cr
\end{array}
$$

\begin{lemma}\label{lm}
On $\Vbar$ we have $l_1m_2 = l_2m_1$. The map $\chi\colon \,
\Vbar \ra \P^1$ that sends $x$ to $[l_1(x):m_1(x)]=[l_2(x):m_2(x)]$ is
an elliptic fibration, defined over $k_g([\cS])$. The
$4$-gons $S_1,S_2,S_3,S_4$ are fibers above $[-1:\alpha_1]$,
$[-1:\alpha_2]$, $[-1:\alpha_3]$, and $[0:1]$ respectively.
\end{lemma}
\begin{proof}
From (\ref{usefulus}) one easily works out that $m_1l_2-l_1m_2 =
Q$, so on $\Vbar$ we find $l_1m_2 = l_2m_1$. The four equations
$l_1=m_1=l_2=m_2=0$ are linearly independent, so the base locus of
the map $\chi$ is given by $\varphi_i=0$ on $\Vbar$, for $1\leq i
\leq 4$. Together with the equations $Q_0=Q_1=Q_2=0$ (see
Proposition \ref{threequad}), this implies that the base locus of
$\chi$ is empty. The fiber $F_0$ of $\chi$ above $[a:b]$ is the
intersection of $\Vbar$ with the three-space given by $bl_i=am_i$
for $i=1,2$. The quadric $Q$ vanishes on this three-space, in which
the fiber $F_0$ is therefore given by $Q_0=Q_1=0$. Since every
smooth intersection of two quadrics in $\P^3$ is a curve of genus
$1$, we deduce that $\chi$ is an elliptic fibration, whose fibers
all have degree $4$. The intersection points of the lines in $S_4$
are $P_{\sqdelta,\omega_r}$ with $\sqdelta\in \{\sqdelta_0,
\sqdelta_{56}\}$, and $r \in \{5,6\}$. From (\ref{usefulus}) and
the identity
\begin{equation}\label{eval}
\sqd_j\varphi_j\big(\sqdelta_I^{-1}(X-\omega_r)\big) =
\pm (\omega_j-\omega_r),
\end{equation}
where the sign is positive if and only if $j \not\in I$,
we find that the $l_i$ vanish on these points, and thus on the
lines in $S_4$. This implies that $S_4$ is contained in the fiber
above $[0:1]$. Since all fibers have degree $4$, the union of the
lines in $S_4$ is a whole fiber. The lines in $S_1\cup S_2 \cup
S_3$ do not intersect any line in $S_4$, so they are fibral as
well, which implies that all $S\in \cS$ are fibers of $\chi$.
Their images are easily computed by evaluating the $l_i/m_i$ on
the appropriate intersection points $P_{\sqdelta,\omega_5}$ of two
lines in the $4$-gons, using (\ref{eval}) and perhaps a computer
algebra package to verify that for instance the ratio
$l_1(P_{\sqdelta_{24},\omega_5}) :m_1(P_{\sqdelta_{24},\omega_5})$,
which is the ratio between
$$
\frac{\omega_1-\omega_5}{\kappa_1}-
\frac{\omega_2-\omega_5}{\kappa_2}+
\frac{\omega_3-\omega_5}{\kappa_3}-
\frac{\omega_4-\omega_5}{\kappa_4}
$$
and
$$
\frac{\omega_1(2\omega_1-c_1)(\omega_1-\omega_5)}{\kappa_1}-
\frac{\omega_2(2\omega_2-c_1)(\omega_2-\omega_5)}{\kappa_2}+
\frac{\omega_3(2\omega_3-c_1)(\omega_3-\omega_5)}{\kappa_3}-
\frac{\omega_4(2\omega_4-c_1)(\omega_4-\omega_5)}{\kappa_4},
$$
does indeed equal $-1:\alpha_2$ (see \cite{elec}).
Since $\chi$ is also given by $[\eta
l_i:\eta m_i]$, and the polynomials $\eta l_i$ and $\eta m_i$ are
fixed by $\rho_g^{-1}(G_{[\cS]})$, we find that $\chi$ is defined
over $k_g([\cS])$.
\end{proof}

\begin{remark}
The map $\chi$ of Lemma \ref{lm} is in fact defined over $k_g(S_1)$.
We found the linear forms $l_i$ and $m_i$ as follows. Using simple
linear algebra we found linear forms $h_1,h_2,h_3,h_4$ vanishing
on the lines in $S_4 \cup S_5$, $S_4\cup S_6$, $S_3 \cup S_5$, and
$S_3 \cup S_6$ respectively. The elliptic fibrations given by $[h_1:h_2]$
and $[h_3:h_4]$ both have the $4$-gons in $\cS'$ as fibers, so they
differ by an automorphism of the base, which fixes the points $[0:1]$
and $[1:0]$ as both fibrations have the same fibers $S_5$ and $S_6$
there. This implies that after the appropriate scaling of the $h_i$ we
may assume $h_1h_4 = h_2h_3$. The space of linear forms vanishing on
$S_4$ is spanned by $h_1$ and $h_2$. We can pick a
$k(S_4)$-basis $l_1' = ah_1+bh_2$ and $l_2' = ch_1+dh_2$ for some
$a,b,c,d\in m_g$. The fibers $F_1$ and $F_2$ of the fibration
$[h_1:h_2]=[h_3:h_4]$ above the points $[-b:a]$ and $[-d:c]$
respectively are then defined over $k(S_4)$, as they are the
complement of $S_4$ inside the hyperplane section given by $l_1'$ and
$l_2'$ respectively. The space of linear forms
vanishing on $F_1$ is spanned by $l_1'$ and $m_1''=ah_3+bh_4$.
For some $p,q$, the form $m_1'=pl_1'+qm_1''=pl_1'+aqh_3+bqh_4$ is also
defined over $k(S_4)$. Set $m_2' = pl_2'+cqh_3+dqh_4$. Then from
$h_1h_4 = h_2h_3$ we also find $l_1'm_2'=l_2'm_1'$ on $\Vbar$.
For some choice of $a,b,c,d,p,q$,
the given $l_i$ and $m_i$ satisfy $l_i' = \eta l_i$, and $m_i' = \eta m_i$.
\end{remark}

\subsubsection{The fibration over a conic}

Let $\tau \in \Gal(m_g/k_g)$ denote the automorphism that fixes all the
$\omega_j$ and the $\sqd_j$ for $j \geq 3$, and sends $\sqd_i$ to
$-\sqd_i$ for $i=1,2$. Then $\tau$ induces the nontrivial automorphism
of the quadratic extension $k_g([\cS])/n_g(\beta_3)$,
generated by $\beta_2$. Since $\tau$ permutes the $4$-gons in $\cS$,
the elliptic fibration ${}^\tau\!\chi$ differs from $\chi$ by some
automorphism $\psi$ of the base curve $\P^1$ by Proposition
\ref{K3ell}, i.e., we have ${}^\tau\!\chi
= \psi \circ \chi$. This implies that the image of the map
$(\chi,{}^\tau\!\chi)\colon \Vbar \ra \P^1\times \P^1$ is contained in the
graph of $\psi$. Under the Segre embedding $\P^1\times \P^1 \ra \P^3$
this graph maps to a conic that we can embed in $\P^2$. We will now
make this explicit. For $i=1,2$ we set
$$
p_i = 2({}^\tau\!l_i)l_i, \qquad q_i =
({}^\tau\!l_i)m_i+({}^\tau\!m_i)l_i+2\alpha_3 l_i{}^\tau\!l_i, \qquad
r_i=\beta_2^{-1}(
({}^\tau\!m_i)l_i-({}^\tau\!l_i)m_i).
$$
Let $C_1\subset \P^2$ be the conic given by $\gamma_1\gamma_2p^2 +
q^2 =\beta_2^2r^2$.

\begin{lemma}\label{pqr}
There is an elliptic fibration $\nu_1\colon \,\Vbar \ra C_1$, defined over
$n_g(\beta_3)$, given by $x \mapsto [p_i(x):q_i(x):r_i(x)]$ for
$i=1,2$, such that the $4$-gons $S_1,S_2,S_3,S_4$ are fibers above
$[2:\gamma_1-\gamma_2:-\gamma_3\beta_2^{-1}]$,
$[2:\gamma_1-\gamma_2:\gamma_3\beta_2^{-1}]$,
$[0:\beta_2:1]$, and $[0:\beta_2:-1]$ respectively.
\end{lemma}
\begin{proof}
Note that the images of the $S_i$ under $\chi$, given in Lemma
\ref{lm}, are $\tau$-invariant. This implies $({}^\tau \!\chi)(S_i) =
{}^\tau\!(\chi({}^{\tau^{-1}}\! S_i)) = \chi({}^\tau\! S_i)$. Note
also that $\tau$ acts on the $S_i$ as the permutation $(S_1\, S_2)(S_3\, S_4)$.
By Proposition \ref{K3ell} the elliptic fibrations ${}^\tau \!\chi$
and $\chi$ differ by an automorphism of $\P^1$.
As an automorphism of $\P^1$ is determined by its action on
the $\chi(S_i)$, this
allows us to check that the automorphism $\psi \colon\, [s:t] \mapsto
[-\alpha_3s-t:(\alpha_3^2+\gamma_1\gamma_2)s+\alpha_3t]$ of $\P^1$
satisfies ${}^\tau\!\chi = \psi \chi$. By Lemma \ref{lm}
it suffices to check that $\psi$ switches the points $[-1:\alpha_1]$
and $[-1:\alpha_2]$ and also the points $[-1:\alpha_3]$ and
$[0:1]$ (see \cite{elec}). We conclude that
$(\chi,{}^\tau\!\chi) \colon \, \Vbar \ra \P^1\times \P^1$ is an
elliptic fibration over the graph of $\psi$.

Let $h \colon \,\P^1\times \P^1 \ra \P^3$ denote the modified Segre
embedding that sends $([a:b],[c:d])$ to
$[x:y:z:w]=[ac:ad+bc:\beta_2^{-1}(ad-bc):bd]$.  Then the composition
$g = h \circ(\chi,\chi^\tau) \colon \, \Vbar \ra \P^3$ is
$\tau$-invariant, so it is defined over $n_g(\beta_3)$. The image of
the graph of $\psi$ under $h$ is the conic given by $y^2-\beta_2^2z^2=4xw$ and
$(\alpha_3^2+\gamma_1\gamma_2)x+\alpha_3y+w=0$. The image of this conic
under the projection $\pi \colon \, \P^3 \ra \P^2 , [x,y,z,w] \mapsto
[2x:y+2\alpha_3x:z]$ is $C_1$, so the composition $\nu_1 = \pi \circ g$ is an
elliptic fibration of $\Vbar$ over $C_1$. Since $\pi$ and $g$ are defined over
$n_g(\beta_3)$, so is $\nu_1$. As $\chi$ is given by
$[l_i:m_i]$, for $i=1,2$, one checks easily that
the fibration $\nu_1$ is given by $[p_i:q_i:r_i]$. The images of
the fibers are easily computed using the images given in Lemma \ref{lm}
and the fact that we have $\big({}^\tau(l_i/m_i)\big)(S_i) =
(l_i/m_i)\big({}^\tau S_i\big)$ as noted above.
\end{proof}

We will construct an automorphism $\psi$ of $\P^2$ such that $\psi
\circ \nu_1$ is an elliptic fibration from $\Vbar$ to a conic $C$, such that
both $C$ and the fibration are defined over $k_g(\cS)$, the field of
definition of the fibration. We know that there is an
elliptic fibration over a conic defined over $k_g(\cS)$
whose fibers include the $4$-gons of $\cS$.
By Proposition \ref{K3ell} it is unique up to an isomorphism of
the conic, so we know such a $\psi$ exists. We will
do this in two steps by first descending to $n_g$ and then to $k_g(\cS)$.
Suppose at some step we have a fibration $\nu_i\colon \, \Vbar \ra
C_i$, with $C_i$ a conic, defined over a field $K_i$ that is
Galois over $K_{i+1}$ with Galois group $G_i$. We
are looking for an automorphism $\psi_i$ of $\P^2$ such that
$\nu_{i+1} =\psi_i \circ \nu_i \colon \, \Vbar \ra C_{i+1}$ with
$C_{i+1} = \psi_i(C_i)$ is defined over $K_{i+1}$.

For all $g\in G_i$ there is an isomorphism $\sigma(g) \colon \,
{}^g\! C_i \ra C_i$ such that $\nu_i = \sigma(g) \circ {}^g\!\nu_i$. Since
$\sigma(g)^*\colon \Pic C_i \ra \Pic {}^g\!C_i$
sends the canonical divisor of $C_i$ to that of
${}^g\!C_i$, the automorphism $\sigma(g)$ is induced by a unique
automorphism of $\P^2$, which
we will also denote by $\sigma(g)$. These automorphisms satisfy the
cocycle condition $\sigma(hg) = \sigma(h)\circ {}^h\!\sigma(g)$.
The automorphism $\psi_i$ that we seek satisfies $\psi_i \circ
\sigma(g) = {}^g\!\psi_i$, so $\sigma$ is a coboundary with values in
$\Aut \P^2_{K_i}$.
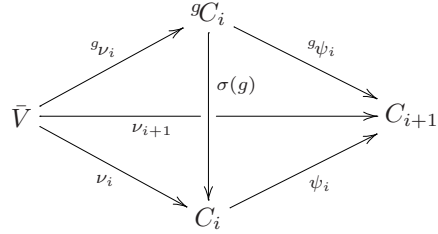
$$
\xymatrix{
&& {}^g\!C_i \ar[rrd]^{{}^g\!\psi_i} \ar[dd]^(0.35){\sigma(g)} \\
\Vbar\ar[rru]^{{}^g\!\nu_i}\ar[rrd]_{\nu_i}\ar'[rr]_(0.7){\nu_{i+1}}[rrrr]&&&&C_{i+1}\\
&&C_i\ar[rru]_{\psi_i}\\
}
$$
We can find $\psi_i$ as follows. Consider the homomorphism
$\GL_3(K_i) \ra \Aut \P^2_{K_i}$ that maps a matrix $M\in \GL_3(K_i)$
to the automorphism that sends $[x:y:z]$ to $[x':y':z']$ with
$(x'\,y'\,z')^t=M(x\,y\,z)^t$. Through this homomorphism we may
identify $\Aut \P^2_{K_i}$ with $\PGL_3(K_i)$.

Our first step will be to lift $\sigma$ to a cocycle for $\GL_3(K_i)$.
The map $\det \colon\, \GL_3(K_i)\ra K_i^*,\,\,M \mapsto \det(M)$
induces a homomorphism $\PGL_3(K_i) \ra K_i^*/{K_i^*}^3$.
For any $G_i$-set $M$, let $Z^1(M)$ denote the corresponding set of
$1$-cocycles with coefficients in $M$. We obtain the following
diagram.
$$
\xymatrix{
&1\ar[d] &&1\ar[d] && 1 \ar[d]\cr
1 \ar[r]&Z^1(\mu_3(K_i))\ar[d]\ar[rr] && Z^1(\SL_3(K_i))\ar[d]\ar[rr] &&
             Z^1(\PSL_3(K_i))\ar[d] \ar[rr]&&H^2(\mu_3(K_i))\ar[d]\cr
1\ar[r]&Z^1(K_i^*)\ar[d]_{[3]}\ar[rr]&& Z^1(\GL_3(K_i))\ar[d]_{\det}\ar[rr]&&
             Z^1(\PGL_3(K_i))\ar[d]_{\det} \ar[rr]&&H^2(K_i^*)\ar[d]\cr
1 \ar[r]&Z^1({K_i^*}^3)\ar[rr] && Z^1(K_i^*)\ar[rr]
            && Z^1(K_i^*/{K_i^*}^3) \ar[rr]&&H^2({K_i^*}^3)\cr
}
$$
Since $\sigma\in Z^1(\PGL_3(K_i))$ is in fact a coboundary, it maps to zero
in $H^1(\PGL_3(K_i))$, so it also maps to zero in $H^2(K_i^*)$. Therefore
we can lift $\sigma$ to an element $\sigma' \in Z^1(\GL_3(K_i))$.
To do this in practice we
note that we are working over a generic field, so we may assume $K_i$
has no cube roots of unity. This implies that $\SL_3(K_i)$ is
isomorphic to the subgroup $\PSL_3(K_i)$ of $\PGL_3(K_i)$, so an
element $M \in \GL_3(K_i)$ is uniquely determined by its determinant
and its image in $\PGL_3(K_i)$. Thus $\sigma' \in
Z^1(\GL_3(K_i))$ is uniquely determined by its image $\sigma$ and $\det
\sigma' \in Z^1(K_i^*)$, which is a lift of $\det \sigma \in
Z^1(K_i^*/{K_i^*}^3)$. To find $\sigma'$ it therefore suffices to find
a lift of $\det\sigma \in Z^1(K_i^*/{K_i^*}^3)$ to $Z^1(K_i^*)$. In our case
$\det \sigma$ will always be trivial, so this step is easy.

The second step is to write $\sigma' = Z^1(\GL_3(K_i))$ as a
coboundary, which is possible as we have $H^1(\GL_3(K_i))=1$ by a
generalization of Hilbert 90. For this we use a standard trick,
discussed in \cite{ser}, Proposition X.3. For any matrix $M$ we set
$$
b_M = \sum_{g \in G} \sigma'(g)({}^g\!M).
$$
One checks that we have $\sigma'(h) \cdot {}^h b_M = b_M$. Choosing $M$
carefully so that $b_M$ is invertible, we may take $\psi_i$ to be the
automorphism associated to the matrix $b_M^{-1}$. After some
simplifications of these automorphisms we get the following lemma,
which can in fact be checked without knowing how we obtained the
equations. Let $C_2\subset \P^2(u,v,w)$ be the conic given by
$\gamma_1\beta_1^2u^2+\gamma_2\beta_2^2v^2+ \gamma_3\beta_3^2w^2$ and set
$$
\eqalign{
u_i & = \beta_1^{-1}(q_i+\gamma_2p_i) =
 \beta_1^{-1}({}^\tau\!l_im_i+{}^\tau\!m_il_i+2\alpha_1l_i{}^\tau\!l_i),\cr
v_i & = \beta_2^{-1}(q_i-\gamma_1p_i)
 = \beta_2^{-1}({}^\tau\!l_im_i+{}^\tau\!m_il_i+2\alpha_2l_i{}^\tau\!l_i),\cr
w_i & = -\beta_2\beta_3^{-1}r_i
    = \beta_3^{-1}({}^\tau\!l_im_i-{}^\tau\!m_il_i).\cr
}
$$

\begin{lemma}\label{uvw}
There is an elliptic fibration $\nu_2\colon \,\Vbar \ra C_2$, defined over
$n_g$, given by $x \mapsto [u_i(x):v_i(x):w_i(x)]$ for
$i=1,2$, such that the $4$-gons $S_1,S_2,S_3,S_4$ are fibers above
$[-\beta_1^{-1} :  \beta_2^{-1} :  \beta_3^{-1}]$,
$[ \beta_1^{-1} : -\beta_2^{-1} :  \beta_3^{-1}]$,
$[ \beta_1^{-1} :  \beta_2^{-1} : -\beta_3^{-1}]$, and
$[\beta_1^{-1} : \beta_2^{-1} : \beta_3^{-1}]$ respectively.
\end{lemma}
\begin{proof}
The nontrivial automorphism $\pi$ of the extension $n_g(\beta_3)/n_g$
is induced by the automorphism that fixes all $\omega_j$ and all
$\sqd_j$, except for $\sqd_1$ and $\sqd_3$, which are sent to their
negatives. It acts on the $S_i$ as the permutation $(S_1\,\,S_3)(S_2\,\,S_4)$.
The composition $\nu_2\colon \Vbar \ra \P^2$ of $\nu_1$ and
$$
\psi\colon \, \P^2 \ra \P^2,\,\, [p:q:r]\mapsto [\beta_1^{-1}(q+\gamma_2 p)
  : \beta_2^{-1}(q-\gamma_1 p) : -\beta_2\beta_3^{-1}r]
$$
is given by $x \mapsto [u_i(x):v_i(x):w_i(x)]$. The inverse of $\psi$
is given by
$$
[u:v:w] \mapsto [ \beta_1u-\beta_2v:
\gamma_1\beta_1u+ \gamma_2\beta_2v: \gamma_3\beta_2^{-1}\beta_3w  ].
$$
Substituting this into the equation for $C_1$, we find that the conic
$\psi(C_1)$ is equal to $C_2$. Alternatively, one checks that we
have
$$
\gamma_1\beta_1^2u_i^2+\gamma_2\beta_2^2v_i^2+ \gamma_3\beta_3^2w_i^2 =
4\gamma_3b_i{}^\tau\!l_i l_iQ,
$$
with $b_1 = \omega_1+\omega_2-\omega_3-\omega_4$ and  $b_2 =
-c_4(\omega_1^{-1}+\omega_2^{-1}-\omega_3^{-1}-\omega_4^{-1})$ (see
\cite{elec}).
Since $Q$ vanishes on $V$, this also shows that $\nu_2$ maps $\Vbar$
to $C_2$. As $\pi$ permutes $\cS$, the $4$-gons in $\cS$
are also fibers of ${}^\pi\!\nu_2$, so $\nu_2$ and ${}^\pi\!\nu_2$
differ by an isomorphism on the base by Proposition \ref{K3ell}.
Therefore, there is a unique isomorphism $h\colon \,C_2 \ra
{}^\pi\!C_2=C_2$ such that ${}^\pi\!\nu_2 = h \circ \nu_2$. Since $h$
fixes the anticanonical divisor on $C_2$, which determines the
embedding of $C_2$ in $\P^2$, the isomorphism $h$ comes from a unique
automorphism of $\P^2$, which we will also denote by $h$.
With the points 
$\nu_1(S_i)$ given in Lemma \ref{pqr}, one easily computes
$\nu_2(S_i) = \psi(\nu_1(S_i))$ to be as claimed. With the identity
$({}^\pi\!\nu_2)(S_i) = {}^\pi\!(\nu_2({}^{\pi^{-1}}S_i))$ we check that
we have $({}^\pi\!\nu_2)(S_i) = \nu_2(S_i)$ for $1\leq i \leq 4$, so
$h$ fixes the four points $\nu_2(S_i)$. As these points all lie on $C_2$,
no three of them are collinear. This implies that the action of $h$ on
the four points determines $h$ uniquely, which means that $h$ is the
identity, so ${}^\pi\!\nu_2 =\nu_2$. We conclude that $\nu_2$ is
defined over the fixed field $n_g$ of $\pi$.
\end{proof}

Finally, we let $C_3 \subset \P^2(x,y,z)$ be the conic given by
$$
\gamma_1\beta_1^2(x+\gamma_1\Delta_4y+(\alpha_2+\alpha_3)z)^2+
\gamma_2\beta_2^2(x+\gamma_2\Delta_4y+(\alpha_1+\alpha_3)z)^2+
\gamma_3\beta_3^2(x+\gamma_3\Delta_4y+(\alpha_1+\alpha_2)z)^2=0,
$$
and we set
$$
\eqalign{
x_i&=2\Delta_4\big((\alpha_2\alpha_3-\alpha_1^2)u_i+
(\alpha_1\alpha_3-\alpha_2^2)v_i+(\alpha_1\alpha_2-\alpha_3^2)w_i \big),\cr
y_i&=-\gamma_1u_i-\gamma_2v_i-\gamma_3w_i, \cr
z_i&=\Delta_4\big((\gamma_2-\gamma_3)u_i+ (\gamma_3-\gamma_1)v_i+
(\gamma_1-\gamma_2)w_i\big).\cr
}
$$
\begin{lemma}\label{xyz}
There is an elliptic fibration $\nu_3\colon \,\Vbar \ra C_3$, defined over
$k_g(\cS)$, given by $x \mapsto [x_i(x):y_i(x):z_i(x)]$ for
$i=1,2$, such that the $4$-gons of $\cS$ are fibers of $\nu_3$.
\end{lemma}
\begin{proof}
Let $\pi_1,\pi_2 \in \Gal(n_g/k_g(\cS))$ denote the automorphisms induced by the
permutations $(\omega_1 \,\,\omega_2 \,\,\omega_3)$ and
$(\omega_1 \,\,\omega_2)$ on the $\omega_j$ respectively
and the corresponding permutation on the $\sqd_j$. Then $\pi_1$ and
$\pi_2$ induce generators of $\Gal(n_g/k_g(\cS))$. They act on $\cS$,
the $\alpha_j,\beta_j,\gamma_j$, and the $\phi_j$
as the permutations $(1\,\,2\,\,3)$ and $(1\,\,2)$ on the indices,
except that $\pi_2$ also negates the $\gamma_j$. Note also that we have
$\pi_1(\Delta_4) = \Delta_4$ and $\pi_2(\Delta_4) = -\Delta_4$.
The composition $\nu_3\colon \Vbar \ra \P^2$ of $\nu_2$ and
$\psi\colon \, \P^2 \ra \P^2,\,\, [u:v:w]\mapsto [x:y:z]$ with
$$
\eqalign{
x&=2\Delta_4\big((\alpha_2\alpha_3-\alpha_1^2)u+
(\alpha_1\alpha_3-\alpha_2^2)v+(\alpha_1\alpha_2-\alpha_3^2)w \big),\cr
y&=-\gamma_1u-\gamma_2v-\gamma_3w, \cr
z&=\Delta_4\big((\gamma_2-\gamma_3)u+ (\gamma_3-\gamma_1)v+
(\gamma_1-\gamma_2)w\big).\cr
}
$$
is given by $P \mapsto [x_i(P):y_i(P):z_i(P)]$. The inverse of $\psi$
is given by
$$
[x:y:z] \mapsto [x+\gamma_1\Delta_4y+(\alpha_2+\alpha_3)z:
  x+\gamma_2\Delta_4y+(\alpha_1+\alpha_3)z:
  x+\gamma_3\Delta_4y+(\alpha_1+\alpha_2)z].
$$
Substituting this in the equation for $C_2$, we find that the conic
$\psi(C_2)$ is equal to $C_3$, so $\nu_3$ maps $\Vbar$ to $C_3$.
Note that for all $g \in \Gal(n_g/k_g(\cS))$ we have ${}^g\!C_3 = C_3$.
As in the
proof of Lemma \ref{uvw}, for all $g \in \Gal(n_g/k_g(\cS))$ there is an
automorphism $h_g$ of $\P^2$ such
that we have ${}^g\!\nu_3 = h_g \circ \nu_3$. Evaluating $\psi$ at
the points $\nu_2(S_i)$ given in Lemma \ref{uvw} and using the
identity $({}^g\!\nu_3)(S_i) =
{}^g\!(\psi(\nu_2({}^{g^{-1}}S_i)))$, we check that for all
$g\in \Gal(n_g/k_g(\cS))$ we have $({}^g\!\nu_3)(S_i)=\nu_3(S_i)$ for
$1\leq i \leq 4$. It suffices to check this for $g=\pi_1,\pi_2$.
As in the proof of Lemma \ref{uvw} this implies that
for all $g\in \Gal(n_g/k_g(\cS))$ the automorphism $h_g$ is the
identity, so ${}^g\!\nu_3 = \nu_3$. We conclude that $\nu_3$
is defined over $k_g(\cS)$.
\end{proof}

\begin{remark}\label{polyinv}
Even though the fibration $\nu_3$ is defined over $k_g(\cS)$, the
polynomials $x_i,y_i,z_i$ that describe $\nu_3$ are not, and neither is
the defining equation of $C_3$. The latter issue is easily resolved by
multiplying the given equation for $C_3$ by $\Delta_4$ to obtain an
equation defined over
$k(\cS)$. To settle the former, we can replace $x_i,y_i,z_i$ with
Galois-invariant polynomials as follows. Again we
descend from the field $k_g([\cS])$ to $k_g(\cS)$ in
steps. The polynomials $l_i$ and $m_i$ are not defined over
$k_g([\cS])$, but $\eta l_i$ and $\eta m_i$ are, as they
are fixed by the elements of $G_{[\cS]}$, see Lemma \ref{KcS}.
Scaling by $\beta_2{}^\tau\!\eta\eta$, we find that the polynomials
$u_i'=\beta_2{}^\tau\!\eta\eta u_i$,
$v_i'=\beta_2{}^\tau\!\eta\eta v_i$,
and $w_i'=\beta_2{}^\tau\!\eta\eta w_i$ are also defined
over $n_g(\beta_3)$ and define $\nu_2$ as well.
Since $\nu_2$ is defined over $n_g$, the polynomials ${}^\pi\!u_i',
{}^\pi\!v_i'$, and ${}^\pi\!w_i',$ define $\nu_2$ as well,
where $\pi$ is the nontrivial
automorphism of $n_g(\beta_3)/n_g$ as in the proof of Lemma
\ref{uvw}. This implies that $\nu_2$ can
also be defined by $[u_i'':v_i'':w_i'']$ with
$u_i''=u_i'+{}^\pi\!u_i'$, $v_i''=v_i'+{}^\pi\!v_i'$, and
$w_i''=w_i'+{}^\pi\!w_i'$ unless these polynomials vanish on $V$, which
they turn out not to. One checks for instance that we have
$$
\eqalign{
\frac{\beta_1u_1''}{\beta_2} = &8(\sqd_1\sqd_2-\sqd_3\sqd_4)
  \left(\frac{\kappa_1^{-1}\delta_1\varphi_1^2+\kappa_4^{-1}\delta_4\varphi_4^2}{\omega_1-\omega_4}
  +\frac{\kappa_2^{-1}\delta_2\varphi_2^2+\kappa_3^{-1}\delta_3\varphi_3^2}{\omega_2-\omega_3}\right)\cr
  &+4(\delta_1+\delta_2-\delta_3-\delta_4)(\omega_1-\omega_2-\omega_3+\omega_4)
  \left((\omega_1-\omega_2)\kappa_1^{-1}\kappa_2^{-1}\sqd_1\sqd_2\varphi_1\varphi_2+
        (\omega_3-\omega_4)\kappa_3^{-1}\kappa_4^{-1}\sqd_3\sqd_4\varphi_3\varphi_4
  \right)
}
$$
(see \cite{elec}). With the same trick we descend to
$k_g(\cS)$. The polynomials
$$
\eqalign{
x_i'&=2\Delta_4\big((\alpha_2\alpha_3-\alpha_1^2)u_i''+
(\alpha_1\alpha_3-\alpha_2^2)v_i''+(\alpha_1\alpha_2-\alpha_3^2)w_i'' \big),\cr
y_i'&=-\gamma_1u_i''-\gamma_2v_i''-\gamma_3w_i'', \cr
z_i'&=\Delta_4\big((\gamma_2-\gamma_3)u_i''+ (\gamma_3-\gamma_1)v_i''+
(\gamma_1-\gamma_2)w_i''\big).\cr
}
$$
are defined over $n_g$ and $\nu_3$ is defined by
$[x_i':y_i':z_i']$. Since $\nu_3$ is defined over $k_g(\cS)$, it is
also defined by $[{}^g\!x_i':{}^g\!y_i':{}^g\!z_i']$ for any
$g\in \Gal(n_g/k_g(\cS))$ and therefore by
$[x_i'':y_i'':z_i'']$ with
$x_i'' = \sum_{g\in \Gal(n_g/k_g(\cS))} {}^g\!x_i'$,
$y_i'' = \sum_{g\in \Gal(n_g/k_g(\cS))} {}^g\!y_i'$, and
$z_i'' = \sum_{g\in \Gal(n_g/k_g(\cS))} {}^g\!z_i'$, unless these
polynomials, defined over $k(\cS)$, vanish on $V$. An explicit
computation shows that they do not. Over $k_g(\cS)$ we can
write $f=f_2f_4$ with $f_4 = X^4-c_1X^3+c_2X^2-c_3X+c_4$, where the
$c_r$ are the symmetric polynomials in the $\omega_j$ ($1\leq j \leq
4$) of degree $r$. The image in $A_4=k_g[X]/f_4$ of $\delta \in A$ can be
written as $\delta'=d_3X^3+d_2X^2+d_1X+d_0$ with $d_i \in
k_g(c_1,c_2,c_3,c_4)$. We may also consider the $d_i$ to be
independent transcendentals contained in $k_g$ (here the $d_i$
are not the same as in Example \ref{generic}).
Let $N$ denote a square root of the norm
$N_{A_4/k_g}(\delta')$ of $\delta'$. Note that as always we have
$\varphi_j = \sum_{i=0}^5 \omega_j^i a_i$, where the $a_i$ are
the original coordinates of $\P(A)$. The polynomials
$x_i'',y_i''$, and $z_i''$ are quadratic in the $a_i$ with
coefficients in the field $\Q(c_1,c_2,c_3,c_4,d_0,d_1,d_2,d_3,N)$.
Expressing them as such takes a file that has size about one
megabyte \cite{elec}.
Note also that the equations for the fibration do not
depend on $\omega_5$
and $\omega_6$. This shows that the fibration does indeed factor
through the quadric $D_{\omega_5\omega_6}$ in $\P^3$ of
Remark \ref{coneovercone}, because that is the image of $\Vbar$
under the projection from $\P(\Abar)$ to $\P^3$ using only the
coordinates $\varphi_1,\ldots,\varphi_4$.
\end{remark}

\begin{remark}
In any specific example, we can consider the specialization of
the equations for $C_3$ and the fibration $\nu_3$
in Lemma \ref{xyz}, or better, Remark \ref{polyinv}. For a proper
closed subset in the family of all curves of genus $2$ and choices of
$\delta$ these equations may vanish. Outside this subset, this
specialization gives us an elliptic
fibration $\nu$ of a surface $V$ over a conic $C$.
If $V$ is everywhere locally solvable,
then so is $C$. Since $C$ satisfies the Hasse principle, this implies
that $C$ has a rational point, which can be found by standard
algorithms. This gives us a rational fiber $F_0$ on $V$. The linear system
of hyperplanes through $F_0$ is $2$-dimensional and determines an
elliptic fibration $\nu'$ of $V$ over $\P^1$, given by linear
polynomials. The $4$-gons that are fibers of $\nu'$ belong to the
complementary \gonquad{} of the \gonquad{} whose $4$-gons are fibers
of $\nu$. Applying the same trick again, we also obtain an elliptic
fibration over $\P^1$ that is equivalent to $\nu$.  On the other hand,
even if $C$ is everywhere locally solvable,
$V$ may not be.
\end{remark}

\include{new10}

\end{document}

%% file: allints.pstex_t
\begin{picture}(0,0)%
\includegraphics{allints.pstex}%
\end{picture}%
\setlength{\unitlength}{3947sp}%
\begingroup\makeatletter\ifx\SetFigFont\undefined%
\gdef\SetFigFont#1#2#3#4#5{%
  \reset@font\fontsize{#1}{#2pt}%
  \fontfamily{#3}\fontseries{#4}\fontshape{#5}%
  \selectfont}%
\fi\endgroup%
\begin{picture}(9924,6858)(1864,-8182)
\put(2296,-3023){\makebox(0,0)[lb]{\smash{\SetFigFont{12}{14.4}{\rmdefault}{\mddefault}{\updefault}{\color[rgb]{0,0,0}$\omega 12$}%
}}}
\put(2281,-3909){\makebox(0,0)[lb]{\smash{\SetFigFont{12}{14.4}{\rmdefault}{\mddefault}{\updefault}{\color[rgb]{0,0,0}$\omega 34$}%
}}}
\put(3946,-4224){\makebox(0,0)[lb]{\smash{\SetFigFont{12}{14.4}{\rmdefault}{\mddefault}{\updefault}{\color[rgb]{0,0,0}$\omega 13$}%
}}}
\put(3939,-5101){\makebox(0,0)[lb]{\smash{\SetFigFont{12}{14.4}{\rmdefault}{\mddefault}{\updefault}{\color[rgb]{0,0,0}$\omega 24$}%
}}}
\put(4029,-2724){\makebox(0,0)[lb]{\smash{\SetFigFont{12}{14.4}{\rmdefault}{\mddefault}{\updefault}{\color[rgb]{0,0,0}$\psi$}%
}}}
\put(4051,-1801){\makebox(0,0)[lb]{\smash{\SetFigFont{12}{14.4}{\familydefault}{\mddefault}{\updefault}{\color[rgb]{0,0,0}$\omega$}%
}}}
\put(4471,-2281){\makebox(0,0)[lb]{\smash{\SetFigFont{12}{14.4}{\rmdefault}{\mddefault}{\updefault}{\color[rgb]{0,0,0}$\omega\psi$}%
}}}
\put(3556,-2273){\makebox(0,0)[lb]{\smash{\SetFigFont{12}{14.4}{\rmdefault}{\mddefault}{\updefault}{\color[rgb]{0,0,0}$\emptyset$}%
}}}
\put(2266,-6309){\makebox(0,0)[lb]{\smash{\SetFigFont{12}{14.4}{\rmdefault}{\mddefault}{\updefault}{\color[rgb]{0,0,0}$\omega 23$}%
}}}
\put(2281,-5416){\makebox(0,0)[lb]{\smash{\SetFigFont{12}{14.4}{\rmdefault}{\mddefault}{\updefault}{\color[rgb]{0,0,0}$\omega 14$}%
}}}
\put(10509,-8124){\makebox(0,0)[lb]{\smash{\SetFigFont{12}{14.4}{\rmdefault}{\mddefault}{\updefault}{\color[rgb]{0,0,0}$\omega \psi 4$}%
}}}
\put(8903,-8124){\makebox(0,0)[lb]{\smash{\SetFigFont{12}{14.4}{\rmdefault}{\mddefault}{\updefault}{\color[rgb]{0,0,0}$\omega \psi 3$}%
}}}
\put(5708,-8116){\makebox(0,0)[lb]{\smash{\SetFigFont{12}{14.4}{\rmdefault}{\mddefault}{\updefault}{\color[rgb]{0,0,0}$\omega \psi 1$}%
}}}
\put(7329,-8116){\makebox(0,0)[lb]{\smash{\SetFigFont{12}{14.4}{\rmdefault}{\mddefault}{\updefault}{\color[rgb]{0,0,0}$\omega \psi 2$}%
}}}
\put(8461,-7681){\makebox(0,0)[lb]{\smash{\SetFigFont{12}{14.4}{\rmdefault}{\mddefault}{\updefault}{\color[rgb]{0,0,0}$\psi 3$}%
}}}
\put(10029,-7674){\makebox(0,0)[lb]{\smash{\SetFigFont{12}{14.4}{\rmdefault}{\mddefault}{\updefault}{\color[rgb]{0,0,0}$\psi 4$}%
}}}
\put(6886,-7666){\makebox(0,0)[lb]{\smash{\SetFigFont{12}{14.4}{\rmdefault}{\mddefault}{\updefault}{\color[rgb]{0,0,0}$\psi 2$}%
}}}
\put(5304,-7658){\makebox(0,0)[lb]{\smash{\SetFigFont{12}{14.4}{\rmdefault}{\mddefault}{\updefault}{\color[rgb]{0,0,0}$\psi 1$}%
}}}
\put(6248,-7666){\makebox(0,0)[lb]{\smash{\SetFigFont{12}{14.4}{\rmdefault}{\mddefault}{\updefault}{\color[rgb]{0,0,0}$\omega 1$}%
}}}
\put(7831,-7673){\makebox(0,0)[lb]{\smash{\SetFigFont{12}{14.4}{\rmdefault}{\mddefault}{\updefault}{\color[rgb]{0,0,0}$\omega 2$}%
}}}
\put(9406,-7681){\makebox(0,0)[lb]{\smash{\SetFigFont{12}{14.4}{\rmdefault}{\mddefault}{\updefault}{\color[rgb]{0,0,0}$\omega 3$}%
}}}
\put(10981,-7673){\makebox(0,0)[lb]{\smash{\SetFigFont{12}{14.4}{\rmdefault}{\mddefault}{\updefault}{\color[rgb]{0,0,0}$\omega 4$}%
}}}
\end{picture}

%% file: new10.tex
\newcommand\Br{\mathop{{\rm Br}}}
\newcommand\Jac{\mathop{{\rm Jac}}}
\newcommand\End{\mathop{{\rm End}}}
\newcommand\im{\mathop{{\rm im}}}
\newcommand\bralg{\Br_1}
\newcommand\cores{\mathop{{\rm cores}}}
\newcommand\p{{\mathfrak p}}
\newcommand\q{{\mathfrak q}}
\newcommand\G{\mathbb{G}}
\newcommand\C{\mathbb{C}}
\newcommand\F{\mathbb{F}}
\newcommand\A{\mathbb{A}}
\newcommand\N{\mathbb{N}}
\newcommand\inv{\mathop{{\rm inv}}\nolimits}

\section{Arithmetical applications}\label{sectionsha}
In this section we apply the theory of the previous sections, combined with
the idea of the Brauer-Manin obstruction, to give an example of a K3
surface arising from $2$-descent on a family of curves of genus $2$ that
has rational points everywhere locally but not globally, and hence a family
of curves of genus $2$ all having nontrivial Tate-Shafarevich group.

\subsection{The Brauer group and the Brauer-Manin obstruction}
To explain the Brauer-Manin obstruction, let $V$ be a variety over a number
field $K$.  If there is a place
$\p$ of $K$ at which $V$ has no points, then of course $V$ has no
$K$-rational points.  But if $V$ has points everywhere locally, we can
sometimes use the Brauer group to prove that it does not have any points
defined over $K$.  We begin by defining the Brauer group of a scheme;
this material is taken from the beginning of \cite{Mi}, chapter 4.

\begin{definition}
Let $R$ be a ring.  Then an {\em Azumaya algebra} $A$ over $R$ is a free
$R$-algebra of finite rank as an $R$-module such that $A \otimes_R A^{\rm op}$
is isomorphic to $\End(A)$ by the map taking $a \otimes a'$ to the endomorphism
$x \ra axa'$.
\end{definition}

\begin{definition}
Let $V$ be a scheme.  An {\em Azumaya algebra} $\cal A$ on $V$ is a coherent
sheaf
of $R$-algebras whose stalk at every point $x$ of $V$ is an Azumaya algebra over
the local ring of $V$ at $x$.  The {\em Brauer group} of $V$ is the
semigroup of Azumaya algebras on $V$ under tensor product
modulo the subsemigroup of endomorphism
algebras of locally free sheaves.
\end{definition}

The Brauer group of a field $K$ is often defined as
the semigroup of finite-dimensional central simple algebras over $K$ under
the operation of tensor product modulo the subsemigroup $\{M_n(K): n \in \N\}$.
This is a special case of the definition above.
Alternatively, $\Br K$ can be thought of as $H^2(K,{K^{\rm sep}}^*)$,
which can be rewritten as $H^2_{\hbox{\small \rm \'et}}(\Spec K, {K^{\rm sep}}^*)$
using the standard equivalence (\cite{Mi}, Example III.1.7)
between \'etale cohomology of $\Spec K$ and Galois cohomology of
$\Gal(K^{\rm sep}/K)$-modules.  We extend this definition to general schemes
as follows.

\begin{definition}
For a variety $V$ over a field $K$, we use $\bar V$ to denote $V \otimes_K
\bar K$.
\end{definition}

\begin{definition}
The {\em cohomological Brauer group} $\Br V$ of a scheme $V$ is
$H^2_{\hbox{\small \rm \'et}}(V,\G_m)$.  If $V$ is defined over a field $K$,
then the {\em algebraic part of the Brauer group} $\bralg V$
is the kernel of the natural map $\Br V \ra \Br \bar V$.
\end{definition}

We can use the Brauer group to find obstructions to the existence
of points, but it is difficult to compute.
However, we can compute the algebraic part of the
Brauer group.  By \cite{Mi}, Prop.\ IV.2.15, the Brauer group and the
cohomological Brauer group are
isomorphic when $V$ is a smooth variety over a field.  For $V$ defined over
a number field, the algebraic part of
the Brauer group can be computed as $H^1(K,\Pic \bar V)$, as shown in
\cite{brd}, sect.\ 4.1.
As stated in Corollary \ref{VisK3}, we have
$\Pic \bar V = \NS \bar V$, a finitely generated free abelian group.

In this paper, we will only
consider the algebraic part of the Brauer group.  Now we explain how to use
the Brauer group to show that $V$ has no $K$-rational points.

\begin{definition}
Let $K$ be a number field and $\p$ a place of $K$.
The {\em local invariant} $\inv_\p(s)$ of an element $s$ of the Brauer group of
the local field $K_\p$ is its image under the natural homomorphism
$\Br K_\p \ra \Q/\Z$ (\cite{ser},
Proposition XII.6).
\end{definition}

Recall that this homomorphism is always injective,
and that it is surjective if $\p$ is non-archimedean, while its image is
generated by $1/2$ if $\p$ is real and is trivial if $\p$ is complex.

\begin{definition}
For $s \in \Br V$, $\p$ a place of $K$, and $P \in V(K_\p)$, the
{\em local invariant} of $s$ at $P$ is the local invariant $\inv_\p(s_P)$
of the element
$s_P \in \Br K_\p$ obtained by pulling back the cohomology class $s$ by the map
$\Spec K_\p \ra V$ whose image is the point $P$.  More concretely, it is the
local invariant of the Azumaya algebra over $K_\p$ whose multiplication table
is given by evaluating the elements of the multiplication table of an
Azumaya algebra representing $s$ at the coordinates of $P$.
\end{definition}

Every $K$-point $P$ of $V$ corresponds
to a morphism $\Spec K \rightarrow V$ that induces a map
$\Br V \rightarrow \Br K$ by pulling back cohomology.
We denote the image of an element $s \in \Br V$ under this
map by $s(P)$, yielding a map $V(K) \rightarrow \Br K$ that is
also denoted by $s$. Similarly for every place $\p$
(archimedean or non-archimedean) of $K$  we get a
map $s_\p \colon V(K_\p) \rightarrow \Br K_\p$. We obtain the following
commutative diagram, where the top horizontal map is the diagonal
embedding and $\lambda_s$ is defined to be the composition shown.
$$
\xymatrix{
V(K) \ar[rr] \ar[dd]_s && \prod_\p V(K_\p) \ar[dd]_{\prod_\p s_\p}
                                        \ar[rrdd]^{\lambda_s} && \cr
\cr
\Br K \ar[rr] && \bigoplus_v \Br K_\p \ar[rr]_{\sum \inv_\p} && \Q/Z\cr
}
$$
The bottom row is exact by class field theory, so the image of
the top horizontal map is contained in $\lambda_s^{-1}(0)$.
If we show that $\cap_{s \in \Br V} \lambda_s^{-1}(0) = \emptyset$, then
we may conclude that
there is no $K$-rational point on $V$, and we say that {\em the
Brauer-Manin obstruction blocks the existence of rational points
on $V$}.  It suffices to let $s$ run over a set of generators of
$\Br V/\Br K$. Often it is more convenient to consider only
elements of $\Br V/\Br K$ belonging to a subset $B$, and then we
speak of the Brauer-Manin obstruction on $B$ blocking the
existence of rational points.


\begin{proposition}\label{const}
Let $B$ be a subgroup of $\Br V/\Br K$ of order $2$
generated by the element $s$.  Then
the Brauer-Manin obstruction on $B$ blocks the existence of rational points on
$V$ if and only if, for all $\p$, the local invariant of $s$ is constant on
$K_\p$-points of $V$ and the constants do not add to $0$.
\end{proposition}

\begin{proof}
The sufficiency of this condition is clear.  For the necessity, note that,
if the local invariant is not constant on $K_\p$-points for some $\p$, we
may choose an arbitrary collection of local points at other places with
invariants adding to $\alpha \in \{0,1/2\}$ and then choose a $K_\p$-point
whose invariant is equal to $\alpha$, thus obtaining a system of points with
local invariants adding to $0$.  In the case where the local invariant
is constant on $K_\p$-points for all $\p$, the necessity of the condition that the
constants not add to $0$ is obvious.
\end{proof}

Now let us explain how to construct K3 surfaces with elements of the
Brauer group that are likely to give nontrivial obstructions.
Although it is usually easy to calculate $H^1(K, \Pic \bar V)$, at least
when $H^1(V,{\cal O}_V) = 0$ and so $\Pic \bar V$ is finitely generated,
it is difficult to find Azumaya algebras corresponding to nontrivial
cohomology classes.  Doing so requires finding
rational divisors in rational divisor classes.  Such divisors always exist if
$V$ has points everywhere locally, but in practice it is not easy to
find them.
Over an algebraically closed field the effective divisors in a linear
equivalence class constitute a $\P^n$, so the problem reduces to finding
a rational point on a locally solvable Brauer-Severi variety.  This can
be reduced to solving a norm equation from a field over which all elements
of $\Pic \bar V$ are defined, but even when the splitting field of $f$ is
fairly small this is likely to be impractical.

We will follow \cite{brd}, section 4.4,
in using elliptic fibrations on varieties to construct nontrivial elements
of the Brauer group.  For a variety with a fibration $\phi$,
we define the {\em vertical Picard group} $\Pic_\phi \bar V$ to be the subgroup of
$\Pic \bar V$ spanned by the classes of components of fibers.  The {\em vertical
Brauer group} $\Br_\phi V$ can then be defined as $H^1(K,\Pic_\phi \bar V)$.
(We deviate from the standard notations $\Pic_{\rm vert}$ and $\Br_{\rm vert}$
used in \cite{brd}
because it is necessary for us to distinguish between vertical Picard and
Brauer groups coming from different fibrations on the same variety.)
The inclusion of $\Pic_\phi \bar V$ into $\Pic \bar V$ gives a natural map
$\Br_\phi V \ra H^1(K, \Pic \bar V)$,
which need not be either injective nor surjective.

It is shown in \cite{brd}, Prop.\ 4.21,
that elements of $\Br_\phi V$
are represented by Azumaya algebras that are pulled back
from central simple algebras on the function field of
the target of $\phi$.  The stalk of such an Azumaya algebra is constant on all
fibers of $\phi$, so the local invariants of such algebras are too.

\begin{definition}\label{quat}
Let $F$ be any field, and $a$ and $b$ nonzero elements of $F$.  The symbol
$(a,b)$ denotes the central simple $F$-algebra of rank $4$ with basis $1, i, j, k$
and multiplication given by $i^2 = a, j^2 = b, ij = k, ji = -k$.
\end{definition}

\begin{proposition}\label{odd}
If $F$ is a local field, the local invariant of $(a,b)$ is $0$ if and only if
the quadratic form $x^2 -ay^2 - bz^2$ represents $0$ nontrivially in $F$;
otherwise it is $1/2$.
In particular, if the residue characteristic of $F$ is odd, then the invariant
is $0$ if and only if at least
one of the following conditions holds: $a$ and $b$ both have even valuation;
one of $a$ and $b$ is a square; or $ab$ is a square.
\end{proposition}

\begin{proof} The first statement is well-known (note in particular that if
$a$ or $b$ is a square then the local invariant is $0$).
The second follows from the discussion at the beginning of \cite{ser},
section XIV.4, which applies just as well to any finite extension of $\Q_p$
as to $\Q_p$ itself.
\end{proof}

\begin{proposition}\label{hasbr}
Let $\phi: V \ra \P^1$ be an elliptic fibration, and suppose that $V$ has bad
fibers of type $I_4$ over $(\alpha:1)$, where $[\Q(\alpha):\Q] = 4$.  Suppose
further that the field of definition of the components of the fiber at
$(\alpha:1)$ is
$\Q(\alpha,\sqrt c)$, where $c \in \Q(\alpha)$ is of square norm, and that
the $\Q(\alpha)$-components of this fiber consist of two disjoint lines.
Then the
pullback of the algebra $\cores_{\Q(\alpha)/\Q}(c,t-\alpha) \in \Br \Q(t)$
to $V$, where $t$ is the coordinate on the standard affine patch of $\P^1$,
is an element of $\Br_\phi V$.
\end{proposition}

\begin{proof}
This follows immediately from \cite{brd}, Proposition 4.28.
\end{proof}

\begin{remark}
Note that the hypotheses entail that the fiber consists of two pairs of disjoint
lines $L_1,M_1$ and $L_2,M_2$, where all the $L_i$ and $M_i$ are defined and conjugate over
one and the same quadratic extension $\Q(\alpha,\sqrt c)$ of $\Q(\alpha)$.
\end{remark}

For this to be useful to us, we need to know how to compute the local invariants
of the algebra $\cores_{\Q(\alpha)/\Q}(c,t-\alpha) \in \Br \Q$.
The following proposition addresses this question.

\begin{proposition}\label{coresinv}
(\cite{bre}, Lemma 5 (i)) The local invariant of
$\cores_{\Q(\alpha)/\Q} (c,t-\alpha)$ at $p$ is the sum of those of $(c,t-\alpha)$
at places of $\Q(\alpha)$ lying above $p$.
\end{proposition}

Now we are ready to show how we chose $f$ and $\delta$ so that there would be
a nontrivial element of the Brauer group arising from an elliptic fibration.

\begin{definition}\label{greek}
From now on, $f$ will always denote a polynomial which is
the product of three irreducible
quadratic polynomials $f_1, f_2, f_3$ and
$\delta$ will be an element of $(\Q[X]/(f))^*$.
For given $f$ and $\delta$,
let $\iota$ be an isomorphism from $\Q[X]/(f)$ to $\oplus_{i=1}^3 \Q[X]/(f_i)$
(which exists by the Chinese Remainder Theorem),
fix isomorphisms $\kappa_i$ from $\bar \Q[X]/(f_i)$ to $\bar \Q \oplus \bar \Q$
(again by the Chinese Remainder Theorem),
let $\upsilon_i$ be the component of $\iota(\delta)$ in $\Q[X]/(f_i)$,
and let $\upsilon_{ik}$ be the $k$th component of $\kappa_i(\upsilon_i)$.
For convenience define $\delta_j$ so that $\delta_{2(i-1)+k} =
\upsilon_{ik}$ (this notation coincides with the $\delta_j$ of Example
\ref{generic}).
Also let $\sigma_i$ be the nontrivial automorphism of $\Q[X]/(f_i)$
and let $r_i$ be a fixed root of $f_i$ in $\bar \Q$.
\end{definition}

\begin{theorem}\label{whichgp}
With $f$ and $\delta$ as above, let $V$ be the K3 surface
constructed from $f, \delta$. Suppose further that the splitting
field of $f$ is of degree $8$; that the norm of $\upsilon_1$ is a
square; that the norms of $\upsilon_2$ and $\upsilon_3$ multiplied
by the discriminant of $f_1$ are squares; and that the
$\upsilon_i$ are otherwise generic.  Then the field of definition
of the lines of $V$ has degree $32$, both elliptic fibrations
associated to the factorization $f = (f_1)(f_2 f_3)$ in Remark
\ref{factorell} satisfy the conditions of Proposition \ref{hasbr},
and the elements of the respective vertical Brauer groups
constructed in that proposition map to the same element of
$H^1(\Q,\Pic \bar V)$ and hence to the same element of $\Br V/\Br
K$.
\end{theorem}
\begin{proof}
The condition on $\delta$ shows that the norm of the projection of
$\delta$ to $\Q[X]/(f_2 f_3)$ is a square, so the elliptic
fibrations are defined over $\Q$ by Remark \ref{factorell}.  As we
did in the discussion just after Lemma \ref{lkLkLambda}, fix a
square root $\sqrt{\delta_j}$ of $\delta_j$ for $j\in
\{1,\ldots,6\}$ (corresponding to $\epsilon_j$ in Example
\ref{generic}). Let us now determine the action of the absolute
Galois group of $\Q$ on the lines. By Lemma \ref{kernelrho}, it
factors through the extension $m$ of the splitting field
$l=\Q(r_1,r_2,r_3)$ of $f$ obtained by adjoining all elements of
the form $\sqrt{\delta_i} \sqrt{\delta_j}$. Note that $\delta_1
\delta_2 = \upsilon_{11}\upsilon_{12} =
N_{(\Q[X]/f_1)/\Q}(\upsilon_1)$ is a square by hypothesis.
Similarly, letting $\Delta_1$ denotes the discriminant of $f_1$,
we find that $\Delta_1\delta_3\delta_4$ and
$\Delta_1\delta_5\delta_6$ are squares, and thus that
$\delta_3\delta_4$ and $\delta_5\delta_6$ are squares in
$\Q(\sqrt{\Delta_1})=\Q(r_1)$. It follows that up to square
factors in $\Q(r_1)$ every element of the form $\sqrt{\delta_i}
\sqrt{\delta_j}$ is equivalent to $\delta_1 \delta_3$, $\delta_1
\delta_5$, or $\delta_3\delta_5$. We conclude that
$m=l(\sqrt{\delta_1}\sqrt{\delta_3},\sqrt{\delta_1}\sqrt{\delta_5})$.
As the $\upsilon_i$ are otherwise generic, we find that the field
$m$ of definition of the lines has degree $4$ over $l$ and
therefore degree $32$ over $\Q$.

We now describe the Galois group of $m$ over $\Q$
in terms of the $s, t$ described in the
discussion following Proposition \ref{rhosurj}.
The $\upsilon_i$ are generic aside from their given properties,
so the field  $m'= l(\sqrt{\delta_1},\ldots,\sqrt{\delta_6}) =
m(\sqrt{\delta_1})$ has degree $2$ over $m$ (cf. discussion
following Lemma \ref{lkLkLambda} and Example \ref{generic}).
Recall that for a permutation
$p \in S_6$, the automorphism $t_p$ of $m$ is induced from $m'$
by sending $\delta_{j}$ and $\sqrt{\delta_j}$ to
$\delta_{p(j)}$ and $\sqrt{\delta_{p(j)}}$ respectively.
For a subset $I \subseteq \{1,2,\dots,6\}$, the automorphism
$s_I$ of $m$ is induced from $m'$ by fixing all $\delta_j$ and
sending $\sqrt{\delta_j}$ to $\pm \sqrt{\delta_j}$ where the sign
is negative if and only if $j \in I$. We will also refer to the
$32$ lines on $V$ using the notation $L_I$ introduced before
Lemma \ref{whichI}.

The automorphism $\sigma_1$ lifts to an automorphism of $m'$ that
fixes $r_2$ and $r_3$, and therefore all $\delta_j$ for $3 \leq j
\leq 6$. Because $\sigma_1$ sends $\sqrt{\Delta_1}$ to
$-\sqrt{\Delta_1}$, and $\delta_3\delta_4$ and $\delta_5\delta_6$
are squares up to a factor $\Delta_1$, this lift changes the signs
of the square roots of one of $\delta_3$ and $\delta_4$ and of one
of $\delta_5$ and $\delta_6$. The induced automorphism of $m$ is
then $t_{(12)} s_{i,j}$ for some $i \in \{3,4\}$ and $j \in
\{5,6\}$. The automorphisms $\sigma_2$ and $\sigma_3$ lift to
$t_{(34)}$ and $t_{(56)}$ respectively. Together with
$s_{\{1,2\}}$ and $s_{\{3,4\}}$ these elements generate a subgroup
of $\Gal(m/\Q)$ of order $32$, so they generate the full Galois
group. The orbits of this group on the lines are of order $8$ and
each orbit contains two nonintersecting lines from each of the
four fibers of one fibration. In particular, all the four $I_4$
fibers are conjugate, so they are indeed defined over a field of
degree $4$.

By Remark \ref{factorell} the fibers of the elliptic fibrations are orbits
of $\Lambda$ under the group generated by $s_1$ and $s_2$
in $\Aut \Lambda$ (only their product $s_1s_2$ is in the subgroup of
$\Aut \Lambda$
induced by Galois). Consider first the fibration associated to
the \gonquad{} $\cS$ that contains a $4$-gon $S$ containing $L_0$. Then
that $4$-gon is $S = \{L_0,L_1,L_2,L_{12}\}$. One easily checks
that the subgroup $\Gal(m/k(S))$ is generated by $t_{(34)}$,
$t_{(56)}$, and $s_{\{1,2\}}$. This group fixes $\Q(r_1)$, so we
have $\Q(r_1) \subset k(S)$, and in fact the group
$\Gal(m/\Q(r_1))$ is generated by $\Gal(m/k(S))$ and $s_{\{3,4\}}$.
Under $\Gal(m/k(S))$ the $4$-gon $S$ breaks up into the orbits
$\{L_0,L_{12}\}$ and $\{L_1,L_2\}$, each consisting of two disjoint lines.

Let $L$ be any line in $S$.
The subgroup $\Gal(m/k(L))$ is generated by $t_{(34)}$ and
$t_{(56)}$, which is normal in $\Gal(m/\Q(r_1))$. Therefore,
$k(L)$ is Galois over $\Q(r_1)$ and the corresponding Galois group
is $(\Z/2\Z)^2$. Since $k(S)$ is one of the quadratic subfields, it
follows from elementary Galois theory that $k(L)$ can be obtained
from $k(S)$ by adjoining the square root of an element $c \in\Q(r_1)$.
As $k(S)$ has degree $4$ over $\Q$ and $\Q(r_1)$ is a quadratic
subextension, the norm of $c$ from $k(S)$ to $\Q$ is indeed a square.
The arguments for the opposite fibration are completely similar.

The final statement of the proposition,
that the elements of the vertical Brauer group obtained in this
way give the same element of $H^1(\Q,\Pic \bar V)$, is proved by a
calculation using {\sc magma}.  The point is to verify that the
images of the nontrivial elements of $H^1(\Q,\Pic_\phi V)$ and
$H^1(\Q,\Pic_{\phi'} V)$ constructed in Proposition \ref{hasbr} in
$H^1(\Q,\Pic \bar V)$ are equal, where $\phi$ and $\phi'$ are the
two fibrations associated to the factorization $f = (f_1)(f_2
f_3)$ as in Remark \ref{factorell}.  See \cite{elec} for details.
\end{proof}

\begin{remark}
In special cases it is possible for the Picard group of $V$ to have higher
rank than in the generic case (in fact this happens in the example that
we present immediately below).  Thus we may have constructed elements,
not of $H^1(\Q,\Pic \bar V)$, but only of $H^1(\Q,P)$, where $P$, the subgroup
of $\Pic \bar V$ generated by the classes of the lines, is a proper subgroup of
$\Pic \bar V$.  The inclusion $P \ra \Pic \bar V$ gives a map
from $H^1(\Q,P)$ to $H^1(\Q,\Pic \bar V)$, which allows us to consider
the elements we have constructed as elements of the Brauer group.
It is possible that our elements could be in the kernel of this map in
some situations.
However, they are well-defined Brauer classes,
and so they may be used to attempt to prove that $V$ has no rational points.
If their local invariants do not add to $0$, it indicates that they are
not in the kernel.
\end{remark}

\begin{remark} It is not essential for the method that the element of
$H^1(\Q,\Pic \bar V)$ be obtainable from two different fibrations.  However,
as was first pointed out by Swinnerton-Dyer, this means that the local invariants
are constant on the fibers of two different fibrations and is therefore much
more likely to be constant than would otherwise be the case.  This greatly
facilitated finding the example below.
\end{remark}

\subsection{An explicit example}\label{sectionexample}
Now let us present our example, which was found by searching over various
choices of $f_1, f_2, f_3$ with small splitting fields $\Q(i), \Q(\sqrt 2),
\Q(\sqrt 5)$ such that the discriminant of $f_1f_2f_3$ is small
and $\delta$ such that the field of definition of the
lines would not introduce new bad primes.
For the rest of this paper, let $f = (x^2+1)(x^2-2x-1)(x^2+x-1) = f_1(x)
f_2(x)f_3(x)$, let $A_f = \Q[X]/f(X)$,
let $E$ be the algebra $ \Q(i) \oplus \Q(\sqrt 2) \oplus \Q(\sqrt 5)$,
fix an isomorphism $\iota: A_f \ra E$ as in Definition \ref{greek},
and let
\begin{equation}\label{deltais}\delta = (-2X^5 + 3X^4 + 5X^3 - 8X^2 + 7X + 7)/6
\end{equation}
be the element of $A_f$, so that $\iota(\delta) =
(3,-(1 + \sqrt 2), (1+\sqrt 5)/2)$.
Then $f,\delta$ satisfy the conditions of Theorem \ref{whichgp}.
For nonzero rational $t$ let $C_t$ be the curve $y^2 = tf(x)$.
Let $V$ be the K3 surface $V_{f,\delta}$.
Our goal is to prove the main theorem, Theorem \ref{main}, restated and
slightly reworded here for ease of reading:

\begin{theorem}\label{mainnew}
Let $S$ be the union of $\{5\}$
with the set of primes that split completely in the field of definition
of the lines of $V$, which is
$$F = \Q\left(\sqrt {-1}, \sqrt 2, \sqrt 5, \sqrt {-3(1 + \sqrt 2)}, \sqrt {6(1 + \sqrt 5)}\right).$$
Then for all $n$ which are products of elements of $S$, the $2$-part of the
Tate-Shafarevich
group of the Jacobian of the curve $C_{-6n}$ is nontrivial.
\end{theorem}

To do so, we will follow the strategy outlined in the introduction:
first we will show that $\delta$ gives an element of the
fake Selmer group of the curve in question by showing that the corresponding
principal homogeneous space of the Jacobian of $C_{-6n}$ has points everywhere
locally. Then we will use the element of the Brauer group described
above to prove that $V_{f,\delta}$ has no rational points and conclude 
that the principal homogeneous space has no rational points either.  
We begin by summarizing some results
from \cite{S} on the fake Selmer group of a hyperelliptic curve of genus~$2$.

\begin{definition}\label{descent}
Let $g$ be a squarefree polynomial of degree $6$ over $\Q$ and $C$ the curve
$y^2 = g(x)$.
For any field $K$ of characteristic $0$, let $H_K = \ker(N:
(A_g \otimes K)^*/{(A_g \otimes K)^*}^2K^* \ra K^*/{K^*}^2)$
with $A_g = \Q[X]/g(X)$.
(Note that the norm is well-defined,
because $\deg g$ is even and so $N(K^*) \subset {K^*}^2$.)
Write $H$ instead of $H_\Q$.  As in
\cite{S}, Prop.\ 5.5, let the {\em fake Selmer group}
$\Sel^{(2)}_{\rm fake}(\Q,\Jac C)$ of $C$ be the
subgroup of $H$ consisting of elements that are everywhere
locally in the image of the
$2$-descent map on the Jacobian of $C$.  Define $\Delta_K$,
the {\em descent map over $K$},
to be the function from the set of $K$-rational points of $C$ which are
neither Weierstrass points nor points at infinity
to $H_K$ such that $\Delta_K(x_0,y_0) = 1 \otimes x_0 - X \otimes 1$.
\end{definition}

\begin{proposition}\label{descjac}
The function $\Delta_K$ may be extended
multiplicatively to a function from $J(K)$ to $H_K$.
\end{proposition}

\begin{proof}
This is essentially \cite{Sc}, Lemma 2.1, as modified by the discussion
in section 2.5; see also
\cite{S}, sect.\ 6 for explicit formulas for Weierstrass points.
\end{proof}

%

\begin{proposition}\label{selmer}
Let $S$ be the set of primes introduced in Theorem \ref{mainnew}.
Then for all $n$ that are products of elements of $S$,
the fake Selmer group of the Jacobian of $C_{-6n}$ contains $\delta$.
Equivalently, the principal homogeneous space of $\Jac(C_{-6n})$
corresponding to $\delta$ is everywhere locally solvable.
\end{proposition}

\begin{proof}
We need only show that $\delta$ is in the image of the local Selmer maps
for primes of bad reduction and $\infty$.  The primes of bad reduction
are $2$, $3$, $5$, and primes dividing $n$.  At a prime $p$ dividing $n$,
we have that $\delta$
corresponds under $\iota$
to an element
of the form $(3,3a^2,3b^2)$ in $E \otimes \Q_p$.  Therefore, $\delta$ is contained
in ${(A_f \otimes \Q_p)^*}^2 \Q_p^*$, so it is the identity element in $H_{\Q_p}$
and is
in the image of every homomorphism.  Thus $\delta$ is in the image of the local
$2$-descent map.  Note also that every product of elements of $S$ is positive,
is equivalent to $1$ or $5$ in $\Q_p^*/{\Q_p^*}^2$ for $p = 2, 5$,
and is a unit locally at $3$.
It follows that we need only check the
assertion for $n = 1$, $p = 2, 3, 5, \infty$, and $n = 5$, $p = 2, 3, 5$.
This can be done by using {\sc magma}'s {\tt TwoSelmerGroup} command
to compute the fake Selmer groups of
$C_{-6}$ and $C_{-6 \cdot 5}$ and verifying that $\delta$ is an element
of both.
The last statement follows by applying the fact that
the principal homogeneous space over a field $K$ corresponding to $\delta$ has points
if and only if $\delta$ is in the image of $\Delta_K$ to all completions of
$\Q$.
\end{proof}


\begin{proposition}\label{nontriv}
Assume Stoll's condition $(\dagger)$ (this is automatic
for curves of even genus).  If the image of $\Jac(C_t)/2\Jac(C_t)$ under
the $2$-descent map $\Delta_\Q$ is properly
contained in the fake Selmer group, then $\Jac(C_t)$ has nontrivial Tate-Shafarevich
group.
\end{proposition}

\begin{proof}
Let $c = \dim \Sel^{(2)} \Jac(C_t) - \dim\Sel^{(2)}_{\rm fake} \Jac(C_t)$
and let $d$ be the dimension of the ``Cassels kernel'' of $\Jac(C_t)$,
namely, the kernel of the descent map modulo $2\Jac(C_t)$.  Since we are
assuming condition $(\dagger)$, we have $c = 0$ if and only if condition
$(\ddagger)$, and condition $(\ddagger)$ implies that $d = 0$.  Both $c$ and $d$
are at most $1$, so $c \ge d$.  Under the assumptions of the proposition,
it follows that
$$\eqalign{\dim\Sel^{(2)} \Jac(C_t) &= \dim\Sel^{(2)}_{\rm fake} \Jac(C_t) + c\cr
&> \dim \im \Jac(C_t) + c \cr
& \ge \dim \im \Jac(C_t) + d\cr
& = \dim \Jac(C_t)/2\Jac(C_t),}$$
from which the conclusion is immediate.
\end{proof}

Now we indicate the relation between the descent map and the K3 surface
$V_{f,\delta}$ that we have defined.

\begin{proposition}\label{usev}
With $f$ and $\delta$ as above, let $V$ be the K3 surface constructed from
$f, \delta$.  Suppose that $V$ has no rational points.  Then $\delta$ is
not in the image of the $2$-descent map for the Jacobian of the curve $y^2
= f(x)$.
\end{proposition}

\begin{proof}
Suppose, to the contrary, that $\delta$ were in the image; that is, that
$\Delta_\Q(D) = \delta$ for some divisor $D$ of degree $0$.  By
Riemann-Roch, this divisor may be taken to be of the form $(x_1,y_1) +
(x_2,y_2) - H$, where $H$ is the hyperelliptic divisor.  Since
$\Delta_\Q(H) = 1$, this means that $\Delta_\Q(D) =
\Delta_\Q((x_1,y_1)+(x_2,y_2)) = (x_1-X)(x_2-X)$ in $A_f^*/(A_f^*)^2\Q^*$,
where we have identified $A_f \otimes_\Q \Q$ with $A_f$.  In other words,
for some $r \in \Q$ and $q \in A_f$ we have $\delta q^2 = rx_1x_2 -
r(x_1+x_2)X + rX^2$.  In particular, the point of $\P^5$ whose coordinates
are the coefficients of $q$ lies on $V$.
\end{proof}

\begin{remark}
Propositions \ref{nontriv} and \ref{usev} together imply that if $\delta$ 
is an element of the fake Selmer group $\Sel^{(2)}_{\rm fake}(\Q,\Jac C_t)$
and the corresponding $V_{f,\delta}$ has no rational points, then 
$\Jac(C_t)$ has nontrivial Tate-Shafarevich group. This also follows from 
the fact that $V_{f,\delta} = V_{tf,\delta}$ is a quotient of the homogeneous 
space of $\Jac C_t$ corresponding to $\delta$. Indeed, this implies that
if $V_{f,\delta}$ has no rational points, then neither does the homogeneous 
space, which implies its image in the Tate-Shafarevich group is nontrivial.
\end{remark}

It is also worth mentioning that the existence of the nontrivial
element of $\Sha$ that we will exhibit
is not a consequence of the results of
\cite{PS} on odd curves, which we now describe briefly.

\begin{definition}
(\cite{PS})
Let $C$ be a curve over $\Q$ of genus $g$.  We say that $C$ is {\em deficient}
at $p$, where $p$ is a prime or $\infty$, if $C$ has no rational divisor of
degree $g-1$ over $\Q_p$.  We say that $C$ is {\em even} or {\em odd} depending
on the parity of the number of places at which $C$ is deficient.
\end{definition}

Denote the $2$-primary part of $\Sha(\Jac(C))$ by $\Sha(\Jac(C))[2^\infty]$.
Then, if $\Sha(\Jac(C))[2^\infty]$ is finite, its order
is a square if and only if $C$ is even (\cite{PS}, Theorem 11).
It follows that if $C$ is odd that $\Sha(\Jac(C))[2^\infty]$ is nontrivial.
The following proposition proves
that the elements of $\Sha$ that we find are not merely artifacts of the
oddness of our curves.

\begin{proposition}
If $n$ is a product of primes in $S$, then $C_{-6n}$ is even.
\end{proposition}

\begin{proof}
A curve is never deficient at a prime $p$ of good reduction, for there are
always points over all sufficiently large finite fields of characteristic $p$
and therefore over unramified extensions of $\Q_p$ of all sufficiently large
degrees.  The primes of bad reduction are $2, 3, 5$, and those dividing $n$.
The point $(i,0)$ on $C_{-6n}$
is defined over $\Q_p$ for all $p$ congruent to $1$ mod $4$,
which includes $5$ and all primes dividing $n$.
It is also clear
that there are real points on $C_{-6n}$.  As a result, the only primes that
need to be considered are $2$ and $3$.  One checks (for example, using the
{\tt IsDeficient} command in {\sc magma}) that $C_{-6n}$ is deficient at those primes.
It follows that $C_{-6n}$ is even.
\end{proof}

We now start to show that $V$ has no rational points, from which it follows that
$\delta$ is not in the image of the global $2$-descent map.
First we will limit the set of primes to be considered in the calculation
of the Brauer-Manin obstruction,
then we will explain how to calculate the invariants there, and after that
we will actually compute the invariants.

\begin{lemma}\label{constant}
Let $W$ be a variety over $\Q$
with good reduction at $p$ and let $s \in \Br W$.
Then the local invariant of $s$ at $p$ is constant.
\end{lemma}

\begin{proof} This is a weaker version of Theorem 1 of \cite{bre}.
\end{proof}

\begin{lemma}\label{goodp}
Let $W$ be an elliptic surface over $\Q$ with an Azumaya algebra $s$ given by
$\cores(c,t - \alpha)$ as in Proposition \ref{hasbr}.
Let $p$ be a prime of good reduction for $W$
such that the fiber at infinity of $W$ has smooth
$\Q_p$-rational points (in particular, this is true if the fiber at infinity
has good reduction mod $p$).
Then the local invariant of $s$ at $p$ is equal to $0$.
\end{lemma}

\begin{proof}
By Lemma \ref{constant}, the local invariant is constant, so it suffices
to evaluate it at one point.  Since the fiber at infinity has smooth
$\Q_p$-rational points, so do all sufficiently near fibers.  In particular,
for all sufficiently large integers $k$ the fiber at $p^{-2k}$ has rational
points, and the local invariant is therefore
$$\cores(p^{-2k}-\alpha,c) =
 \sum_{\p | p} (p^{-2k}-\alpha,c)_\p.$$  But $p^{-2k} - \alpha$ is a
square in all completions at primes above $p$
for all sufficiently large $k$, and hence
the local invariant is $0$ for sufficiently large $k$.
\end{proof}

\begin{lemma}\label{inf}
Let $W$ be an elliptic surface over $\Q$
with an Azumaya algebra $\cores (c,t-\alpha)$,
where $c$ has square norm.  Suppose that the valuation
of $\alpha/4p^n$ is positive at all primes above $p$ at which $c$
is not a square.  Then the invariant at $p$ is $0$ on all fibers of coordinate
$t$ where $v(t) \le n$.
\end{lemma}

\begin{proof}
We notice that $t/(t - \alpha)$ is a square for
all $t$ with $v(t) \le n$,
and therefore that the invariants of $(c,t)$ and $(c,t-\alpha)$
are equal at all primes above $p$ where $c$ is not a square.
They are also equal at primes above $\p$ at which $c$ is a square, being
both equal to $0$ there, so they are equal at all primes above
$p$.  Thus the invariant at $p$ of $\cores(c,t-\alpha)$ is equal to that
of $\cores (c,t)$ (see Proposition \ref{coresinv}).
But this is equal to the invariant of
$(N(c),t)$, which is $0$ because $N(c)$ is a square.
\end{proof}

\begin{lemma}\label{constdisc}
Let $t_0 \in \Q_p$ and $n \in \Z$ be such that
$p^n/4(t_0-\alpha)$ has positive valuation in $\Q(\alpha)_\p$
for all primes $\p$ above $p$ where $c$ is not
a square locally.  Then
the local invariant of the Azumaya algebra $\cores(c,t-\alpha)$
is the same for all values of $t$ in the disc $t_0 + p^n \Z_p$.
\end{lemma}

\begin{proof}
It is sufficient to show that $(t-\alpha)/(t_0-\alpha)$ is a square
for all $t \in t_0 + p^n \Z_p$ at all primes $\p$
above $p$ where $c$ is not a square.
But this is equal to $1 + kp^n/(t_0 - \alpha)$
for some $k \in \Z_p$, which by assumption is congruent to $1$ mod $4\p$.
It is therefore a square by Hensel's lemma.
\end{proof}

\begin{lemma}\label{disc}
It can be effectively determined whether there are $\Q_p$-rational points of
$V$ mapping under a given fibration to a given disc $D$ in $\P^1(\Q_p)$.
\end{lemma}

\begin{proof}
By changing coordinates, we may assume that the disc is $\{(\Z_p:1)\}$.
Consider the graph of the fibration as a subscheme of
$\P^5 \times \P^1$.  It is sufficient to determine whether there is a
standard affine patch $\A^5$ of $\P^5$ such that the intersection of the
patch $\A^5 \times D$
of $\P^5 \times \P^1$ with the graph of the fibration
has $\Z_p$-points.  Since the graph is smooth, this can
be determined using Hensel's lemma.
\end{proof}

\begin{theorem}\label{calcloc}
Let $V$ be a smooth variety with an elliptic fibration $\phi$ over
$\P^1$ satisfying the hypotheses of Proposition \ref{hasbr}
and let $\cores(c,t-\alpha) \in \Br \phi$ be the Azumaya algebra
constructed there.
Let $p$ be a prime.  Then the set of
values of the local invariant of $\cores(c,t-\alpha)$ on $V(\Q_p)$
can be effectively determined.
\end{theorem}

\begin{proof}
Using Hensel's lemma or Lemma \ref{disc} we can check whether
$V$ has any $\Q_p$-points. If not, then we are done, so we may assume that
$V$ has $\Q_p$-points.  If $V$ has good reduction at $p$, then by Lemma
\ref{constant} the local invariant $\cores(c,t-\alpha)$ is constant, so it
suffices to evaluate it at one point of $V$.  Since $V$ is nonsingular, its
$\Q_p$-points are dense, so we can find a point for which $t-\alpha$ is
bounded away from $0$ and $\infty$.  It is then easy to compute the local
invariant there.

We will describe the algorithm for calculating the values of the
local invariant of
$\cores (c,t-\alpha)$ at a prime $p$ of bad reduction for $V$ and then
prove that it terminates.  The point is that this local invariant turns out
to be locally constant. In fact, we will show how to find an explicit
finite covering of $\P^1(\Q_p)$ by discs in the $p$-adic topology, such
that for each disc the invariant is constant on the set of $\Q_p$-points on
$V$ mapping to that disc under the elliptic fibration.

First we show that such a disc exists around each point in $\P^1(\Q_p)$,
starting with the point at infinity.  After a change of variables we may
assume that the fiber at $t = \infty$ is smooth over $\Q_p$.  If the
fiber at $\infty$ does have points over $\Q_p$, then by Lemma \ref{inf} we
can find an $n$ such that the local invariant is $0$ on all points on
fibers above $t$ with $v(t) < n$.  In this case the disc $v(t) < n$ has the
desired property.  If the fiber at $t=\infty$ does not have any
$\Q_p$-points, then there is a $p$-adic neighborhood of $\infty\in
\P^1(\Q_p)$ above which there are no $\Q_p$-points either, so the invariant
is clearly constant above any disc contained in this neighborhood.

Now observe that, by Lemma \ref{constdisc}, every finite $t_0$ has a suitable
neighborhood except for those that are roots of the minimal polynomial
of $\alpha$, so we are reduced to considering the finite set of these $t_0$.
Here the argument will be more subtle.
Fix such $t_0$ and let $\p$ be the corresponding place of $\Q(\alpha)$,
so that the completion map $\rho\colon \Q(\alpha) \hookrightarrow
\Q(\alpha)_\p$ sends $\alpha$ to $t_0$.  Here we used the obvious embedding
$\Q_p \hookrightarrow \Q(\alpha)_\p$, which is an isomorphism as $\alpha$
is contained in $\Q_p$.

First consider the case that the fiber above $t=t_0=\rho(\alpha)$ does have
a point $P$ over $\Q_p \isom \Q(\alpha)_\p$.  Then $P$ must be a smooth
point of the $\Q(\alpha)_\p$-component it lies on, because by hypothesis
each $\Q(\alpha)$-component is smooth. Therefore the geometric component on
which $P$ lies is defined over $\Q(\alpha)_\p\isom \Q_p$, so by hypothesis
all geometric components of the fiber at $t=t_0=\rho(\alpha)$ are defined
over $\Q_p$, which means that $\rho(c)$ is a square. For all other places
$\mathfrak{q}$ above $p$ the embedding $\Q_p \hookrightarrow \Q(\alpha)_\q$
does not send $t_0$ to $\alpha$, so we can find an integer $n$ such that
under each such embedding the element $p^n/4(t_0-\alpha)$ has positive
valuation in $\Q(\alpha)_\q$. By Lemma \ref{constdisc} the invariant is
constant above the disc $t_0+p^n\Z_p$. In the implementation of this
algorithm it will be useful to remember that $c$ is a square in
$\Q(\alpha)_\p$, as this implies that the local invariant corresponding to
the place $\p$ is $0$ globally.  In the case that the fiber above $t=t_0$
does not have any $\Q_p$-points, there is a $p$-adic neighborhood of $t_0
\in \P^1(\Q_p)$ above which there are no $\Q_p$-points either, so the
invariant is constant above any disc contained in this neighborhood.

The algorithm is as follows. First apply a change of coordinates to $\P^1$
if necessary, so that the fiber at infinity is smooth. Then find a
neighborhood of infinity above which the local invariant is constant. This
reduces to considering $t$ in some set of the form $p^n \Z_p$. For each
$t_0$ that maps to $\alpha$ in $\Q(\alpha)_\p$ for some place $\p$ above
$p$, find an appropriate disc around $t_0$. In doing so we may find that
the value of the local invariant corresponding to some $\p$ is $0$
for all $\Q_p$-points of $V$.  Such $\p$ can
be ignored in the remainder of the computation. If there are points above
any of the discs found thus far, compute the corresponding local
invariants.

We now divide the remaining region of $\P^1(\Q_p)$ into discs and place
them in a queue, recording which local invariants we already know to
occur. For each disc, we use Lemma \ref{constdisc} to try to discover that
the local invariant is constant there. If it is, and if the local invariant
is already known to occur, the disc may be ignored, as the conclusions do
not depend on whether there are rational points there. If it is, and if the
local invariant is {\em not} already known to occur, we use
Lemma \ref{disc} to test whether there are rational points in that disc. If
there are rational points on the disc, we record that the invariant on the
disc occurs.   The only possible values of the invariant of an element
of $\Br V/\Br \Q$ of order $2$ are $0,\frac{1}{2} \in \Q/\Z$; if both
of these are now known to occur, we are done.
On the other hand,
if the local invariant is not constant on the disc, we divide
it into $p$ smaller discs and add them to the end of the queue.  In other
words, we perform a breadth-first search.  The reason for using
breadth-first rather than depth-first search is that local invariants are
more likely to be different on relatively distant points of a $p$-adic
disc, so breadth-first search will tend to find different invariants more
quickly. When the queue is empty we are done and have recorded all values
of the local invariant.

Now we show that this algorithm terminates.  It could fail to do so only if
there is an infinite descending sequence of discs above which there are
rational points but in which the local invariant is not constant.  Let
$t_0$ be the unique point in the intersection of the discs. Then we have
shown before that there is a disc around $t_0$ on which the local invariant
of $\cores(c,t-\alpha)$ is constant, which is a contradiction, so the
algorithm does terminate.
\end{proof}

\begin{remark}
This theorem is somewhat special to our situation.  If the
Azumaya algebra were constructed using
fibers of type $I_2$, for example, this proof would not apply, and indeed in
this case the local invariant associated to a place $\p$ above $p$ need not be
locally constant on $\P^1$ around a $t_0$ that maps to $\alpha$ in $\Q(\alpha)_\p$.
\end{remark}

Let us now show how to apply this algorithm in our example.

\begin{proposition}\label{obstr}
The Brauer-Manin obstruction blocks the existence of rational points on
the K3 surface $V$ corresponding to $(f,\delta)$.
\end{proposition}

\begin{proof}
The surface $V$ is defined by the equations $q_1 = q_2 = q_3 = 0$, where
$$\eqalign{
q_1 &= x_1^2 + 2x_1x_3 + 8x_1x_4 + 10x_1x_5 + 24x_1x_6 + x_2^2 + 8x_2x_3 + 10x_2x_4 + 24x_2x_5 + 58x_2x_6 + 5x_3^2 + 24x_3x_4\cr
&+ 58x_3x_5 + 152x_3x_6 + 29x_4^2 + 152x_4x_5 + 346x_4x_6
+ 173x_5^2 + 856x_5x_6 + 1017x_6^2,\cr
q_2 &= x_1x_2 + 2x_1x_3 + 3x_1x_4 + 6x_1x_5 + 17x_1x_6 + x_2^2 + 3x_2x_3 + 6x_2x_4 + 17x_2x_5 + 40x_2x_6 + 3x_3^2 + 17x_3x_4 \cr
&+ 40x_3x_5 + 97x_3x_6 + 20x_4^2 + 97x_4x_5 + 230x_4x_6
+ 115x_5^2 + 563x_5x_6 + 675x_6^2,\cr
q_3 &= 6x_1x_3 + 8x_1x_4 + 12x_1x_5 + 30x_1x_6 + 3x_2^2 + 8x_2x_3 + 12x_2x_4 + 30x_2x_5 + 84x_2x_6 + 6x_3^2 + 30x_3x_4 + 84x_3x_5 \cr
&+ 192x_3x_6 + 42x_4^2 + 192x_4x_5 + 466x_4x_6
+ 233x_5^2 + 1112x_5x_6 + 1360x_6^2.\cr}$$
The primes of bad reduction of $V$ are those dividing the discriminant of $f$,
namely $2,3,5$, and those involved in $\delta$, which are again $2, 3, 5$.
Each of the fibrations is defined by two alternative pairs of linear forms
(see Lemma \ref{xyz} and following remarks).  We let
$$\eqalign{
l_1 &= x_1 + 13x_3 + 17x_4 + 68x_5 + 123x_6,\cr
l_2 &= x_2 - 16x_3 - 19 x_4 -84x_5 -145x_6,\cr
l_3 &= 2x_2 -8x_3 -8x_4 -42x_5 -68x_6,\cr
l_4 &= -x_1 -4x_2 +9x_3 +5x_4 +46x_5 +61x_6,\cr}$$
and then one of the fibrations, which we will denote by $F_1$,
is given by $(l_1:l_2)$ or $(l_3:l_4)$, while
the other, which will be written $F_2$,
is given by $(l_1:l_3)$ or $(l_2:l_4)$.  To verify that the two
sets of defining equations for each $F_i$ give the same map,
note that $l_1l_4-l_2l_3 = -q_1-4q_2+q_3$.
We have already seen that $F_1$ satisfies the hypotheses of
Proposition \ref{hasbr}; now we calculate the local invariant of the
associated Azumaya algebra by means
of the algorithm described in Theorem \ref{calcloc}.
We will show that this Azumaya algebra
has local invariant $1/2$ at $2$ and $0$ at all other primes;
the energetic reader may wish to verify that the Azumaya algebra corresponding
to $F_2$ has local invariant $1/2$ at $2$, $3$, and $5$, and
local invariant $0$ elsewhere.

The $I_4$-fibers of $F_1$ lie over the point $(129r^3+187r^2+285r-1469)/449$,
where $r^4-12r+13 = 0$.  The fact that this extension is totally complex
implies that the local invariant at $\infty$ is 0, because that local invariant
is the sum of local invariants of central simple algebras over $\C$.
The components of these fibers are defined over the
extension $\Q(r,\sqrt{c})$ where $c = -6r^3-9r^2-12r+57$
(note that actually $c$ has minimal polynomial $r^2 - 6r + 18$, so
$c \in \Q(i)$, as expected from the proof of Theorem \ref{whichgp}).

The fiber at $\infty$ is the vanishing locus of the polynomials
$$\eqalign{
x_1 &+ 55x_3 + 71x_4 + 290x_5 + 519x_6,\cr
x_2 &- 16x_3 - 19x_4 - 84x_5 - 145x_6,\cr
x_3^2 &+ 414x_3x_4 - 4538x_3x_5 - 5876x_3x_6 + 448x_4^2 - 3796x_4x_5 -
    4200x_4x_6 - 23754x_5^2 - 73732x_5x_6 - 56083x_6^2,\cr
7596x_3x_4 &- 83956x_3x_5 - 108804x_3x_6 + 8241x_4^2 - 70306x_4x_5 -
    77950x_4x_6 - 438959x_5^2 - 1362754x_5x_6 - 1036791x_6^2,\cr}$$
and it has good reduction outside $2,3,5,397,449$.
This can be checked
either by a Groebner-basis computation or by embedding the fiber in
$\P^3$ as the curve defined by two quadrics $Q_1, Q_2$ and observing that
the curve is nonsingular away from primes dividing the discriminant
of $\det (tM_1 + M_2)$, where $M_1$ and $M_2$ are the symmetric matrices
corresponding to $Q_1, Q_2$.  By Lemma \ref{goodp}, the local invariant
at every other prime is $0$.  We examine these primes in turn.

First we consider $p = 2$.  We find that there are no rational points
in the fibers over $t$ with $v_2(t) \le 1$.  In fact, none of the affine
patches of the graph of the fibration (see Lemma \ref{disc})
has points mod $2^3$ for such $t$.
On the other hand, the point $[1:0:5:3:6:1]$ modulo
$8$ can be lifted to a $2$-adic point on the fiber at $0$.
Let us show that the local
invariant is $1/2$ above all $t$ with $v_2(t) \ge 2$.  Indeed, by Lemma
\ref{constdisc} it suffices to consider $t = 0, 4$.  There is one prime $\p$
of $\Q(r)$ above $2$,
and it is totally ramified.  Therefore it suffices to show that
the conics $x^2 - (t-\alpha) y^2 - c z^2$ are not solvable at
$\p$ for $t = 0,4$.  In both cases this can be checked modulo
$\p^5$.

Next we consider $p = 3$.  This time we find that there are no rational points
in fibers above $t$ with $v_3(t) \le -1$, and that this can be checked modulo
$3^2$.  However, the point $[8:0:7:4:2:1]$ modulo $9$ can be lifted to
a $3$-adic point on the fiber at $4$.
Let us show that the local invariant
is $0$ above all $t$ with $v_3(t) \ge 0$.  Again, by Lemma \ref{constdisc}
it suffices to consider $t = 0, 1, 2$.  There are two primes of $\Q(r)$
lying above $3$, both unramified of degree $2$, and $c$
has valuation $1$ at both.  On the other hand, $\alpha$ has valuation $0$
at both primes and $\alpha$ is not congruent to any integer modulo either
prime.
Furthermore, $2-\alpha$ is a square at both, while $-\alpha$ and $1-\alpha$
are squares at neither.  In particular, it follows from Lemma \ref{odd}
that the local invariant at $t$ is $0$ above all $t \in \Z_3$.

For $p = 5$, the computations are simpler, because $c$ is a
square at all primes of $\Q(r)$ lying above $5$.  Indeed, there is one
ramified prime at which the completion is isomorphic to $\Q_5(\sqrt 5)$
and $c$ is congruent to 4 mod $\sqrt 5$, and one unramified prime
of degree $2$ at which $c$ is congruent to $2$ mod $5$, which is
a square in $\F_{25}$.  Thus the local invariant is $0$ above
all $t \in \P^1(\Q_5)$.

Finally, we need to consider the primes where $V$ has good reduction but
the fiber at infinity has bad reduction, namely $397$ and $449$.
Note that $c$ is a unit at every place above $397$.
Also, for every point $P$ whose image $t_P$ in $\P\sp 1$ is not congruent
to $\alpha$ mod any $\p$ above $397$, the difference $t_P-\alpha$ is a unit
at all $\p$ above 397.
Proposition \ref{odd} then shows that the local invariant of $(c,t_P-\alpha)$
is $0$ at all such $\p$, so that the local invariant of $\cores(c,t_P-\alpha)$
is $0$ by Proposition \ref{coresinv}.
every $\Q_{397}$-point on $V$ whose reduction mod $p$ does not lie
on an $I_4$-fiber will have local invariant $0$.
Over $\F_{397}$, the $I_4$-fibers lie above $(47:1)$, $(144:1)$, $(224:1)$,
$(379:1)$ (the first coordinates are the roots of the minimal polynomial of
$\alpha$ in $\F_{397}$).  Since $V$ has good reduction
at $397$, every $\F_{397}$-point on $V$ lifts to a $\Q_{397}$-point, so it suffices
to verify that the $\F_{397}$-point
$(246 : 16 : 98 : 0 : 1 : 0)$ maps to $(0:1)$ by $F_1$.  But by Proposition
\ref{constant}, it follows that the local invariant at $397$ is identically $0$.

To prove that the local invariant at $449$ is 0, it is enough to check that
$c$ is a square at all places of $\Q(r)$
above $449$: indeed, it is congruent to $204$ or $251$ at all of these places.
Alternatively, it can be verified that the fiber at infinity contains a
$\Q_{449}$-point lying above $(246:105:375:347:1:0)$.  This completes the proof.
\end{proof}

Theorem \ref{mainnew} now follows by
combining Propositions \ref{selmer}, \ref{nontriv}, \ref{usev},
and \ref{obstr}.

\subsection{The Richelot isogeny}
Following a suggestion of Nils Bruin and Victor Flynn, we now study the
interaction of the element of $\Sha$ constructed in Theorem \ref{mainnew}
with the Richelot
isogeny on the Jacobian.  A Richelot isogeny (cf.\ \cite{CF}, chapter 9)
is an isogeny of the Jacobian
of a curve of genus $2$ to that of another curve whose kernel is a maximal
isotropic subgroup of the $2$-torsion.  Given a curve of genus $2$ with
Weierstrass points
$W_1, \dots, W_6$, such subgroups consist of $0$ and three divisors of
the form $W_i - W_j$ such that no $W_i$ appears in more than one of them;
since $W_i - W_j = W_j - W_i$ in the Jacobian, they correspond to
partitions of the Weierstrass points into two pairs.  Letting the equation
of the curve be
$y^2 = f(x)$, where $\deg f = 6$, the Weierstrass points are $(\delta_i,0)$ where
$\delta_i$ is a root of $f$, so such partitions correspond to
factorizations of $f$ as products $f_1f_2f_3$ of three quadratic factors.
The isogenous curve is then defined by the equation
$y^2 = cg_1g_2g_3$, where $g_i = f_{i+1}f'_{i+2} - f'_{i+1}f_{i+2}$,
indices are read mod $3$,
and $c$ is the determinant of the matrix of coefficients of the $f_i$
(\cite{CF}, sect.\ 9.2).
Observe that multiplying one of the $f_i$ by a constant $k$ multiplies
$c$, $g_{i+1}$, and $g_{i+2}$ by $k$.  This is compatible with the evident
fact that if $C$ is isogenous to $C'$ then the twist of $C$ by $k$ is
isogenous to the twist of $C'$ by $k$.
In particular we find that the Jacobian of the curve
$$C_t: y^2 = t(x^2+1)(x^2-2x-1)(x^2+x-1)$$
is isogenous to that of
$$C'_{-t}: y^2 = -t(x^2+1)(x^2+2x-1)(x^2-4x-1) = -tg_1(x)g_2(x)g_3(x) = -tg(x).$$
In addition we observe that there is only one rational Richelot isogeny on each
of these curves, because the $2$-torsion points that are not rational are
defined over extensions of degree $4$.  A Galois-stable subgroup containing such a point
contains four elements other than the identity, so it cannot have order $4$.
It follows that the only rational
maximal isotropic subgroup is the one made up of the four rational points.
From now on, we will denote the Richelot isogeny from $\Jac(C_t)$ to
$\Jac(C'_{-t})$ by $\phi_t$.

First we note that $\Sha(\Jac(C'_{-6n}))$ is also nontrivial
for $n$ in the set $S$ described in Theorem \ref{mainnew}.

\begin{theorem}\label{main2}
Let $S$ be the set of primes described in Theorem \ref{mainnew}.  Then for all $n$
which are products of elements of $S$, the $2$-part of the Tate-Shafarevich
group of the Jacobian of the curve $y^2 = -6ng(x)$ is nontrivial.
\end{theorem}

\begin{proof}
The proof of this theorem is very similar to that of Theorem
\ref{mainnew},
except that one uses $\delta' = (3,1 + \sqrt 2, (1+\sqrt 5)/2)$.
Note that the field of definition of the lines of $V_{g,\delta'}$ turns out to
be the same as that of $V_{f,\delta}$,
because $-1$ is a square in the splitting field of $g$.
\end{proof}

It is natural to ask whether these systematically-occurring elements of
$\Sha$ are in the kernel of the map induced on
Tate-Shafarevich groups by $\phi_n$.  It is not clear whether
a general result can be obtained here.  We will prove only
that they are not in the kernel for the smallest case $n = 1$.
To do so, we will calculate the Selmer groups of these isogenies, as in \cite{Sc},
in the special case $n = 1$.

We now sketch the method of \cite{Sc} for calculating Selmer groups.  Let $C$ be
a curve and $\Jac C$ its Jacobian.  Let $A$ be an abelian variety and
$\phi: A \ra J$ an isogeny from $A$ to $J$.  Let $\hat \phi$ be the dual
isogeny from $\hat J$ to $\hat A$.  Then we find divisors $D_1,
\dots, D_r$ the union of whose Galois orbits spans $\ker \hat \phi$,
and functions $\psi_1, \dots, \psi_r$ such that the divisor of $\psi_i$ is $kD_i$,
where $k$ is the exponent of $\ker \phi$ (here $2$).  The functions
$\psi_i$ are defined over number fields $K(D_i)$; we use them to define an
evaluation map from an open subset of
$C(K)$ to $\oplus_i(K \otimes_\Q K(D_i))^*$ for every extension $K$ of $\Q$, and hence a
descent map $\delta_{\phi,C}$ from $J(\Q)$ to $\oplus_i K(D_i)^*/{K(D_i)^*}^2)$,
extending the evaluation map from points of $C$ to divisors supported on an
open subset of $C$ by multiplicativity.  It is
proved in \cite{Sc}, Lemma 2.1, that this yields a well-defined map on the
Jacobian.  As in the usual descent procedure, the image is always contained
in the subgroup of $\oplus_i K(D_i)^*/{K(D_i)^*}^2$ unramified away from $2$, $\infty$,
and the bad primes of $J$, which are a subset of the bad primes of $C$
(\cite{Sc}, page 454).  Provided that Schaefer's
Assumptions I and II are satisfied, the Selmer group $\Sel_\phi$ of $\phi$
is then isomorphic to the subgroup of $\oplus_i K(D_i)^*/{K(D_i)^*}^2$ of
elements unramified away from $2$, $\infty$, and bad primes of $J$ which are
in the images of the local descent maps
at $2$, $\infty$, and the bad primes of $J$.
Assumption I is satisfied because every rational
divisor class on a curve of genus $2$ is represented by a rational divisor, and
Assumption II is satisfied because the natural map from $A[\phi]$ to
$\mu_2(\Q \oplus \Q)$ given by the Weil pairing is an isomorphism
(this is because there is a Galois-invariant basis for $\ker \phi$, not merely
a Galois-invariant spanning set).

\begin{definition}
With notation as in the discussion above, the Tate-Shafarevich group
$\Sha_\phi$ is defined to be $\Sel_\phi/\delta_{\phi,C}(J(\Q))$.
\end{definition}

\begin{lemma}\label{howsel}
The Selmer group $\Sel_{\phi_n}$ of $\phi_n$ is a
subgroup of $(\Q\oplus\Q)^*/{(\Q\oplus\Q)^*}^2$, and the descent map $\delta_{\phi_n,C}$
takes a $K$-point $(x,y)$ of $C_n$ to $(f_1(x),f_2(x)) \in (K\oplus K)^*/{(K\oplus K)^*}^2$.
For $C'_{-n}$ the Selmer group is again a subgroup of
$(\Q\oplus\Q)^*/{(\Q\oplus\Q)^*}^2$, and
the descent map corresponding to the isogeny dual to $\phi_n$
takes a $K$-point $(x,y)$ of $C'_{-n}$ to $(g_1(x),g_2(x)) \in (K\oplus K)^*/{(K\oplus K)^*}^2$.
\end{lemma}

\begin{proof}
We will only prove this for $C_n$, the proof for $C'$ being identical.
Let $D$ be a rational divisor class of order $2$.
Then $D$ is represented by a divisor $(\alpha,0)+(\alpha',0)-D_\infty$, where
$\alpha$ and $\alpha'$ are roots of one of the quadratic factors of
$f$ and $D_\infty$ is the divisor of poles of the
function $x$ on the curve.  It follows that $2D$ is the divisor
of the quadratic factor of $f$ whose roots are $\alpha, \alpha'$.
The kernel of $\phi_n$ is isomorphic to $(\Z/2)^2$
with trivial Galois action, so any two factors of $f$ give functions whose
divisors are doubles of the divisors that span the kernel.  Since the factors
are defined over $\Q$ and we use two of them, the description of Schaefer's
procedure above states that the descent map and its target are as claimed.
\end{proof}

\begin{proposition}\label{howbig}
There is an exact sequence
$$0 \ra J(\Q_p)[\phi]/\phi'(J'(\Q_p)[2]) \ra J(\Q_p)/\phi'(J'(\Q_p))
\ra J'(\Q_p)/2J'(\Q_p) \ra J'(\Q_p)/\phi(J(\Q_p)) \ra 0.$$
\end{proposition}

\begin{proof}
This is a special case of \cite{Sc}, Proposition 2.6.
\end{proof}

\begin{theorem}\label{seliso}
The Selmer groups of the Richelot isogenies on the Jacobians of the curves
$y^2 = -6f(x)$ and $y^2 = 6g(x)$ are isomorphic to $(\Z/2)^3$.
\end{theorem}

\begin{proof}
Let the Jacobians of $C_{-6}$ and $C'_{6}$
be denoted $J$ and $J'$ respectively.
The Selmer groups of $J$ and $J'$
will be calculated together, because there is no obvious way
to determine the size of $J(\Q_p)/\phi'(J'(\Q_p))$ without computing
the size of $J'(\Q_p)/\phi(J(\Q_p))$ at the same time.  We will compute them
using Proposition \ref{howbig}.

The primes of bad reduction of both $C_{-6}$ and $C'_{6}$
are $2, 3, 5$.
Using \cite{Sc}, Propositions 2.4 and 2.5, we can calculate the order of
$J'(\Q_p)/2J'(\Q_p)$, the third nonzero term in the exact sequence of
Proposition \ref{howbig}:
it is $16$ for $p = 2$, while it is $4$ for $p = 3, 5$ and $2$ for
$p = \infty$.  It is easy to calculate that the first term of the exact
sequence in Proposition \ref{howbig} has order $4$
for $p = 2, 3, 5$ and $2$ for $p = \infty$.  For both curves
the $2$-torsion points map to $(10,2)$, $(2,-2)$, and their product
$(5,-1)$ in $\Sel_\phi$ and
$\Sel_\phi'$ under $\delta_{\phi,C}$ and $\delta_{\phi',C'}$ respectively.
The sum of the
dimensions of the second and fourth nonzero terms in the exact sequence
must be equal to the sum of the dimensions of the first and third nonzero
terms.  For $p \in \{3,5,\infty\}$, the sum of the dimensions of the
images of the subgroups generated by the $2$-torsion is already equal to
the sum of the dimensions of the first and third nonzero terms, so for
these $p$ the images of the $2$-torsion points generate the local image.
However, for $p = 2$, we need to find two additional generators.

We check that there is a $\Q_2(\sqrt 6)$-rational point on
$C_{-6}$ with $x$-coordinate $3 + \sqrt{6}$ whose image under
$\delta_R$ is $(10,-2)$.  Also, there is a $\Q_2(\sqrt
6)$-rational point on $C'_{6}$ with $x$-coordinate $5 + \sqrt{6}$
whose image is $(10,-2)$.  These new generators are independent of
the image of the $2$-torsion points of $C_{-6}$ and $C'_{6}$
respectively, so we have found the full Selmer group of the
isogeny for both $C$ and $C'$.
$$
\begin{array}
{|c|c|c|}
\hline
p  &  C  &  C' \\
\hline
2 & (5,1),(2,2),(1,-1) & (5,1),(2,2),(1,-1) \\
3 & (1,-1),(-1,1)      &   (1,-1),(-1,1) \\
5 & (5,1),(2,2)        & (5,1),(2,2) \\
\infty & (1,-1)        &  (1,-1) \\
\hline
\end{array}
$$

It is now straightforward to see that the Selmer groups in both cases are
generated by $(5,1),(2,2),(1,-1)$.
\end{proof}

It follows immediately that the Selmer groups contain $(1,-1)$, which
is not in the image of the $2$-torsion.
We can now show that $(1,-1)$ is in the Tate-Shafarevich group of
the Richelot isogeny, from which it will follow
that the elements of $\Sha$ found in Theorem \ref{mainnew}
and Theorem \ref{main2} are not in the kernel of
the map induced by the Richelot isogeny for $n = 1$.

\begin{theorem}\label{invis}
The element of
$\Sha(\Jac(C_{-6}))$ described in Theorem \ref{mainnew} is not in the kernel of the
map on Tate-Shafarevich groups induced by
the Richelot isogeny, and similarly for $C'_{6}$.
\end{theorem}

\begin{proof}
We will do this only for $C$, the calculations for $C'$ being essentially
identical.  We will first
represent that element as an explicit cocycle with values
in $\Jac(C_{-6})[2]$, apply $\phi$, and show that the image is
the nontrivial element $(1,-1)$ in $\Sel_{\phi'}$.

Recall (Definition \ref{greek}) that we write $\delta_i$ for the components
of the image of $\delta$ that gives the nontrivial element of $\Sha(\Jac(C_{-6n}))$
under a fixed isomorphism $A_f \otimes \bar \Q \ra \oplus_1^6 \bar \Q$ in which
$\bar \Q[X]/(f_j)$ corresponds to components $2j-1$ and $2j$.
Given $\sigma \in \Gal (\bar \Q /\Q)$, write $s$ for the permutation of
$\{1,2,\dots,6\}$ induced by $\sigma$ on the $\delta_j$.
By remarks in the proof of \cite{Sc}, Proposition 2.2 and
in \cite{Sc}, section
2.5, we can write the element of $\Sha$ as a cocycle with values in $\mu_2(A_f \otimes \bar\Q)/\pm 1$.
As in the discussion preceding Lemma \ref{zott},
we identify $A_f \otimes\bar\Q$ with $\oplus_1^6 \bar \Q$ with Galois acting by
${}^\sigma (a_1, \dots, a_6) = ({}^\sigma a_{s^{-1}(1)}, \dots,
{}^\sigma a_{s^{-1}(6)})$.  Indeed, the cocycle corresponding to the
element of $\Sha$ constructed in Theorem \ref{main2} is the one
that takes $\sigma \in \Gal(\bar \Q/\Q)$ to ${}^\sigma \alpha/\alpha$, where
$$\alpha = \left(\sqrt{(\delta_1\delta_j)}\right)_{j=1}^6 =
\sqrt{(1,1,-3(1+\sqrt{2}),-3(1-\sqrt{2}),3(1+\sqrt{5})/2,3(1-\sqrt{5})/2)}.$$
In particular, this cocycle, which we will denote $z_\alpha$,
factors through $\Gal(F/\Q)$, where $F$ is as in Theorem \ref{mainnew}.

Now let us write $z_\alpha$ as a cocycle with values in $J[2]$.
Following the description in
\cite{Sc}, section 2.5, we see that $(r_i)\in \mu_2(L')/\pm 1$
corresponds to a $2$-torsion point $T$ on the Jacobian
such that $e(T,(\delta_j,0)-(\delta_1,0)) = r_j$
for all $j$ from $1$ to $6$, where $e$ is the Weil pairing.
The rational $2$-torsion divisors are $0$ and $(\delta_{2k},0)-(\delta_{2k-1},0)$
for $k = 1, 2, 3$.
Since two $2$-torsion divisors of the form $(\delta_i,0)-(\delta_j,0)$
have Weil pairing $1$ or
$-1$ depending on whether the number of points appearing in both counted with
multiplicity is even or odd, the rational $2$-torsion points arise only
from the following sequences of square roots of $1$:
$$(1,1,1,1,1,1),\quad (1,1,-1,-1,1,1), \quad(1,1,1,1,-1,-1),
\quad (1,1,-1,-1,-1,-1).$$
We claim that the image of $z_\alpha(\sigma)$ under the pairing with the points
$(\delta_j,0)-(\delta_1,0)$
is one of these sequences corresponding to an element of the kernel of the
Richelot isogeny if and only if $\sigma$ fixes a square root $i$ of $-1$.
Indeed, it is clear that the first two components of $z_\alpha(\sigma)$ are
always $1$, while the product of components $3$ and $4$ is $1$ if and only if
$\sigma$ fixes a square root of $\delta_3 \delta_4 = -9$, and similarly
for components $5$ and $6$.

These points are the kernel of the Richelot isogeny, which therefore maps
 $z_\alpha$ to a cocycle that factors
through $\Q(i)$ and takes the nontrivial element of this Galois group
to the rational $2$-torsion point arising from the factor $g_1$.
This corresponds to the element $(1,-1)$ of the Selmer group of $\phi'$.
To see this, note that $(1,i)$ is a square root of $(1,-1)$, and
the action of $\sigma$ multiplies $(1,i)$ by $(1,1)$ if $\sigma$ fixes $i$ and
by $(1,-1)$ otherwise.  But the image of the $2$-torsion point coming from
$g_1$ under the $\phi'$-Weil pairing is $(1,-1)$, because the pullback of this
point pairs trivially with the point coming from the factor $f_1$ on
$J$ and nontrivially with the point coming from $f_2$.  As a result,
the image of $z_\alpha$ under the Richelot isogeny
corresponds to the cocycle $\sigma \ra {}^\sigma(1,i)/(1,i)$.

Recall from Definition \ref{descent} and Proposition \ref{descjac} that
the full $2$-descent map $\Delta_{C'}$ is a map from $\Jac(C')(\Q)$ to
$\Q[X]/(g)$.  Composing $\Delta_{C'}$ with an isomorphism
from $\Q[X]/(g)$ to $\oplus_{i=1}^3 \Q[X]/(g_i)$
obtained from the Chinese Remainder Theorem, we see that the map takes
a point $(x_0,y_0)$ to $(x_0-\rho_i)_{i=1}^3$; in other words, its components
are defined by the functions $x-\rho_i$.
We remark that
$g_1 = N (x-\rho_1)$ and $g_2 = N(x-\rho_2)$.  It follows that
if $P$ is a point of $\Jac(C'_{6})$, then the image of $P$ under $\delta_{R,C'}$,
the descent map
for the Richelot isogeny, is the norm of the first two components of its
image under $\Delta_{C'}$.

However, using {\sc magma}'s {\tt TwoSelmerGroup} command to compute the fake Selmer
group, it is easy to verify that the norm map from the fake Selmer group for
full $2$-descent to the
Selmer group of the Richelot isogeny is an isomorphism for $C'_{6}$.
It follows that
the image of an element of $\Sha(C'_{6})$ cannot be the image of a
rational divisor, so $(1,-1)$ belongs to the Tate-Shafarevich group of the
Richelot isogeny on $C'_{6}$.
\end{proof}

\begin{remark}
This proof applies to any twist $C'_{6n}$, where $n$ is a product of elements of
the set $S$ described in Theorem \ref{main2},
for which we know that no rational point
on the Jacobian of $C'_{6n}$ maps by $\delta_R$
to $(1,-1)$.  However, there is no reason to expect this to be true in general;
if the rank is large, it is very unlikely that the fake Selmer group for
multiplication by $2$ would be isomorphic to the Selmer group of the Richelot
isogeny, and the proof would fail.
\end{remark}